\documentclass{amsart}
\usepackage{amssymb}
\usepackage{amsmath}
\usepackage{epic}
\usepackage[all]{xypic}
\usepackage{esint}

\newtheorem{thm}{Theorem}[section]

\newtheorem{cor}[thm]{Corollary}
\newtheorem{lemma}[thm]{Lemma}
\newtheorem{prop}[thm]{Proposition}
\theoremstyle{definition}
\newtheorem{defn}[thm]{Definition}
\newtheorem{ex}[thm]{Example}

\theoremstyle{remark}
\newtheorem{rem}[thm]{Remark}

\newtheorem*{ack}{Acknowledgment}
\newtheorem*{org}{Organization of the article}

\newcommand{\ZZ}{{\mathbb  Z}}
\newcommand{\RR}{{\mathbb  R}}
\newcommand{\CC}{{\mathbb  C}}
\newcommand{\HH}{{\mathbb  H}}
\newcommand{\OO}{{\mathbb O}}

\newcommand{\E}{{\mathcal E}}
\newcommand{\F}{{\mathcal F}}
\newcommand{\J}{{\mathcal J}}
\newcommand{\X}{{\mathcal X}}
\renewcommand{\H}{{\mathcal H}}

\newcommand{\g}{\mathfrak{g}}
\newcommand{\p}{\mathfrak{p}}
\newcommand{\h}{\mathfrak{h}}

\renewcommand{\r}{\mathfrak{r}}
\renewcommand{\u}{\mathfrak{u}}
\renewcommand{\v}{\mathfrak{v}}
\renewcommand{\k}{\mathfrak{k}}
\newcommand{\m}{\mathfrak{m}}
\newcommand{\gc}{\mathfrak{g}_{\CC}}
\newcommand{\pc}{\mathfrak{p}_{\CC}}

\newcommand{\f}{\mathfrak{f}}

\newcommand{\so}{\mathfrak{so}}
\renewcommand{\sl}{\mathfrak{sl}}
\newcommand{\su}{\mathfrak{su}}
\renewcommand{\sp}{\mathfrak{sp}}
\newcommand{\gl}{\mathfrak{gl}}
\renewcommand{\o}{\mathfrak{o}}

\newcommand{\GL}{\mathop{\rm GL}}

\newcommand{\SL}{\mathop{\rm SL}}
\newcommand{\SO}{\mathop{\rm SO}}
\newcommand{\SU}{\mathop{\rm SU}}
\newcommand{\Sp}{\mathop{\rm Sp}}
\renewcommand{\S}{\mathop{\rm S}}
\newcommand{\U}{\mathop{\rm U}}
\newcommand{\PSL}{\mathop{\rm PSL}}
\newcommand{\PSO}{\mathop{\rm PSO}}
\newcommand{\PSU}{\mathop{\rm PSU}}
\newcommand{\PSp}{\mathop{\rm PSp}}
\newcommand{\Aff}{\mathop{\rm Aff}}
\newcommand{\Diff}{\mathop{\rm Diff}}
\newcommand{\Spin}{\mathop{\rm Spin}}

\newcommand{\Gc}{G_{\CC}}
\newcommand{\Pc}{P_{\CC}}
\newcommand{\G}{\Gamma}

\newcommand{\GV}{\mathop{\rm GV}}
\newcommand{\vol}{\mathop{\rm vol}}
\newcommand{\Ad}{\mathop{\rm Ad}}
\newcommand{\ad}{\mathop{\rm ad}}
\newcommand{\Aut}{\mathop{\rm Aut}}
\newcommand{\id}{\mathop{\rm id}}
\newcommand{\Lie}{\mathop{\rm Lie}}
\newcommand{\Hom}{\mathop{\rm Hom}}
\newcommand{\dev}{\mathop{\rm dev}}
\newcommand{\hol}{\mathop{\rm hol}}
\newcommand{\Graph}{\mathop{\rm Graph}}

\newcommand{\Fol}{\mathop{\rm Fol}}
\newcommand{\tr}{\mathop{\rm tr}}
\newcommand{\ch}{\mathop{\rm ch}}
\newcommand{\Z}{\mathop{\rm Z}}
\newcommand{\sign}{\mathop{\rm sign}}
\newcommand{\bt}{\bullet}
\newcommand{\wt}{\widetilde}

\renewcommand{\theenumi}{\roman{enumi}}

\newdimen\theight
\def\TeXref#1{%
             \leavevmode\vadjust{\setbox0=\hbox{{\tt
                     \quad\quad  {\small \textrm #1}}}%
             \theight=\ht0
             \advance\theight by \lineskip
             \kern -\theight \vbox to
             \theight{\rightline{\rlap{\box0}}%
             \vss}%
             }}%

\title[Characteristic classes of transversely homogeneous foliations]{Secondary characteristic classes of transversely homogeneous foliations}

\author{Jes\'{u}s A. \'{A}lvarez L\'{o}pez}
\thanks{The authors are supported by the Spanish MICINN grants MTM2008-02640 and MTM2011-25656.}
\address{Jes\'{u}s A. \'{A}lvarez L\'{o}pez, Departamento de Xeometr\'{i}a e Topolox\'{i}a, Universidad de Santiago de Compostela, Spain}
\email{jesus.alvarez@usc.es}
\author{Hiraku Nozawa}
\thanks{The second author is partially supported by Research Fellowship of the Canon Foundation in Europe and the EPDI/JSPS/IH\'{E}S Fellowship. This paper was written during the stay of the second author at Institut des Hautes \'{E}tudes Scientifiques, Institut Mittag-Leffler and Centre de Recerca Matem\`{a}tica; he is very grateful for their hospitality.}
\address{Hiraku Nozawa, Department of Mathematical Sciences, Colleges of Science and Engineering, Ritsumeikan University, Nojihigashi 1-1-1, Kusatsu, Shiga, 525-8577, Japan}
\email{hnozawa@fc.ritsumei.ac.jp}

\date{}


\begin{document}

\maketitle

\begin{abstract}
Let $G$ be a simple Lie group of real rank one, and $S^{q}_{\infty}$ the ideal boundary of the corresponding hyperbolic symmetric space of noncompact type ($\mathbf{H}_{\RR}^{n}$, $\mathbf{H}_{\CC}^{n}$, $\mathbf{H}_{\HH}^{n}$ or $\mathbf{H}_{\OO}^{2}$). We show the finiteness of the possible values of the secondary characteristic classes of transversely homogeneous foliations on a fixed manifold whose transverse structures are modeled on $(G,S^{q}_{\infty})$, except the case of transversely conformally flat foliations of even codimension $q$. For this exceptional case, we construct examples of foliations which break the finiteness and we show a weaker form of the finiteness. These are generalizations of a finiteness theorem of secondary classes of transversely projective foliations by Brooks-Goldman and Heitsch to other transverse structures. We also show Bott-Thurston-Heitsch type formulas to compute the Godbillon-Vey classes of certain foliated bundles, and then we obtain a rigidity result on transversely homogeneous foliations on the unit tangent sphere bundles of hyperbolic manifolds.
\end{abstract}

\tableofcontents

\section{Introduction}

\subsection{Secondary characteristic classes of foliations and a theorem of Brooks-Goldman-Heitsch}
For a codimension $q$ smooth foliation $\F$ of a smooth manifold $M$, we have the characteristic homomorphism $\Delta_{\F} : H^{\bt}(WO_{q}) \to H^{\bt}(M;\RR)$ (see Section~\ref{sec:secintro}). The cohomology classes in the image of $\Delta_{\F}$ are called the secondary characteristic classes of $\F$. These are cobordism invariants of foliations, which come from the continuous cohomology of the Haefliger's classifying space $B\G^{q}$~\cite{Haefliger1979}. The relation between the dynamics or geometry of foliations and secondary characteristic classes has been one of the main themes in the study of foliations (see the review article~\cite{Hurder2002} or~\cite[Chapter~7]{CandelConlon2003}). Main examples of foliations with nontrivial secondary characteristic classes are quotient of homogeneous foliations on homogeneous spaces by lattices, which have been extensively studied~\cite{KamberTondeur1975a,Yamato1975,Baker1978,Heitsch1978,Pittie1979,Pelletier1983,Asuke2010}. Transversely homogeneous foliations are generalizations of these foliations, whose secondary characteristic classes can be computed in a similar way. These foliations were used in the construction of families of foliations whose characteristic classes nontrivially and continuously vary by Thurston~\cite{Thurston1972b,Bott1978} and Rasmussen~\cite{Rasmussen1980}. Other families with this property, constructed by Heitsch~\cite{Heitsch1978}, are quotient of homogeneous foliations on homogeneous spaces by lattices. Their constructions imply that there are uncountably many foliations which are not mutually cobordant, and certain homology groups with integer coefficients of the classifying space $B\G^{q}$ are uncountable~\cite[Section~6]{Heitsch1978}. 

In spite of the role played by transversely homogeneous foliations in the construction of these examples, Brooks-Goldman and Heitsch showed that transversely projective foliations, a class of transversely homogeneous foliations, satisfy the following remarkable finiteness property of the secondary characteristic classes. Let $G$ be a Lie group and $P$ a closed subgroup of $G$. A $(G,G/P)$-foliation is a foliation whose transverse structure is modeled by the $G$-action on $G/P$ (see Definition~\ref{defn:trhomfol}). When $G = \SL(q+1;\RR)$ and $G/P = S^{q}$, a $(G,G/P)$-foliation is called a {\em transversely projective foliation\/}. Fix a smooth manifold $M$ with finitely presented fundamental group. Let $\Fol(G,G/P)$ be the set of $(G,G/P)$-foliations on $M$, and let 
\[
\Sigma(G,G/P) =\# \{\, \Delta_{\F} \in \Hom(H^{\bullet}(WO_{q}), H^{\bullet}(M;\RR)) \mid \F \in \Fol(G,G/P) \,\}\;.
\]

\begin{thm}[Brooks-Goldman~\cite{BrooksGoldman1984} in the case of $q=1$ and Heitsch~\cite{Heitsch1986} for $q > 1$]\label{thm:BGH}
$\Sigma(\SL(q+1;\RR),S^{q}) < \infty$. 
\end{thm}

In this article, we will generalize Theorem~\ref{thm:BGH} for other cases of $(G,G/P)$. We also prove Bott-Thurston-Heitsch type formulas to compute secondary characteristic classes and apply such formulas to obtain certain rigidity of foliations.

\subsection{A sufficient condition for the finiteness of secondary characteristic classes}
We assume that $G$ is connected, linear algebraic and semisimple. Let $\Gc$ be a connected complex semisimple Lie group such that $\Lie(\Gc) = \Lie(G) \otimes \CC$ as a Lie algebra over $\RR$. Our first result is the following. 

\begin{thm}\label{thm:finite}
If the map $H^{\bt}(\Gc/P;\RR) \to H^{\bt}(G/P;\RR)$ is trivial on positive degrees, then $\Sigma(G,G/P) < \infty$.
\end{thm}

When $(G,G/P)=(\SL(q+1;\RR),S^{q})$ for odd $q$, the assumption of Theorem~\ref{thm:finite} on $(G,P)$ is satisfied (see Section~\ref{sec:oddtrproj}). So Theorem~\ref{thm:finite} implies Theorem~\ref{thm:BGH} for odd $q$. The following cases are our examples of $(G,G/P)$: $(\SO_{0}\left(n+1,1\right),S^{n}_{\infty})$, $(\SU\left(n+1,1\right),S^{2n+1}_{\infty})$, $(\Sp\left(n+1,1\right),S^{4n+3}_{\infty})$ and $(F_{4(-20)},S^{15}_{\infty})$, where $S^{n}_{\infty}$, $S^{2n+1}_{\infty}$, $S^{4n+3}_{\infty}$ and $S^{15}_{\infty}$ are the ideal boundaries of the corresponding noncompact symmetric spaces $\mathbf{H}_{\RR}^{n}$, $\mathbf{H}_{\CC}^{n}$, $\mathbf{H}_{\HH}^{n}$ and $\mathbf{H}_{\OO}^{2}$, respectively. According to the case of manifolds, $\left(\SO_{0}\left(n+1,1\right),S^{n}_{\infty}\right)$-foliations are called {\em transversely conformally flat} foliations, and $\left(\SU\left(n+1,1\right),S^{2n+1}_{\infty}\right)$-foliations are called {\em transversely spherical} $CR$ foliations. The unit tangent sphere bundles of hyperbolic manifolds have typical examples of these $(G,G/P)$-foliations (see Example~\ref{ex:mob}). The map $H^{\bt}(\Gc/P;\RR) \to H^{\bt}(G/P;\RR)$ is trivial on positive degrees except in the case of transversely conformally flat foliations of even codimension (see Section~\ref{sec:example}). Thus we get the following. 

\begin{cor}\label{cor:main}
If $(G,G/P)$ is $(\SO_{0}(n+1,1),S^{n}_{\infty})$ for odd $n$, $(\SU(n+1,1),S^{2n+1}_{\infty})$, $(\Sp(n+1,1),S^{4n+3}_{\infty})$ or $\left(F_{4(-20)},S^{15}_{\infty}\right)$, then $\Sigma(G,G/P) < \infty$.
\end{cor}

\begin{rem}
	\begin{enumerate}
	
		\item Since $\SU(1,1) \cong \SL(2;\RR)$ and $\SO_{0}(2,1) \cong \PSL(2;\RR)$, where $\SO_{0}(2,1)$ is the identity component of $\SO(2,1)$, Corollary~\ref{cor:main} for $(G,G/P)=(\SU(1,1),S^{1}_{\infty})$ or $(\SO_{0}(2,1),S^{1}_{\infty})$ is essentially contained in Theorem~\ref{thm:BGH}. Hantout~\cite{Hantout1988} also investigated this type of finiteness results, but his result does not imply this corollary.

		\item Note that the actions of $\SU(n+1,1)$ and $\Sp(n+1,1)$ on spheres may not be effective, depending on $n$, because their stabilizers are equal to the centers. But, by a slight modification of the proof of Theorem~\ref{thm:finite}, we can show the finiteness for the case where $(G,G/P)$ is $(\PSU(n+1,1),S^{2n+1}_{\infty})$ or $(\PSp(n+1,1),S^{4n+3}_{\infty})$ (see Section~\ref{sec:center}).

		\item $\Sigma(G,G/P)$ is bounded in terms of the degree of defining polynomials of a real algebraic variety $\Hom(\pi_{1}M,G)$ by~\cite{Milnor1964}  (see Remark~\ref{rem:fincomp}). Using a product of the same manifolds, it is not difficult to see that, for any $k>0$, there exists a closed manifold $M$ such that $\Sigma(G,G/P) > k$.

		\item It is not difficult to see that every nontrivial secondary characteristic class of $(G,G/P)$-foliations is a multiple of the Godbillon-Vey class for these cases (see Proposition~\ref{prop:onlyGV3}). 
		
	\end{enumerate}
\end{rem}

\subsection{Bott-Thurston-Heitsch type formulas}

The Godbillon-Vey class $\GV(\F)$ of a foliation $\F$ is the secondary characteristic class first discovered in~\cite{GodbillonVey1971}, and it is specially important for transversely homogeneous foliations as suggested by results of Pittie~\cite{Pittie1979}. In the standard notation, $\GV(\F) = (2\pi)^{q+1}\Delta_{\F}(h_{1}c_{1}^{q})$ for a codimension $q$ foliation~\cite[Theorem~7.20]{KamberTondeur1975b}. A typical example of transversely projective foliations is suspension foliations of projective actions; namely, for a manifold $N$ and a homomorphism $\pi_{1}N \to \SL(q+1;\RR)$, we get an $S^{q}$-bundle $p : \widetilde{N} \times_{\pi_{1}N} S^{q} \to N$ foliated by a transversely projective foliation transverse to the fibers of $p$ (Example~\ref{ex:suspfol}). The Bott-Thurston-Heitsch formula for the Godbillon-Vey class of transversely projective foliations computes the Godbillon-Vey class of such foliations.

\begin{thm}[{\cite{Thurston1972b} and~\cite[Appendix by Brooks]{Bott1978} for $q=1$, and Heitsch \cite[Theorem~4.2]{Heitsch1978} and~\cite[Theorem~2.3]{Heitsch1983} for $q>1$}]\label{thm:BT1}
Let $N$ be a manifold and $\hol : \pi_{1}N \to \SL(q+1;\RR)$ a homomorphism. Let $p_{M} : M \to N$ be the corresponding $S^{q}$-bundle with the suspension foliation $\F$. Then, for any orientation on the fibers of $p_{M}$, we have
\begin{equation}\label{eq:BTH1}
\frac{1}{(2\pi)^{q+1}}\fint_{p_{M}} \GV(\F) = (q+1)^{q+1}e(p_M)
\end{equation}
in $H^{q+1}(N;\RR)$, where $e(p_{M})$ is the Euler class of the $S^{q}$-bundle $p_{M}$.
\end{thm}

\begin{rem}\label{rem:aff}
The case $q=1$ is special because there are different choices of $\SL(2;\RR)$-actions on $S^{1}$. To get~\eqref{eq:BTH1}, the $\SL(2;\RR)$-action on the homogeneous space $\SL(2;\RR)/\Aff(1;\RR)\approx S^{1}$ should be used in the construction of the suspension foliation $\F$, where
\[
  \Aff(1;\RR)=\left\{\,\left.\begin{pmatrix} a & b \\ 0 & 1/a \end{pmatrix}\, \right| \; a \in \RR^\times,\ b \in \RR \,\right\}\;.
\]
\end{rem}

This formula is important as one of the few methods to calculate the Godbillon-Vey class explicitly. Heitsch obtained a similar formula for other secondary characteristic classes of transversely projective foliations~\cite[Theorem~4.2]{Heitsch1978}, \cite[Theorem~2.3]{Heitsch1983}. 

We generalize this formula. Note that, for a manifold $N$ and a homomorphism $\pi_{1}N \to G$, we have a suspension foliation of the total space of a $G/P$-bundle over $N$, which naturally admits a structure of a $(G,G/P)$-foliation (Example~\ref{ex:suspfol}). Let $\SO_{0}(n+1,1)$ be the identity component of $\SO(n+1,1)$.

\begin{thm}\label{thm:BT}
Let $(G,G/P)$ be $(\SO_{0}(n+1,1),S^{n}_{\infty})$ for odd $n>1$, $(\SU(n+1,1),S^{2n+1}_{\infty})$ for $n>0$, $(\Sp(n+1,1),S^{4n+3}_{\infty})$ for $n \geq 0$ or $\left(F_{4(-20)},S^{15}_{\infty}\right)$. Let $q=\dim G/P$, $N$ a manifold and $\hol : \pi_{1}N \to G$ a homomorphism. Let $p_{M} : M \to N$ be the corresponding $G/P$-bundle with the suspension foliation $\F$. Then, for any orientation on the fibers of $p_{M}$, we have 
\begin{equation}\label{eq:BT}
\frac{1}{(2\pi)^{q+1}}\fint_{p_{M}} \GV(\F) = r_{G}\, e(p_M)
\end{equation}
in $H^{q+1}(N;\RR)$, where $e(p_{M})$ is the Euler class of the $S^{q}$-bundle $p_{M}$, and $r_{G}$ is a nonzero constant, depending on $(G,G/P)$, given in Table~\ref{table: r_G} in the case of classical groups.
\end{thm}

\renewcommand{\arraystretch}{2.5}
\begin{table}[h]
\[
\begin{array}{|c|c|} 
\hline
(G,G/P) & r_{G} \\ 
\hline 
(\SO_{0}(n+1,1),S^{n}_{\infty}) & \displaystyle n^{n+1}\\
(\SU(n+1,1),S^{2n+1}_{\infty}) & \displaystyle \frac{(n+1)^{2n+2}}{2^{n-1}(n+2)} \cdot \frac{(2n+1)!}{n! (n+1)!} \\
(\Sp(n+1,1),S^{4n+3}_{\infty}) & \displaystyle \frac{(2n+3)^{4n+3}}{2^{n-2}(n+2)} \cdot \frac{(4n+3)!}{(2n+1)!(2n+2)!} \\ 
\hline
\end{array}
\]
\caption{The constant $r_G$.}
\label{table: r_G}
\end{table}
\renewcommand{\arraystretch}{1}

Note that it is not difficult to see that both sides of~\eqref{eq:BT} are equal up to a nonzero constant factor like in the case of the original Bott-Thurston formula for the codimension one case (see~\cite[Section~3]{BrooksGoldman1984}). This relation was already pointed out in the case of $(\SO_{0}(n+1,1),S^{n}_{\infty})$ by Reznikov~\cite[Section~5.16]{Reznikov1996}.  

\begin{rem}
	\begin{enumerate}
	
		\item Rasmussen~\cite[Theorem~5.1]{Rasmussen1980} also obtained a similar formula for the case of $(\SO_{0}(3,1),S^{2}_{\infty})$. The codimension one case (excluded from Theorem~\ref{thm:BT}), where $(G,G/P)$ is either of $(\SO_{0}(2,1),S_{\infty}^{1})$ or $(\SU(1,1),S_{\infty}^{1})$, corresponds to the original Bott-Thurston formula (Theorem~\ref{thm:BT1} for $q=1$).
		
		\item Note that, in the case of $(\SO_{0}(n+1,1),S^{n}_{\infty})$ for even $n$, the Euler classes of $S^{n}$-bundles are trivial with real coefficients. So this type of formulas is not true in that case. But we will show a similar formula with the volume of the holonomy homomorphism (see Proposition~\ref{prop:r1}).

		\item Theorem~\ref{thm:BT1} for $q=1$ was used by Brooks-Goldman~\cite{BrooksGoldman1984} to prove Theorem~\ref{thm:BGH} for $q=1$. Heitsch~\cite{Heitsch1986} used Theorem~\ref{thm:BT1} and its generalization to other secondary characteristic classes to prove Theorem~\ref{thm:BGH}. Based on a calculation similar to the proof of Theorem~\ref{thm:BT}, we can give an alternative proof of Theorem~\ref{thm:BGH} for even $q$ (see Remarks~\ref{rem:simpler} and~\ref{rem:alt}). This alternative proof is slightly simpler than the original proof due to Heitsch~\cite{Heitsch1986}. 
		\item We can compute the constant $r_{F_{4(-20)}}$ based on Lie algebra cohomology like in the case of classical groups. But it requieres a number of special techniques on $F_{4}$. So we do not include the computation in this article.
	\end{enumerate}
\end{rem}

\subsection{The case of $G/P = S^{q}$ for even $q$}
In this case, it is easy to see that the assumption of Theorem~\ref{thm:finite} on the triviality of $H^{\bt}(\Gc/P;\RR) \to H^{\bt}(G/P;\RR)$ for positive degrees is never satisfied (see Proposition~\ref{prop:cri}). In fact, by using a Bott-Thurston-Heitsch type formula in Proposition~\ref{prop:r1} for the Godbillon-Vey class of transversely conformally flat foliation of even codimension, we get the following infiniteness result.

\begin{thm}\label{thm:infinite}
For each even $q$, there exists a connected noncompact smooth manifold $X$ with finitely presented fundamental group and a family $\{\F_{m}\}_{m \in \ZZ}$ of codimension $q$ transversely conformally flat foliations of $X$ such that $\GV(\F_{m}) \neq \GV(\F_{m'})$ if $m \neq m'$. 
\end{thm}

As far as we know, this is the first example of a family of transversely conformal foliations on a connected manifold whose Godbillon-Vey classes take infinitely many different values. We do not know compact examples. Asuke~\cite{Asuke2010} constructed finite families of transversely holomorphic foliations on compact homogeneous spaces whose Godbillon-Vey classes take different values. (Note that complex codimension one transversely holomorphic foliations are real codimension two transversely conformal foliations.)

\begin{rem}
Asuke~\cite{Asuke2010} proved that the Godbillon-Vey class does not change nontrivially for smooth families of transversely holomorphic foliations. As pointed out by Morita~\cite{Morita1979}, it is not known if there exists a smooth family of transversely conformal foliations of codimension greater than two whose Godbillon-Vey classes vary continuously and nontrivially. 
\end{rem}

We will show the finiteness of secondary characteristic classes in a weaker form in this case. Let $\chi(\nu \F)$ be the Euler class of the normal bundle $\nu \F$ of $\F$.  Let
  \begin{align*}
    \Sigma(G,G/P,z)&=\#\{\,\Delta_{\F}\mid\F\in\Fol(G, G/P),\ \chi(\nu \F)=z\,\}
  \end{align*}
for any fixed $z\in H^{q}(M;\RR)$. We get the following.

\begin{thm}\label{thm:even}
If $G/P = S^{q}$ for even $q$, then $\Sigma(G,G/P,z) < \infty$ for each $z\in H^{q}(M;\RR)$.
\end{thm}

\subsection{Transversely conformal foliations}

By a theorem of Zhukova~\cite[Corollary 2]{Zhukova2012}, a complete transversely conformal foliation of codimension $q > 2$ is Riemannian or $(\PSO(q+1,1),S^{q}_{\infty})$ on each connected component of $M$. For transversely real analytic conformal foliations on compact manifolds, it was proved by Tarquini~\cite[Th\'{e}or\`{e}me~0.0.1]{Tarquini2004}. Let $\Fol_{\mathcal{C}}^{q}$ be the set of codimension $q$ complete transversely conformal foliations on $M$. Let 
\[
	\Sigma_{\mathcal{C}}^{q} = \# \{\, \Delta_{\F} \mid \F \in \textstyle{\Fol_{\mathcal{C}}^{q}} \,\}\;, \quad
	\Sigma_{\mathcal{C}}^{q}(z) = \# \{\, \Delta_{\F} \mid \F \in \textstyle{\Fol_{\mathcal{C}}^{q}},\ \chi(\nu \F)=z \,\}
\]
for $z$ in $H^{q}(M;\RR)$. Since the secondary characteristic classes of Riemannian foliations are trivial (see~\cite[Section~4.48 and Theorem~4.52]{KamberTondeur1975b}), we get the following corollary of Theorem~\ref{thm:finite}.

\begin{cor}
\begin{enumerate}

\item $\Sigma_{\mathcal{C}}^{q} < \infty$ for odd $q > 1$.

\item $\Sigma_{\mathcal{C}}^{q}(z) < \infty$ for each $z\in H^{q}(M;\RR)$ and even $q > 2$. 

\end{enumerate}
\end{cor}

\subsection{Rigidity of transversely homogeneous foliations with nontrivial secondary invariants}

Let $(G,G/P)$ be one of $(\SO_{0}(n+1,1),S^{n}_{\infty})$, $(\SU(n+1,1),S^{2n+1}_{\infty})$, $(\Sp(n+1,1),S^{4n+3}_{\infty})$ or $\left(F_{4(-20)},S^{15}_{\infty}\right)$. Let $\F_{\Gamma}$ be the standard homogeneous $(G,G/P)$-foliation on $M = \G \backslash G / K_{P}$, where $\G$ is a torsion-free uniform lattice of $G$ and $K_{P}$ is a maximal compact subgroup of $P$ (Example~\ref{ex:mob}). Here $\GV(\F_{\Gamma})$ is nontrivial as computed in Corollary~\ref{cor:FGGV}. Recall that $\deg \GV(\F_{\Gamma}) = \dim M$. Fix an orientation of $M$ so that $\int_{M} \GV(\F_{\Gamma}) > 0$. Then we show the following.

\begin{thm}\label{thm:superrigidity}
\begin{enumerate}

\item\label{i: F is smoothly conjugate to F_Gamma} If $(G,G/P)$ is one of $(\SO_{0}(n+1,1),S^{n}_{\infty})$ for odd $n > 1$, $(\SU(n+1,1),S^{2n+1}_{\infty})$ for $n \geq 1$, $(\Sp(n+1,1),S^{4n+3}_{\infty})$ for $n \geq 0$ or $(F_{4(-20)},S^{15}_{\infty})$, then $\F$ is smoothly conjugate to $\F_{\Gamma}$.

\item\label{i: int_M GV(F) leq int_M GV(F_Gamma} If $(G,G/P)$ is $(\SO_{0}(n+1,1),S^{n}_{\infty})$ for even $n$, then any $(G,G/P)$-foliation $\F$ of $M$ satisfies $\int_{M} \GV(\F) \leq \int_{M} \GV(\F_{\Gamma})$. Moreover the equality holds if and only if $\F$ is smoothly conjugate to $\F_{\Gamma}$.

\end{enumerate}
\end{thm}

The essential part of the proof is to generalize the Bott-Thurston-Heitsch type formulas to foliations which may not be transverse to the fibers (Lemma~\ref{lem:r2}). It allows us to apply the rigidity theory of representations of lattices; in particular, the generalized Mostow rigidity~\cite{Corlette1991,Dunfield1999,FrancavigliaKlaff2006} for lattices of $\PSO(n+1,1)$ or $\PSU(n+1,1)$ and the superrigidity~\cite{Corlette1992} of lattices of $\Sp(n+1,1)$ or $F_{4(-20)}$.

In the codimension one case, we will show the following.

\begin{thm}\label{thm:cod1}
If $(G,G/P)$ is one of $(\SO_{0}(2,1),S^{1}_{\infty})$ or $(\SU(1,1),S^{1}_{\infty})$, then any $(G,G/P)$-foliation $\F$ of $M$ satisfies $\GV(\F) = \GV(\F_{\Gamma})$ or $\GV(\F) = 0$. Moreover the former case holds if and only if $\F$ is smoothly conjugate to $\F_{\Gamma}$.
\end{thm}

To prove Theorem~\ref{thm:cod1}, we will apply a minimality theorem of Chihi-ben Ramdane~\cite{ChihiRamdane2008} and theorems of Thurston~\cite{Thurston1972a} and Levitt~\cite{Levitt1978} to isotope $(G,G/P)$-foliations with nontrivial Godbillon-Vey classes so that they are transverse to the fibers of $\G \backslash G / K_{P} \to \G \backslash G / K_{G}$, where $K_{G}$ is a maximal compact subgroup of $G$. Then we can apply generalized Mostow rigidity~\cite{Goldman1988} for surface group representations.

\begin{rem}
Theorem~\ref{thm:cod1} improves a result of Brooks-Goldman~\cite[Theorem~5]{BrooksGoldman1984}. Theorem~\ref{thm:cod1} is also related to Mitsumatsu defect formula~\cite{Mitsumatsu1985} for the $C^{2}$ stable foliations of the geodesic flows of hyperbolic surfaces, and its generalization with weaker regularity assumption by Hurder-Katok~\cite[Theorem~3.11]{HurderKatok1990}.
\end{rem}

\begin{org}
Sections~\ref{sec:1} and~\ref{sec:2} are devoted to recall fundamental notions in this article, as indicated in the table of contents. In Section~\ref{sec:3}, the complexification of secondary characteristic classes of transversely homogeneous foliations is explained, which will be used in Section~\ref{sec:proof} to prove Theorem~\ref{thm:finite}. Section~\ref{sec:example} is devoted to present the examples of the application of Theorem~\ref{thm:finite}. In Section~\ref{sec:ex}, first, the characteristic classes of homogeneous foliations on homogeneous spaces are calculated in terms of Lie algebra cohomology, and then the Bott-Thurston-Heitsch type formulas of Theorem~\ref{thm:BT} are deduced. Theorems~\ref{thm:even} and~\ref{thm:infinite} are proved in Section~\ref{sec:even}. (Note that the computation in Section~\ref{sec:ex} is used in Section~\ref{sec:even}, but it is not necessary for the proof of Theorems~\ref{thm:even} and~\ref{thm:infinite}.) In Section~\ref{sec:rigidity}, Theorems~\ref{thm:superrigidity} and~\ref{thm:cod1} are proved by applying the modification of the Bott-Thurston-Heitsch type formulas of Theorem~\ref{thm:BT}.
\end{org}

\begin{ack}
We thank Juan Francisco Torres Lopera, Takashi Tsuboi, Bertrand Deroin, Masayuki Asaoka, and MathOverFlow users Tilman and Andr\'{e} Henriques for helpful discussions about the contents of this paper. We are grateful to Michelle Bucher because she taught the second author the application of the Hirzebruch proportionality principle and the proof of the generalized Milnor-Wood inequality. We are indebted to Corey Bregman for pointing out an error in Theorem~\ref{thm:BT1} in an old version of this article.
\end{ack}

\section{Secondary characteristic classes of foliations}\label{sec:1}

\subsection{Fundamentals of secondary characteristic classes}\label{sec:secintro}

Consider the Weil algebra $W(\gl(q;\RR)) = \bigwedge \gl(q;\RR)^{*} \otimes S \gl(q;\RR)^{*}$ of $\gl(q;\RR)$, and its $O(q)$-basic subalgebra,
\[
W(\gl(q;\RR))_{O(q)} \\
= \{\, \beta \in W(\gl(q;\RR)) \mid \iota_{X}\beta=0\ \forall X \in \o(q),\ \Ad(g)^{*}\beta = \beta\ \forall g \in O(q) \,\}\;.
\]
For a principal $\GL(q;\RR)$-bundle $E$ over a smooth manifold $M$ with a  $\GL(q;\RR)$-connection $\nabla^{E}$, the Chern-Weil construction yields a homomorphism of differential graded algebras, $\widehat{\Delta}_{E} : W(\gl(q;\RR)) \to \Omega^{\bt}(E)$. Since the image of $W(\gl(q;\RR))_{O(q)}$ under $\widehat{\Delta}_{E} $ is contained in the image of the pull-back map $\pi^{*} : \Omega^{\bt}(E/O(q)) \to \Omega^{\bt}(E)$ by the $O(q)$-basicness, we get a differential map
\[
\Delta_{E} : W(\gl(q;\RR))_{O(q)} \longrightarrow \Omega^{\bt}(E/O(q))\;.
\]
By the contractibility of the fibers of $E/O(q) \to M$, there exists a section $s : M \to E/O(q)$. Thus we get a differential map given by the composite
\[
\xymatrix{ W(\gl(q;\RR))_{O(q)} \ar[r]^<<<<<{\Delta_{E}} & \Omega^{\bt}(E/O(q)) \ar[r]^<<<<<{s^{*}} & \Omega^{\bt}(M)\;.} 
\]
It is known that 
\[
W(\gl(q;\RR))_{O(q)} = \bigwedge[h_{1}, h_{3}, \ldots, h_{[q]}] \otimes \RR [c_{1}, c_{2}, \ldots, c_{q}]
\]
as a differential graded algebra, where $[q]$ is the maximal odd number less than $q+1$. Its grading is given by $\deg h_{i} = 2i-1$ and $\deg c_{i} = 2i$, and its differential map is determined by $dh_{i} = c_{i}$ and $dc_{i} = 0$. Here, $c_{i}$ is the $i$-th Chern polynomial given by $\det (I_{q} + \frac{t}{2\pi}A) = \sum_{j=0}^{q}c_{j}(A)\,t^{j}$~\cite[pp.~138 and~139]{KamberTondeur1975b}. (Note that these Chern polynomials differ from the usual one by $\sqrt{-1}$-factors.) This construction yields nothing for a general $\GL(q;\RR)$-connection because $H^{\bt}(W(\gl(q;\RR))_{O(q)}) = 0$. The normal bundle $\nu \F = TM/T\F$ of a foliated manifold $(M,\F)$ has a special $\gl(q;\RR)$-connection called a {\em Bott connection}~\cite{Bott1972}. For a Bott connection $\nabla$ on $\nu \F$, the frame bundle $\mathcal{P}(\nu \F)$ with the principal $\GL(q;\RR)$-connection associated to $\nabla$ satisfies $\Delta_{\mathcal{P}(\nu \F)}(c_{i}) = 0$ for $i > q$ by Bott vanishing theorem. Thus, letting
\[
WO_{q} = \bigwedge[h_{1}, h_{3}, \ldots, h_{[q]}] \otimes \RR [c_{1}, c_{2}, \ldots, c_{q}]/\mathcal{I}_{q}\;,
\]
where $\mathcal{I}_{q}$ is the ideal of $\RR [c_{1}, c_{2}, \ldots, c_{q}]$ generated by the elements of degree greater than $2q$, we get a differential map $\Delta_{\F}:WO_{q} \to \Omega^{\bt}(M)$. The map induced on cohomology,
\[
\Delta_{\F} : H^{\bt}(WO_{q}) \longrightarrow H^{\bt}(M; \RR)\;,
\]
depends only on $\F$ and is denoted with the same symbol. The cohomology $H^{\bt}(WO_{q})$ is nontrivial, $\Delta_{\F}$ is called the {\em characteristic homomorphism} of $\F$, and the elements of its image are the {\em secondary characteristic classes} of $\F$. For $I = \{i_{1}, \ldots, i_{k}\} \subseteq \{1,3,\cdots,[q]\}$ and $J = \{j_{1}, \ldots, j_{l}\}$, where $1 \leq j_{m} \leq q$, let $h_{I} c_{J} = h_{i_{1}} \cdots h_{i_{k}} c_{j_{1}} \cdots c_{j_{l}}$. Vey showed that the union of Pontryagin classes $\{\, c_{J} \mid \text{$j$ is even}\ \forall j \in J \,\}$ and exotic classes 
\[
\{\, h_{I} c_{J} \mid i_{1} + |J| \geq q+1,\ i_{1} \leq j\ \text{for any odd}\ j \in J\,\}
\]
 is a basis of $H^{\bt}(WO_{q})$ as an $\RR$-vector space, where $i_{1} = \min I$~\cite[Theorem~2]{Heitsch1973}.


\begin{ex}\label{ex:GV}
Let $\F$ be a codimention $q$ foliation on $M$ defined by the kernel of a $q$-form $\omega$. By the Frobenius theorem, we have $d\omega = \eta \wedge \omega$ for some $1$-form $\eta$. Then $\eta \wedge (d\eta)^{q}$ is a closed $(2q+1)$-form on $M$, which is equal to $(2\pi)^{q+1}[\Delta_{\F}(h_{1}c_{1}^{q})]$~\cite[Theorem~7.20]{KamberTondeur1975b}. This characteristic class $(2\pi)^{q+1}[\Delta_{\F}(h_{1}c_{1}^{q})]$ is called the {\em Godbillon-Vey class} of $\F$~\cite{GodbillonVey1971}. The notation $\GV(\F) = (2\pi)^{q+1}[\Delta_{\F}(h_{1}c_{1}^{q})]$ is standard and used many times in this article.
\end{ex}

\subsection{Examples of foliations with nontrivial characteristic classes}

Quotient of homogeneous foliations on homogeneous spaces by lattices are the main examples of foliations with nontrivial secondary characteristic classes.

\begin{ex}[Roussarie's example~\cite{GodbillonVey1971}]\label{ex:R}
Let $\Gamma$ be a torsion-free uniform lattice of $\SL(2;\RR)$. Let $\pi : \SL(2;\RR) \to  \SL(2;\RR)/\Aff (1;\RR)$ be the canonical projection, where $\Aff (1;\RR)$ is the subgroup of $\SL(2;\RR)$ given in Remark~\ref{rem:aff}. Then the fibers of $\pi$ induce a codimension one foliation on $M=\Gamma \backslash \SL(2;\RR)$. Let $\{\omega, \eta, \theta\}$ be a basis of $\sl(2;\RR)^{*}$ so that the fibers of $\pi$ are defined by $\ker \omega$ and
\[
d\omega  = \eta \wedge \omega\;, \quad
d\eta  = \omega \wedge \theta\;, \quad
d\theta  = -\eta \wedge \theta\;.
\]
By their left invariance, the $1$-forms $\omega$, $\eta$ and $\theta$ on $\SL(2;\RR)$ induce $1$-forms on $M$, which are denoted with the same symbols. Let $\F$ be the foliation on $M$ defined by the kernel of $\omega$. By the definition of $\GV(\F)$, we get
\[
\GV(\F) = [\eta \wedge d\eta] = [\eta \wedge \omega \wedge \theta]\;.
\]
Since $\eta \wedge \omega \wedge \theta$ is a volume form on $M$, it follows that $\GV(\F) \neq 0$. In fact, by the Bott-Thurston formula (Theorem~\ref{thm:BT} for $q=1$), we get 
\[
\int_{M} \GV(\F)=16\pi^{2}e\;,
\] 
where $e$ is the Euler number of the surface $\Gamma \backslash \SL(2;\RR)/\SO(2)$. Note that, in the original setting of the Bott-Thurston formula (\cite{Thurston1972b} and~\cite[Appendix by Brooks]{Bott1978}), they considered the $S^{1}$-bundle $\PSL(2;\RR) \to \PSL(2;\RR)/\Aff (1;\RR)$ and hence the formula is $\int_{M} \GV(\F)=4\pi^{2}e$.
\end{ex}

\begin{ex}\label{ex:mob}
The following example is a generalization of the last example to higher dimensions. Let $G$ be $\SO_{0}(n+1,1)$, $\SU(n+1,1)$, $\Sp(n+1,1)$ or $F_{4(-20)}$, and consider $G/P$ as the ideal boundary of the corresponding hyperbolic symmetric space $G/K_{G}$:
\begin{align*}
\mathbf{H}_{\RR}^{n} & = {\SO}_{0}(n+1,1)/\SO(n+1)\;, \\
\mathbf{H}_{\CC}^{n} & = \SU(n+1,1)/\S(\U(n+1)\U(1))\;, \\
\mathbf{H}_{\HH}^{n} & = \Sp(n+1,1)/\Sp(n+1)\Sp(1)\;, \\
\mathbf{H}_{\OO}^{2} & = F_{4(-20)}/\Spin(9)\;.
\end{align*}
Let $K_{G}$ be a maximal compact subgroup of $G$, and take the maximal compact subgroup $K_{P} = K_{G} \cap P$ of $P$. The ideal boundary of $G/K_{G}$ is a sphere of real dimension $n$, $2n+1$, $4n+3$ and $15$, respectively. $\G \backslash G/ K_{P}$ admits a foliation $\F_{\G}$ whose lift to $G/ K_{P}$ is defined by the fibers of $G/ K_{P} \to G/P$. Here, $\G \backslash G/ K_{G}$ is a real, complex, quaternionic or octonionic hyperbolic manifold, and $\G \backslash G/ K_{P} \to \G \backslash G/ K_{G}$ is its unit tangent sphere bundle (see Section~\ref{sec:example}), depending on the choice of $G$. Later, we will compute $\GV(\F_{\G})$ (Proposition~\ref{prop:volGV}), and this Godbillon-Vey class is essentially the unique nontrivial secondary characteristic class for these foliations (Section~\ref{sec:GVspan}). Yamato~\cite{Yamato1975} studied the secondary characteristic classes of $\F_{\G}$ in the case where $G=\SO_{0}(n+1,1)$.
\end{ex}

\begin{ex}\label{ex:hom2}
This is a further generalization of the last example. Let $G$ be a Lie group and $P$ a closed subgroup of $G$. Let $K$ be a closed subgroup of $P$. Let $\G$ be a torsion-free uniform lattice of $G$. Then the fibers of the canonical projection $G/K \to G/P$ define a foliation $\F_{\G}$ on a closed manifold $\G\backslash G/K$. The characteristic classes of this type of foliations were extensively studied and calculated by Kamber-Tondeur~\cite{KamberTondeur1975a}, Baker~\cite{Baker1978}, Heitsch~\cite{Heitsch1978}, Pittie~\cite{Pittie1979}, Pelletier~\cite{Pelletier1983} and Asuke~\cite{Asuke2010}.
\end{ex}

\section{Transversely homogeneous foliations}\label{sec:2}

\subsection{Definition of $(G,G/P)$-foliations}

Let $(M,\F)$ be a foliated manifold. Let $G$ be a Lie group and $P$ a closed subgroup of $G$. When the group $G$ is endowed with the discrete topology, it is denoted by $G^{\delta}$. We denote the $G$-action on $G/P$ by $(g,xP) \mapsto g \cdot xP$.

\renewcommand{\theenumi}{\arabic{enumi}}

\begin{defn}\label{defn:trhomfol}
A ({\em Haefliger\/}) {\em cocycle\/} with values in $(G,G/P)$, defining $\F$, is a triple $(\{U_{i}\},\{\pi_{i}\},\{\gamma_{ij}\})$, where:
\begin{enumerate}

\item $\{U_{i}\}$ is an open covering of $M$,

\item each $\pi_{i}$ is a submersion $U_{i} \to G/P$ such that the leaves of $\F|_{U_{i}}$ are the fibers of $\pi_{i}$, and

\item each $\gamma_{ij}$ is a continuous map $U_{i} \cap U_{j} \to G^{\delta}$ such that $ \pi_{i}(x) = \gamma_{ij}(x) \cdot \pi_{j}(x)$ for any $x\in U_{i} \cap U_{j}$.

\end{enumerate}
Two cocycles with values in $(G,G/P)$, defining $\F$, are called {\em equivalent\/} when their union is contained in some cocycle with values in $(G,G/P)$, defining $\F$. When $\F$ is endowed with an equivalence class of cocycles (or a maximal cocycle) with values in $(G,G/P)$, defining $\F$, it is called a {\em $(G,G/P)$-foliation\/}.
\end{defn}

\renewcommand{\theenumi}{\roman{enumi}}

Cocycles valued in $(G,G/P)$ are examples of $1$-cocycles valued in groupoids defined by Haefliger~\cite{Haefliger1958}. Transversely homogeneous foliations are natural generalizations of quotients of homogeneous foliations on homogeneous spaces in terms of $1$-cocycles with values in groupoids.

\begin{ex}
Example~\ref{ex:R} is an $(\SL(2;\RR),S^{1})$-foliation, and Example~\ref{ex:hom2} is a $(G,G/P)$-foliation. Example~\ref{ex:mob} is a special case of Example~\ref{ex:hom2}, where $(G,G/P)$ is $(\SO_{0}(n+1,1), S^{n}_{\infty})$, $(\SU(n+1,1), S^{2n+1}_{\infty})$, $(\Sp(n+1,1), S^{4n+3}_{\infty})$ or $(F_{4(-20)}, S^{15}_{\infty})$, and where $S^{n}_{\infty}$, $S^{2n+1}_{\infty}$, $S^{4n+3}_{\infty}$ or $S^{15}_{\infty}$ are the ideal boundaries of the corresponding hyperbolic symmetric spaces. 
\end{ex}

\begin{ex}[Suspension foliations]\label{ex:suspfol} Let $N$ be a smooth manifold and $h : \pi_{1}N \to G$ a homomorphism. A $\pi_{1}N$-action on $G/P$ is defined by $\pi_{1}N \to G \to \Diff(G/P)$, where the second homomorphism is the $G$-action on $G/P$. Then the quotient space $\wt{N} \times_{\pi_{1}N} G/P$ of the diagonal $\pi_{1}N$-action on $\wt{N} \times G/P$ has a foliation $\F$ induced by the horizontal foliation $\wt{N} \times G/P = \bigsqcup_{x \in G/P} \wt{N} \times \{x\}$. Here, it is easy to see that $\F$ naturally admits a structure of $(G,G/P)$-foliation by definition. (One can also apply Proposition~\ref{prop:h1} below.)
\end{ex}

\begin{ex}
Let $(M_{i},\F_{i})$ be a smooth manifold with a $(G,G/P)$-foliation for $i\in\{0,1\}$. Assume that we have a closed transversal $S_{i}$ of $(M_{i},\F_{i})$ such that $S_{0}$ is diffeomorphic to $S_{1}$ as $(G,G/P)$-manifolds. Let $U_{i}$ be an open tubular neighborhood of $S_{i}$ such that the leaves of $\F_{i}|_{U_{i}}$ are the fibers of a normal bundle of $S_{i}$. We can paste $U_{0} \setminus S_{0}$ and $U_{1} \setminus S_{1}$ to construct another manifold with a $(G,G/P)$-foliation. Chihi and ben Ramdane~\cite{ChihiRamdane2008} used this method to construct manifolds with $(\SL(2;\RR),S^{1})$-foliations with nontrivial Godbillon-Vey classes and dense holonomy groups in $\SL(2;\RR)$.
\end{ex}

\begin{ex}
Let $(M,\F)$ be a smooth manifold with a $(G,G/P)$-foliation. If we have a smooth map $f : M' \to M$ which is transverse to $\F$, we can pull back $\F$ to $M'$ as a $(G,G/P)$-foliation. This construction can be used when $f$ is a branched covering whose branch locus is transverse to $\F$.
\end{ex}

\begin{ex}
Thurston~\cite{Thurston1972b} constructed examples of codimension one foliations on Seifert fibered $3$-manifolds whose Godbillon-Vey class varies nontrivially by making surgery to Example~\ref{ex:R}. Rasmussen~\cite{Rasmussen1980} generalized this construction to the case of codimension two. Thurston also constructed families of suspension foliations on the total spaces of $S^{1}$-bundles over closed surfaces of genus two whose secondary characteristic classes vary nontrivially. These examples are constructed by pasting two transversely projective foliations of the total space of $S^{1}$-bundles over punctured tori~\cite[Section~4]{Bott1978}. Heitsch~\cite{Heitsch1978} constructed families of $( \prod_{i=1}^{k} \SL(n_{i};\RR), S^{(\sum_i n_{i})-1} )$-foliations whose characteristic classes vary by deforming $\prod_{i=1}^{k} \SL(n_{i};\RR)$-actions on $S^{(\sum_i n_{i})-1}$.
\end{ex}

\subsection{Haefliger type description of transversely homogeneous foliations}

\subsubsection{Flat principal $G$-bundle associated to $\F$ and the holonomy homomorphism}
Let $(M,\F)$ be a $(G,G/P)$-foliation defined by a cocycle $(\{U_{i}\},\{\pi_{i}\},\{\gamma_{ij}\})$ valued in $(G,G/P)$. The condition $\pi_{i} = \gamma_{ij} \cdot \pi_{j}$ implies the $1$-cocycle condition $\gamma_{ik} = \gamma_{ij} \cdot \gamma_{jk}$. Thus $\{\gamma_{ij}\}$ is a $1$-cocycle valued in $G^\delta$, which defines a flat principal $G$-bundle $\pi_{G} : \X_{G}(\F) \to M$. Recall that 
\[
\X_{G}(\F) = \Big(\bigsqcup_{i} U_{i} \times G \Big)\Big/ (x,y) \sim (x,\gamma_{ij}(x)(y))\;,
\]
and the projection $\pi_{G}$ is induced by the first factor projections $U_{i} \times G \to U_{i}$. The holonomy homomorphism $\pi_{1}M \to G$ of this flat $G$-bundle is called the {\em holonomy homomorphism} of $\F$ and denoted by $\hol (\F)$.

\subsubsection{The Haefliger structure of $\F$}\label{sec:Haefliger}

We recall the description of $(G,G/P)$-foliations in terms of a $G/P$-bundle over $M$, which is a special case of the Haefliger structures of general foliations. It was studied by Blumenthal~\cite{Blumenthal1979} and used by Brooks-Goldman~\cite{BrooksGoldman1984} and Heitsch~\cite{Heitsch1986} to prove Theorem~\ref{thm:BGH}. 

\begin{prop}\label{prop:h1}
A $(G,G/P)$-foliation $\F$ on $M$ is determined by any of the following data:
\begin{enumerate}

\item A flat principal $G$-bundle $\X_{G} \to M$ and a section $s$ of $\X_{G}/P \to M$ such that $s$ is transverse to the foliation $\E$ of $\X_{G}/P$ defined by the flat $G$-connection.

\item A group homomorphism $\hol : \pi_{1}M \to G$ and a submersion $\dev : \widetilde{M} \to G/P$ such that $\dev (\gamma \cdot x) = \hol(\gamma) \cdot \dev(x)$ for all $x\in\wt{M}$ and $\gamma\in\pi_{1}M$. 

\end{enumerate}
\end{prop}

Let $\overline{\gamma}_{ij}(x) : G/P \to G/P$ be the diffeomorphism induced by the left product of $\gamma_{ij}(x)$. Here, $\{\overline{\gamma}_{ij}\}$ is a $1$-cocycle valued in $\Diff(G/P)^{\delta}$, which defines a $G/P$-bundle $\pi_{G/P} : \X_{G/P}(\F) \to M$ with a flat $G$-connection whose holonomy homomorphism is equal to $\hol (\F)$. Recall that 
\[
\X_{G/P}(\F) = \Big(\bigsqcup_{i} U_{i} \times G/P \Big)\Big/ (x,y) \sim (x,\overline{\gamma}_{ij}(x)(y)) = \X_{G}(\F)/P\;,
\]
and the projection $\pi_{G/P}$ is induced by the first factor projections $U_{i} \times G/P \to U_{i}$. The graphs of the maps $\pi_{i}$,
\[
\Graph(\pi_{i}) = \{ \,(x,\pi_{i}(x)) \mid x\in U_{i}\, \} \subset U_{i} \times G/P\;,
\]
define a subset of $\X_{G/P}(\F)$, which gives a global section $s$ of $\X_{G/P}(\F) \to M$. By construction, $\F$ is obtained as the pull-back by $s$ of the foliation of $\X_{G/P}(\F)$ defined by the flat connection. Summarizing, $\F$ determines a flat $G/P$-bundle $\pi_{G/P} : \X_{G/P}(\F) \to M$ with a section $s$, which in turn determines $\F$.

Let $\wt{M}$ be the universal cover of $M$. The pull-back of $\X_{G}(\F)/P \to M$ to $\wt{M}$ is a trivial flat $G/P$-bundle. A section $s$ of $\X_{G}(\F)/P \to M$ yields a section $\wt{s}$ of this trivial $G/P$-bundle over $\wt{M}$ by pull-back. In an obvious way, giving $\wt{s}$ is equivalent to giving a submersion $\dev : \widetilde{M} \to G/P$ that is $\pi_{1}M$-equivariant with respect to $\hol (\F) : \pi_{1}M \to G$; i.e., $\dev (\gamma \cdot x) = \hol (\F)(\gamma) \cdot \dev(x)$ for $x\in\wt{M}$ and $\gamma\in\pi_{1}M$. 

\subsubsection{Enlargement of the Haefliger structure of $\F$}

We will use a bundle larger than the one described in the last section, which was used by Benson-Ellis~\cite{BensonEllis1985}. Let $K_{P}$ be a maximal compact subgroup of $P$. We consider a $G/K_{P}$-bundle $\pi_{G/K_{P}} : \X_{G}(\F)/K_{P} \to M$ with a flat $G$-connection constructed by a $1$-cocycle valued in $\Diff(G/K_{P})^{\delta}$ in a way analogous to $\pi_{G/P}$ in the last section. There is also a $P/K_{P}$-bundle $p : \X_{G}(\F)/K_{P} \to \X_{G}(\F)/P$. Since $P/K_{P}$ is contractible, there is a section $s'$ of $p$, which is unique up to homotopy. We get a section $\hat{s}$ of $\pi_{G/K_{P}}$ defined by the composite 
\[
\xymatrix{ M \ar[r]^<<<<<{s} & \X_{G}(\F)/P \ar[r]^<<<<<{s'} & \X_{G}(\F)/K_{P}\;. }
\]
Clearly, $\hat{s}$ is transverse to the foliation $p^{*}\E_{\hol (\F)}$ of $\X_{G}(\F)/K_{P}$, where $\E_{\hol (\F)}$ is the foliation of $\X_{G}(\F)/P$ defined by the flat $G$-connection. Thus we get the following. 

\begin{prop}\label{prop:h2}
A $(G,G/P)$-foliation $\F$ on $M$ is determined by any of the following data:
\begin{enumerate}

\item\label{i: X_G} A flat principal $G$-bundle $\X_{G} \to M$ and a section $\hat{s}$ of $\X_{G}/K_{P} \to M$ such that $\hat{s}$ is transverse to the foliation $p^{*}\E$ of $\X_{G}/K_{P}$, where $p : \X_{G}/K_{P} \to \X_{G}/P$ is the canonical projection and $\E$ is the foliation of $\X_{G}/P$ defined by the flat $G$-connection.

\item\label{i: hol and widehat dev} A homomorphism $\hol : \pi_{1}M \to G$ and a smooth map $\widehat{\dev} : \widetilde{M} \to G/K_{P}$ such that $\widehat{\dev}$ is transverse to the foliation defined by the fibers of $G/K_{P} \to G/P$ and $\widehat{\dev} (\gamma \cdot x) = \hol(\gamma) \cdot \widehat{\dev}(x)$ for all $x\in \wt{M}$ and $\gamma\in \pi_{1}M$. 

\end{enumerate}
\end{prop}

\section{Characteristic classes of transversely homogeneous foliations}\label{sec:3}

\subsection{Bott connections on the $P/K_{P}$-coset foliation of $G/K_{P}$}

Assume that $G$ is a connected semisimple Lie group and $P$ is a parabolic  subgroup of $G$. Recall that $K_{P}$ is a maximal compact subgroup of $P$. In this section, we will recall the well known construction of a left invariant Bott connection on the normal bundle of the right $P/K_{P}$-coset foliation $\F_{P}$ on $G/K_{P}$, originally due to Kamber-Tondeur~\cite[Theorem~3.7]{KamberTondeur1975a} (announced in~\cite{KamberTondeur1974}).

Let $\g$ and $\p$ denote the Lie algebras of $G$ and $P$, respectively. Let $\sigma : \g/\p \to \g$ be a splitting of the exact sequence 
\[
\xymatrix{0 \ar[r] & \p \ar[r] & \g \ar[r]^<<<<<{\pi} & \g/\p \ar[r] & 0\;. } 
\]
Then consider the connection $\wt{\nabla}$ on the normal bundle $\nu \mathcal{G}_{P}$ of the right $P$-coset foliation $\mathcal{G}_{P}$ on $G$ determined by
\[
\wt{\nabla}_{X} Y = \pi \big([(\textstyle{\id_{\g}} - \sigma \pi)X, \sigma(Y)]\big)
\]
for $X\in\g$ and $Y\in\g/\p$. Observe that $\wt{\nabla}$ is left invariant. For $X\in\p$, we get $\wt{\nabla}_{X}Y = \ad(X)(Y)$. This fact implies that $\wt{\nabla}$ is a Bott connection on $\nu \mathcal{G}_{P}$. If we take an $\ad K_{P}$-equivariant section $\sigma$, then $\wt{\nabla}$ induces a left invariant Bott connection $\nabla$ on $\nu\F_{P}$.

Let $(\bigwedge \g^*)_{K_{P}}$ be the $K_{P}$-basic subalgebra of $\bigwedge \g^*$; namely, 
\[
\left(\bigwedge\g^{*}\right)_{K_{P}} = \left\{\, \beta \in \bigwedge\g^{*} \, \Big| \, \iota_{X}\beta=0\ \forall X \in \Lie (K_{P}),\ \Ad(g)^{*}\beta = \beta\ \forall g \in K_{P} \,\right\}\;,
\]
which is identified to the algebra of left invariant differential forms on $G/K_{P}$. By the left invariance of $\nabla$, we get $\Delta_{\F_{P}} : WO_{q} \to \left(\bigwedge\g^{*}\right)_{K_{P}}$. 

Let $K_{P} = P \cap K_{G}$, which is a maximal compact subgroup of $P$. Since $P$ is parabolic, we have a finite subgroup $F$ of $K_{P}$ such that $P = FP_{0}$, where $P_{0}$ is the identity component of $P$ (see~\cite[Proposition 7.82]{Knapp1996}). Let $\Gc$ be a connected complex Lie group with $\Lie (\Gc) = \g \otimes \CC$. Let $K'$ be the connected Lie subgroup of $\Gc$ such that $\Lie(K') = \Lie (K_{P}) \otimes \CC$. Let $(K_{P})_{\CC} = FK'$. Note that $K_{P} \subset (K_{P})_{\CC}$. Let $\gc^{*} = \Hom_{\CC}(\g \otimes \CC,\CC)$. By complexifying $\nabla$, we get a complex connection $\nabla^{\CC}$ on the complexified normal bundle of the right $\Pc$-coset foliation $\F_{\Pc}$ on $\Gc/(K_{P})_{\CC}$, obtaining the characteristic homomorphism $\Delta_{\F_{\Pc}} : WO_{q} \otimes \CC \to \left(\bigwedge\gc^{*}\right)_{(K_{P})_{\CC}}$. Thus we get that the following diagram commutes:
\begin{equation}\label{eq:ext1}
\xymatrix{                 && (\bigwedge\gc^{*})_{(K_{P})_{\CC}} \ar[d]  \\
WO_{q} \otimes \CC \ar[rru]^{\Delta_{\F_{\Pc}}} \ar[rr]_{\Delta_{\F_{P}}} && (\bigwedge\g^{*})_{K_{P}} \otimes \CC\;,}
\end{equation}
where the vertical arrow is canonical.

\subsection{Complexification of the enlargement of Haefliger structures}

We use the notation on Lie groups and Lie algebras introduced in the last section. Let $\F$ be a $(G,G/P)$-foliation of a manifold $M$. Let $\pi_{G/K_{P}} : \X_{G}(\F)/K_{P} \to M$ be the enlargement of the Haefliger structure considered in Proposition~\ref{prop:h2}. We construct the fiberwise complexification of $\pi_{G/K_{P}}$ as follows. Let  $\hol (\F)_{\CC}$ denote the composite 
\[
\xymatrix{\pi_{1}M \ar[r]^<<<<<{\hol (\F)} & G \ar[r] & \Gc\;.}
\]
Let $\X_{\Gc}(\F)$ be the quotient of $\widetilde{M} \times \Gc$ by the diagonal action of $\pi_{1}M$, obtaining a flat principal $\Gc$-bundle $\pi_{\Gc} : \X_{\Gc}(\F) \to M$ whose holonomy homomorphism is $\hol (\F)_{\CC}$. Then we get a canonical map $\X_{G}(\F)/K_{P} \to \X_{\Gc}(\F)/(K_{P})_{\CC}$, which is a complexification map $G/K_{P} \to \Gc/(K_{P})_{\CC}$ on each fiber. Thus a section $s$ of $\X_{G}(\F)/K_{P}\to M$ gives a section $s_{\CC}$ of $\X_{\Gc}(\F)/(K_{P})_{\CC} \to M$.

The universal covers of $\X_{G}(\F)/K_{P}$ and $\X_{\Gc}(\F)/(K_{P})_{\CC}$ are the products $\widetilde{M} \times G/K_{P}$ and $\widetilde{M} \times \Gc/(K_{P})_{\CC}$, respectively. Consider the diagram
\[
\xymatrix{  (\bigwedge\gc^{*})_{(K_{P})_{\CC}} \ar[rr] \ar[d]  && \Omega^{\bt}(\widetilde{M} \times \Gc/(K_{P})_{\CC}; \CC) \ar[d] \\
               (\bigwedge\g^{*})_{K_{P}} \otimes \CC \ar[rr] && \Omega^{\bt}(\widetilde{M} \times G/K_{P}; \CC)\;, }
\]
where the horizontal arrows are the pull-back homomorphisms  by the second projections and the vertical arrows are the canonical maps defined by complexification. Since $\pi_{1}M$ acts on $G/K_{P}$ and $\Gc/(K_{P})_{\CC}$ by the left product of $G$, left invariant forms on $G/K_{P}$ and $\Gc/(K_{P})_{\CC}$ descend to $\X_{G}(\F)/K_{P}$ and $\X_{\Gc}(\F)/(K_{P})_{\CC}$. Thus we get the commutative diagram
\begin{equation}\label{eq:XXc}
\xymatrix{  (\bigwedge\gc^{*})_{(K_{P})_{\CC}} \ar[rr] \ar[d]  && \Omega^{\bt}(\X_{\Gc}(\F)/(K_{P})_{\CC}; \CC)  \ar[d] \\
               (\bigwedge\g^{*})_{K_{P}} \otimes \CC \ar[rr] && \Omega^{\bt}(\X_{G}(\F)/K_{P}; \CC)\;. }
\end{equation}

Combining the diagrams~\eqref{eq:ext1} and~\eqref{eq:XXc}, we get the following.

\begin{prop}\label{prop:ext2}
The following diagram is commutative:
\[
\xymatrix{                    && H^{\bt}(\X_{\Gc}(\F)/(K_{P})_{\CC}; \CC) \ar[d] \\
H^{\bt}(WO_{q} \otimes \CC) \ar[rru]^{\Delta_{\widehat{\E}_{\CC}}} \ar[rr]_{\Delta_{\widehat{\E}}} && H^{\bt}(\X_{G}(\F)/K_{P}; \CC)\;, }
\]
where $\widehat{\E}$ is the pull-back of the $(G,G/P)$-foliation of $\X_{G}(\F)/P$ by the projection $\X_{G}(\F)/K_{P} \to \X_{G}(\F)/P$, and $\widehat{\E}_{\CC}$ is the pull-back of the $(\Gc,\Gc/\Pc)$-foliation of $\X_{\Gc}(\F)/P_{\CC}$ by the projection $\X_{\Gc}(\F)/(K_{P})_{\CC} \to \X_{\Gc}(\F)/P_{\CC}$.

\end{prop}

The following simple observation is the unique new idea in our proof of Theorem~\ref{thm:finite}.
\begin{prop}\label{prop:ext3}
Assume that $H^{\bt}(\Gc/P;\RR) \to H^{\bt}(G/P;\RR)$ is trivial on positive degrees. Then the image of $\Delta_{\widehat{\E}} : H^{\bt}(WO_{q}) \to H^{\bt}(\X_{G}(\F)/K_{P};\RR)$ is contained in the image of $\pi_{G/K_{P}}^{*} : H^{\bt}(M;\RR) \to H^{\bt}(\X_{G}(\F)/K_{P};\RR)$.
\end{prop}

\begin{proof}
By Proposition~\ref{prop:ext2}, the image of $\Delta_{\widehat{\E}}$ is contained in the image of 
\[
H^{\bt}(\X_{\Gc}(\F)/(K_{P})_{\CC};\RR) \longrightarrow H^{\bt}(\X_{G}(\F)/K_{P};\RR)\;.
\]
Consider the Leray-Serre spectral sequences associated to the fiber bundles $\X_{\Gc}(\F)/(K_{P})_{\CC} \to M$ and $\X_{G}(\F)/K_{P} \to M$. Since $(K_{P})_{\CC}$ and $K_{P}$ are homotopy equivalent to $P$, it follows that $\X_{\Gc}(\F)/(K_{P})_{\CC}$ and $\X_{G}(\F)/K_{P}$ are homotopy equivalent to $\X_{\Gc}(\F)/P$ and $\X_{G}(\F)/P$, respectively. Thus the restriction map between the $E_{2}$-terms is given by
\[
r : H^{\bt}\big(M, \H^{\bt}(\Gc/P) \big) \longrightarrow H^{\bt}\big(M, \H^{\bt}(G/P)\big)\;,
\]
where $\H^{\bt}(\Gc/P)$ and $\H^{\bt}(G/P)$ are the corresponding local systems associated to $\X_{\Gc}(\F)/P$ and $\X_{G}(\F)/P$, respectively. By the assumption of triviality of $H^{\bt}(\Gc/P;\RR) \to H^{\bt}(G/P;\RR)$ on positive degrees, it follows that the image of $r$ is contained in $H^{\bt}(M;\RR)$.
\end{proof}

\subsection{Two results of Benson-Ellis}

Let $H^{\bt}(\g,K_{P}) = H^{\bt}((\bigwedge\g^{*})_{K_{P}})$. Let $\F$ be a $(G,G/P)$-foliation of a manifold $M$. Assume that $G$ is semisimple. 

\begin{thm}[Benson-Ellis~\cite{BensonEllis1985}]\label{thm:factrization}
The following diagram commutes:
\begin{equation*}
\xymatrix{  & H^{\bt}(\g,K_{P}) \ar[rd]  &             \\
H^{\bt}(WO_{q}) \ar[ru]^{\Delta_{\F_{P}}} \ar[rr]_{\Delta_{\F}} & & H^{\bt}(M;\RR)\;,}
\end{equation*}
\end{thm}

Note that the argument in the last section gives an alternative proof of Theorem~\ref{thm:factrization}. 

Let $U$ be an open subset of $\RR^{\ell}$.

\begin{thm}[{Benson-Ellis~\cite{BensonEllis1985}, see also Haefliger~\cite[Theorem in Section~6]{Heitsch1986}}]\label{thm:rigid} 
For a smooth family $\{\F_{t}\}_{t \in U}$ of $(G,G/P)$-foliations of $M$, the family $\{\Delta_{\F_{t}}\}_{t \in U}$ is locally constant in $\Hom(H^{\bt}(WO_{q}),H^{\bt}(M;\RR))$.
\end{thm}

This rigidity comes from the vanishing results of cohomology of representations of semisimple Lie algebras.

\section{A sufficient topological condition for the finiteness of secondary classes}
\label{sec:proof}

Like in the proof of Theorem~\ref{thm:BGH} by Brooks-Goldman and Heitsch, the unique essential part of the proof of the finiteness theorem (Theorem~\ref{thm:finite}) is the following proposition.
 
\begin{prop}\label{prop:01}
If the holonomy homomorphisms of two $(G,G/P)$-foliations, $\F_{0}$ and $\F_{1}$, on $M$ are homotopic, then $\Delta_{\F_{0}} = \Delta_{\F_{1}}$.
\end{prop}

Theorem~\ref{thm:finite} follows from Proposition~\ref{prop:01} with the arguments of~\cite[Lemma~2]{BrooksGoldman1984}.

\begin{proof}[Proof of Theorem~\ref{thm:finite} using Proposition~\ref{prop:01}]
Recall that we assume that $\pi_{1}M$ is finitely presented. It is well known that $\pi_{0}(\Hom(\pi_{1}M,G))$ is finite (see Remark~\ref{rem:fincomp} at the end of this section). Thus there exist a finite number of $(G,G/P)$-foliations $\F_1,\dots,\F_m$ of $M$ such that, for any $(G,G/P)$-foliation $\F$ of $M$, its holonomy homomorphism is in the same connected component of $\Hom(\pi_{1}M,G)$ as the holonomy homomorphism of some $\F_{i}$. Thus Proposition~\ref{prop:01} implies Theorem~\ref{thm:finite}.
\end{proof}

\begin{proof}[Proof of Proposition~\ref{prop:01}]
Let $\X_{G}(\F_{i})/K_{P} \to M$ be the enlargement of the Haefliger structure of $\F_{i}$ considered in Proposition~\ref{prop:h2} for $i\in\{0,1\}$. Recall that a section $s_{i} : M \to \X_{G}(\F_{i})/K_{P}$ is associated to $\F_{i}$. Consider the foliation $\widehat{\E}_{i} = p^{*}_{i}\E_{\hol (\F_{i})}$ of $\X_{G}(\F_{i})/K_{P}$, where $p_{i} : \X_{G}(\F_{i})/K_{P} \to \X_{G}(\F_{i})/P$ is the canonical projection and $\E_{\hol (\F_{i})}$ is the foliation of $\X_{G}(\F_{i})/P$ defined by the flat $G$-connection. 

The homotopy class of $(\X_{G}(\F_{i})/K_{P},\widehat{\E}_{i})$ as a $(G,G/P)$-foliation is determined by the homotopy class of the holonomy homomorphism of $\F_{i}$. Thus, by assumption and Theorem~\ref{thm:rigid}, we get $\Delta_{\widehat{\E}_{0}} = \Delta_{\widehat{\E}_{1}}$.

By Proposition~\ref{prop:ext3}, the image of $\Delta_{\widehat{\E}_{0}}$ is contained in the image of $p^{*} : H^{\bt}(M;\RR) \to H^{\bt}(\X_{G}(\F_{0})/K_{P};\RR)$. Thus $s_{0}^{*}\Delta_{\widehat{\E}_{0}} =  s_{1}^{*}\Delta_{\widehat{\E}_{0}}$ on $H^{\bt}(WO_{q})$, and therefore
\[
\Delta_{\F_{0}} = s_{0}^{*}\Delta_{\widehat{\E}_{0}} =  s_{1}^{*}\Delta_{\widehat{\E}_{0}} = s_{1}^{*}\Delta_{\widehat{\E}_{1}} = \Delta_{\F_{1}}\;,
\]
as desired.
\end{proof}

For later reference, note the following fact shown in the proof of Proposition~\ref{prop:01}.

\begin{prop}\label{prop:eachsigma}
Let $\sigma\in H^{\bt}(WO_{q})$. If $\Delta_{\widehat{\E}}(\sigma)$ belongs to the image of $p^{*} : H^{\bt}(M;\RR) \to H^{\bt}(\X_{G}(\F_{0})/K_{P};\RR)$, then 
\[
\# \{\, \Delta_{\F}(\sigma) \in H^{\bt}(M;\RR) \mid \F \in \Fol(G,G/P) \,\} < \infty\;.
\]
\end{prop}

\begin{rem}\label{rem:fincomp}
	\begin{enumerate}
	
		\item For a finitely presented group $S$ with $k$ generators, $\Hom(S,\GL(n;\RR))$ is endowed with the structure of a real algebraic variety via a tautological embedding $j : \Hom(S,\GL(n;\RR)) \to \GL(n;\RR)^{k}$ (this is an observation of Lusztig as written in~\cite[Footnote of p.~186]{Sullivan1976}). For an algebraic group $G$ of $\GL(n;\RR)$, observe that 
			\[
				\Hom(S,G) = j\big(\Hom(S,\GL(n;\RR))\big) \cap G^{k}
			\]
is also a real algebraic variety. Thus $\pi_{0}\big(\Hom(S,G)\big)$ is finite by a theorem of Whitney~\cite{Whitney1957}, and its cardinality is bounded in terms of the degrees of the defining polynomials by a theorem of Milnor~\cite{Milnor1964}.

		\item We indicate an alternative way to prove the finiteness of the Godbillon-Vey class by using the complexification of the Haefliger structure of $\F$ under the assumption of the triviality of $H^{\bt}(\Gc/\Pc;\RR) \to H^{\bt}(G/P;\RR)$ on positive degrees. Note that this assumption is weaker than the assumption of the triviality of $H^{\bt}(\Gc/P;\RR) \to H^{\bt}(G/P;\RR)$ on positive degrees. Consider a $\Gc/\Pc$-bundle $\X_{\Gc}(\F)/\Pc \to M$, which is regarded as the complexification of the Haefliger structure $\X_{G}(\F)/P \to M$ of $\F$. Assume that $c_{1}(\E^{\CC}_{\hol (\F)})$ is trivial if $\dim G/P$ is even. By results of Asuke~\cite[Corollary~1.9 and Proposition~2.2]{Asuke2003}, the Godbillon-Vey class extends to $\X_{\Gc}(\F)/\Pc$. So, if $H^{\bt}(\Gc/\Pc;\RR) \to H^{\bt}(G/P;\RR)$ is trivial on positive degrees, then we get the finiteness of the Godbillon-Vey class like in  the above proof of Theorem~\ref{thm:finite}. 

		\item We can show the triviality of  $H^{\bt}(\Gc/\Pc;\RR) \to H^{\bt}(G/P;\RR)$ on positive degrees by using the Schubert cell decomposition of $\Gc/\Pc$ if $\Gc/\Pc$ is a generalized Bott tower; namely, the total space of consecutive complex projective space bundles and $G/P$ is the total space of the corresponding consecutive real projective space bundles. The Schubert cell decomposition of $\Gc/\Pc$ is a cell decomposition whose cells are orbits of the action of a Borel subgroup of $\Gc$. This cell decomposition induces a cell decomposition of $G/P$. In the case of generalized Bott towers, we can contract the inclusion $G/P \to \Gc/\Pc$ cell by cell to a constant map. 
		
	\end{enumerate}
\end{rem}

\section{Examples}\label{sec:example}
\addtocontents{toc}{\protect\setcounter{tocdepth}{1}}

\subsection{The Euler class of the bundle $\Gc/P \to \Gc/G$}

Let us consider the case of $G/P = S^{q}$. We characterize the assumption of Theorem~\ref{thm:finite} by the nontriviality of the Euler class of the sphere bundle 
\[
\xymatrix{ G/P \ar[r] & \Gc/P \ar[r] & \Gc/G\;,}
\]
which is homotopy equivalent to
\begin{equation}\label{eq:Kseq}
\xymatrix{ K_{G}/K_{P} \ar[r] & K_{\Gc}/K_{P} \ar[r]^<<<<<{\varphi} & K_{\Gc}/K_{G}\;.}
\end{equation}

\begin{prop}\label{prop:cri}
$H^{\bt}(\Gc/P;\RR) \to H^{\bt}(G/P;\RR)$ is trivial on positive degrees if and only if the Euler class $e$ of $\varphi$ is nontrivial in $H^{q+1}(\Gc/G;\RR)$.
\end{prop}

\begin{proof}
From the Gysin sequence of $\varphi$, we get an exact sequence
\[
\xymatrix{ H^{q}(\Gc/P;\RR) \ar[r]^{\fint_{\varphi}} & H^{0}(\Gc/G;\RR) \ar[r]^<<<<<{\wedge e} & H^{q+1}(\Gc/G;\RR)\;.}
\]
Thus $e$ is nontrivial if and only if the image of $\fint_{\varphi}$ is trivial. In turn, the image of $\fint_{\varphi}$ is trivial if and only if the restriction map $H^{q}(\Gc/P) \to H^{q}(G/P)$ is trivial.
\end{proof}

\subsection{The case of transversely projective foliations of odd codimension}\label{sec:oddtrproj}

In this case, $(G,G/P) = (\SL(q+1;\RR), S^{q})$ for odd $q$. Let $q = 2k-1$ and $Y_{\ell} = \SU(\ell)/\SO(\ell)$. Now, the sphere bundle~\eqref{eq:Kseq} is 
\begin{equation}\label{eq:Kseq1}
\xymatrix{\SO(2k)/\SO(2k-1) \ar[r] & \SU(2k)/\SO(2k-1) \ar[r]^<<<<<{p_{k}} & Y_{2k}\;.}
\end{equation}
We show that the nontriviality of the Euler class of~\eqref{eq:Kseq1} follows from the Borel's computation of the Betti numbers of homogeneous spaces~\cite{Borel1953}.

\begin{lemma}
The Euler class of~\eqref{eq:Kseq1} is nontrivial in $H^{2k}(Y_{2k})$.
\end{lemma}

\begin{proof}
According to the computation of $H^{\bt}(Y_{\ell})$ by Borel~\cite[Proposition~31.4]{Borel1953}, we get that
\begin{equation}\label{eq:surj}
H^{\bt}(Y_{2k}) \longrightarrow H^{\bt}(Y_{2k-1}) 
\end{equation}
is surjective and
\begin{equation}\label{eq:Y}
\dim H^{\bt}(Y_{2k}) = 2 \dim H^{\bt}(Y_{2k-1})\;.
\end{equation}
Consider also the fibration
\[
\xymatrix{Y_{2k-1} \ar[r]^<<<<<{\iota} & \SU(2k)/\SO(2k-1) \ar[r] & \SU(2k)/\SU(2k-1) \approx S^{4k-1}\;.}
\]
Assume that the Euler class of $p_{k}$ is trivial. Then
\begin{equation}\label{eq:Betti1}
\dim H^{\bt}(\SU(2k)/\SO(2k-1)) = \dim H^{\bt}(S^{2k-1})\cdot \dim H^{\bt}(Y_{2k})\;;
\end{equation}
in particular, $p_{k}^{*} : H^{\bt}(Y_{2k}) \to H^{\bt}(\SU(2k)/\SO(2k-1))$ is injective. By the surjectivity of~\eqref{eq:surj}, we get the surjectivity of $\iota^{*} : H^{\bt}(\SU(2k)/\SO(2k-1)) \to H^{\bt}(Y_{2k-1})$. Thus, by the Leray-Hirsch theorem, we obtain
\begin{equation}\label{eq:Betti2}
\dim H^{\bt}(\SU(2k)/\SO(2k-1)) = \dim H^{\bt}(Y_{2k-1})\cdot \dim H^{\bt}(S^{4k-1})\;.
\end{equation}
But~\eqref{eq:Betti1} and~\eqref{eq:Betti2} contradict~\eqref{eq:Y}. Thus the Euler class of $p_{k}$ is nontrivial.
\end{proof}

So $H^{\bt}(K_{\Gc}/K_{P};\RR) \to H^{\bt}(K_{G}/K_{P};\RR)$ is trivial on positive degrees. Thus Theorem~\ref{thm:finite} gives an alternative proof of Theorem~\ref{thm:BGH} for the case of odd codimension.

\subsection{The case of transversely conformally flat foliations}

Now, let $(G,G/P) = (\SO_{0}(n+1,1),S^{n}_{\infty})$. So $\Gc = \SO(n+2;\CC)$, and
\[
K_{\Gc}  = \SO(n+2)\;, \quad K_{G}  = \SO(n+1) \;, \quad K_{P}  =  \SO(n)\;.
\]
Thus the sphere bundle~\eqref{eq:Kseq} is 
\[
\xymatrix{ \SO(n+1) /\SO(n) \ar[r] & \SO(n+2)/\SO(n)  \ar[d]^<<<<<{\zeta_{\SO}} \\
 & \SO(n+2)/\SO(n+1) \;.}
\]
The isotropy group of the $\SO(n+2)$-action on the unit tangent sphere bundle of $\SO(n+2)/\SO(n+1)$ is $\SO(n)$. So $\zeta_{\SO}$ is the unit tangent sphere bundle of $\SO(n+2)/\SO(n+1) \cong S^{n+1}$. Hence the Euler class of $\zeta_{\SO}$ is equal to the fundamental class of $S^{n+1}$ if $n$ is odd. Thus, by Proposition~\ref{prop:cri}, the assumption of Theorem~\ref{thm:finite} is satisfied in this case. 



\subsection{The case of transversely spherical CR foliations}

Now, let $(G,G/P) = (\SU(n+1,1),S^{2n+1}_{\infty})$, where the codimension $q=2n+1$ is odd. In this case, $\Gc = \SL(n+2;\CC)$ and
\[
K_{\Gc}  = \SU(n+2)\;, \quad
K_{G}  = \S(\U(n+1) \U(1))\;, \quad
K_{P}  = \S(\U(n) \U(1))\;.
\]
Thus the sphere bundle~\eqref{eq:Kseq} is 
\begin{equation*}
\xymatrix{ \S(\U(n+1) \U(1))/\S(\U(n) \U(1)) \ar[r] & \SU(n+2)/\S(\U(n) \U(1)) \ar[d]^<<<<<{\zeta_{\SU}} \\
 & \SU(n+2)/\S(\U(n+1) \U(1))\;.}
\end{equation*}
The isotropy group of the $\SU(n+2)$-action on the unit tangent sphere bundle of $\SU(n+2)/\S(\U(n+1) \U(1))$ is $\S(\U(n) \U(1))$. So $\zeta_{\SU}$ is the unit tangent sphere bundle of $\SU(n+2)/\S(\U(n+1) \U(1)) \cong \CC P^{n+1}$. Thus the Euler class of $\zeta_{\SU}$ is equal to $n+2$ times the fundamental class of $\CC P^{n+1}$. By Proposition~\ref{prop:cri}, the assumption of Theorem~\ref{thm:finite} is satisfied in this case. 

\subsection{The case $(G,G/P) = (\Sp(n+1,1),S^{4n+3}_{\infty})$}

Note that the codimension is always odd in this case. We get $\Gc = \Sp(n+2;\CC)$ and
\[
K_{\Gc}  = \Sp(n+2)\;, \quad
K_{G}  = \Sp(n+1) \Sp(1)\;, \quad
K_{P}  = \Sp(n) \Sp(1)\;.
\]
Thus the sphere bundle~\eqref{eq:Kseq} is 
\[
\xymatrix{ \Sp(n+1) \Sp(1)/\Sp(n) \Sp(1) \ar[r] & \Sp(n+2)/\Sp(n) \Sp(1) \ar[d]^<<<<<{\zeta_{\Sp}} \\
& \Sp(n+2)/\Sp(n+1) \Sp(1)\;.}
\]
The isotropy group of the $\Sp(n+2)$-action on the unit tangent sphere bundle of $\Sp(n+2)/\Sp(n+1) \Sp(1)$ is $\Sp(n)$. Thus $\zeta_{\Sp}$ is the unit tangent sphere bundle of $\Sp(n+2)/\Sp(n+1) \Sp(1) \cong \HH P^{n+1}$. Hence the Euler class of $\zeta_{\Sp}$ is equal to $n+2$ times the fundamental class of $\HH P^{n+1}$. By Proposition~\ref{prop:cri}, the assumption of Theorem~\ref{thm:finite} is satisfied in this case. 

\subsection{The case $(G,G/P) = (F_{4(-20)},S^{15}_{\infty})$}\label{sec:F4}

We recall the explicit presentation of $F_{4(-20)}$, $F_{4}$ and $F_{4}^{\CC}$ as automorphism groups of Jordan algebras due to Freudenthal~\cite{Freudenthal1985} and Yokota~\cite{Yokota1975}. We follow Yokota~\cite{Yokota2009}. Let $\OO$ be the Cayley algebra over $\RR$. Let $M(3;\OO)$ be the $3 \times 3$ matrix group with coefficients in $\OO$. Let $X^{*} = {}^{t}\overline{X}$, where the bar denotes conjugation in $\OO$. Let $I''_{1} = \left(\begin{smallmatrix} -1 & 0 & 0 \\ 0 & 1 & 0 \\ 0 & 0 & 1 \end{smallmatrix}\right)$, 
\begin{align*}
\J(1,2) & = \{\, X \in M(3;\OO) \mid I''_{1}X^{*}I''_{1}  = X \,\}\;, \\
\J & = \{\, X \in M(3;\OO) \mid X^{*}  = X \,\}\;, 
\end{align*}
and $\J^{\CC} = \J \otimes \CC$. A product $\circ$ is defined on these $\RR$-vector spaces by $X \circ Y = \frac{1}{2}(XY + YX)$. Endowed with this product, $\J(1,2)$, $\J$ and $\J^{\CC}$ are called {\em Jordan algebras}. $\J$ can be written as follows:
\[
\J = \left\{\, \begin{pmatrix} \xi_{1} & x_{3} & \overline{x}_{2} \\
                                     \overline{x}_{3} & \xi_{2} & x_{1} \\
                                     x_{2} & \overline{x}_{1} & \xi_{3} \end{pmatrix} \in M(3;\OO) \, \Bigg| \, \xi_{i} \in \RR,\ x_{i} \in \OO \,\right\}\;.
\]
Here, $F_{4(-20)}$, $F_{4}$ and $F_{4}^{\CC}$ are defined as the automorphism groups of these Jordan algebras:
\begin{align*}
F_{4(-20)} & = \{\, \sigma \in {\Aut}_{\RR}(\J(1,2)) \mid \sigma(x \circ y) = \sigma(x) \circ \sigma(y) \,\}\;, \\
F_{4} & = \{\, \sigma \in {\Aut}_{\RR}(\J) \mid \sigma(x \circ y) = \sigma(x) \circ \sigma(y) \,\}\;, \\
F_{4}^{\CC} & = \{\, \sigma \in {\Aut}_{\CC}(\J^{\CC}) \mid \sigma(x \circ y) = \sigma(x) \circ \sigma(y) \,\}\;.
\end{align*}
It is well known that $\Gc = F_{4}^{\CC}$ and $K_{\Gc}  = F_{4}$. We will get an explicit form of the parabolic subgroup $P$.

\begin{lemma}[Announced by Borel~\cite{Borel1950} and proved by Matsushima~\cite{Matsushima1952}]\label{lem:isot}
The isotropy group of  the $F_{4}$-action on $\J$ at $E_{11} = \left(\begin{smallmatrix} 1 & 0 & 0 \\ 0 & 0 & 0 \\ 0 & 0 & 0 \end{smallmatrix}\right)$ is $\Spin(9)$. Thus the orbit of $E_{11}$ under the $F_{4}$-action is the octonionic projective plane $\OO P^{2}=F_{4}/\Spin(9)$.
\end{lemma}
Here, $\OO P^{2}$ is given by the following formula~\cite{Yokota1975}:
\[
\OO P^{2} = \{\, X \in M(3;\OO) \mid X^{2}=X,\ \tr X = 1 \,\}\;.
\]
There is a left $G$-action on $\OO P^{2}$ defined by $(g,X) \mapsto \frac{gX}{\tr (gX)}$. The orbit of $E_{11}$ under this $G$-action is the octonionic hyperbolic plane $\mathbf{H}_{\OO}^{2}=F_{4(-20)}/\Spin(9)$, and the boundary $\partial \mathbf{H}_{\OO}^{2}$ of $\mathbf{H}_{\OO}^{2}$ in $\OO P^{2}$ is given by 
\[
\partial \mathbf{H}_{\OO}^{2} = \{\, X \in \OO P^{2} \mid \tr (X \circ I''_{1}X)=0 \,\}\;.
\]
Since $\OO P^{2}$ consists of the matrices
\[
  X=
    \begin{pmatrix} 
      \xi_{1} & x_{3} & \overline{x}_{2} \\
      \overline{x}_{3} & \xi_{2} & x_{1} \\
      x_{2} & \overline{x}_{1} & \xi_{3} 
    \end{pmatrix}
  \in\J 
\]
such that
\begin{gather*}
  \begin{alignedat}{3}
    \xi_{2}\xi_{3} &= |x_{1}|^{2}\;,\quad \xi_{3}\xi_{1} &= |x_{2}|^{2}\;,\quad \xi_{1}\xi_{2} &= |x_{3}|^{2}\;, \\ 
    x_{2}x_{3} &= \xi_{1}\overline{x}_{1}\;,\quad x_{3}x_{1} &= \xi_{2}\overline{x}_{2}\;,\quad 
    x_{1}x_{2} &= \xi_{3}\overline{x}_{3}\;,
  \end{alignedat} \\
  \xi_{1} + \xi_{2} + \xi_{3} = 1\;,
\end{gather*}
a simple calculation shows that $\tr (X \circ I''_{1}X)=0$ is equivalent to $\xi_{1}=\frac{1}{2}$ for points $X\in\OO P^{2}$ as above, obtaining a diffeomorphism 
\[
\partial \mathbf{H}_{\OO}^{2}\approx\{\, (x_{2},x_{3}) \in \OO^{2} \mid |x_{2}|^{2} + |x_{3}|^{2} = 1/4 \,\}\;;
\]
in particular, $\partial \mathbf{H}_{\OO}^{2} \approx S^{15}$. Then $P$ is the isotropy group of the $G$-action on $\partial \mathbf{H}_{\OO}^{2}$ at $X_{0} = \left(\begin{smallmatrix} 1/2 & 0 & 1/2 \\ 0 & 0 & 0 \\ 1/2 & 0 & 1/2 \end{smallmatrix}\right)$.

We determine the sphere bundle~\eqref{eq:Kseq} in this case. Let $K_{G}$ denote the isotropy group of the $F_{4}$-action at $E_{11}= \left(\begin{smallmatrix} 1 & 0 & 0 \\ 0 & 0 & 0 \\ 0 & 0 & 0 \end{smallmatrix}\right)$, which is a maximal compact subgroup of $G$ isomorphic to $\Spin(9)$ by Lemma~\ref{lem:isot}. A maximal compact subgroup $K_{P}$ of $P$ is given by $K_{P} = K_{G} \cap P$. Since $X_{0} + \frac{1}{2}E_{33} = (X_{0} - \frac{1}{2}E_{11}) \circ (X_{0} - \frac{1}{2}E_{11})$,  we get, for any $\sigma \in K_{P}$, 
\begin{multline*} \textstyle
\sigma (X_{0} + \frac{1}{2}E_{33}) = \sigma \left( (X_{0} - \frac{1}{2}E_{11}) \circ (X_{0} - \frac{1}{2}E_{11}) \right) = \\
\textstyle \sigma (X_{0} - \frac{1}{2}E_{11}) \circ \sigma (X_{0} - \frac{1}{2}E_{11}) = (X_{0} - \frac{1}{2}E_{11}) \circ (X_{0} - \frac{1}{2}E_{11}) = X_{0} + \frac{1}{2}E_{33},
\end{multline*}
which implies that $\sigma$ fixes $E_{33}$. Since the $F_{4}$-action on $\J$ fixes the identity matrix~\cite[Lemma~2.2.4]{Yokota2009}, \cite[Lemma~2.3-(1)]{Yokota1975}, $K_{P}$ is equal to the isotropy group of the $\Spin(9)$-action on $\J$ at $\left(\begin{smallmatrix} 0 & 0 & 0 \\ 0 & 0 & 1 \\ 0 & 1 & 0 \end{smallmatrix}\right)$, which is isomorphic to $\Spin(7)$~\cite[Proof of Theorem~2.7.5]{Yokota2009}, \cite[Remark~6.3]{Yokota1975}. Thus the sphere bundle~\eqref{eq:Kseq} is
\begin{equation*}
\xymatrix{ S^{15} \approx \Spin(9)/\Spin(7) \ar[r] & F_{4}/\Spin(7) \ar[r]^{\zeta_{F_{4}}} & F_{4}/\Spin(9)\;.}
\end{equation*}
We will show the following.

\begin{lemma}\label{lem:sb}
$\zeta_{F_{4}}$ is diffeomorphic to the unit tangent sphere bundle of $F_{4}/\Spin(9)$.
\end{lemma}

The orbit $\mathcal{K}$ of $E_{11}$ under the $F_{4}$-action on $\J$ is $\OO P^{2} = F_{4}/\Spin(9)$ by Lemma~\ref{lem:isot}. Let us describe the tangent space $T_{E_{11}} \mathcal{K}$ of $\mathcal{K}$ at $E_{11}$.

\begin{lemma}\label{lem:tang}
We have 
\begin{equation}\label{eq:tangsp}
T_{E_{11}} \mathcal{K} = \left\{\, \begin{pmatrix} 0 & x_{3} & \overline{x}_{2} \\
 \overline{x}_{3} & 0 & 0 \\
 x_{2} & 0 & 0 \end{pmatrix} \in M(3;\OO) \; \Bigg| \; x_{2}, x_{3} \in \OO \,\right\}\;.
\end{equation}
\end{lemma}

\begin{proof}
Let $\mathfrak{f}_{4} = \Lie(F_{4})$. Consider the infinitesimal $\mathfrak{f}_{4}$-action $\rho : \mathfrak{f}_{4} \to T_{E_{11}} \mathcal{K}$ at $E_{11}$. We get $\rho(\mathfrak{f}_{4}) = T_{E_{11}} \mathcal{K}$ by definition. Let $\sigma = \left(\begin{smallmatrix} 1 & 0 & 0 \\ 0 & -1 & 0 \\ 0 & 0 & -1 \end{smallmatrix}\right)$. Since $\sigma^{2} = 1$, we obtain an involution $\sigma : \mathfrak{f}_{4} \to \mathfrak{f}_{4}$ given by $\sigma(X) = \sigma X \sigma$. Then we get a decomposition $\mathfrak{f}_{4} = (\mathfrak{f}_{4})_{\sigma} \oplus (\mathfrak{f}_{4})_{-\sigma}$, where $(\mathfrak{f}_{4})_{\sigma}$ is the $\sigma$-invariant part and $(\mathfrak{f}_{4})_{-\sigma}$ is the $\sigma$-antiinvariant part. By~\cite[Theorem~2.9.1]{Yokota2009}, \cite[Theorem~2.4.4]{Yokota1990}, we get $\Spin(9) = (F_{4})^{\sigma}$. By Lemma~\ref{lem:isot}, it follows that $\
 \rho((\mathfrak{f}_{4})_{\sigma}) = 0$. On the other hand, for $X\in(\mathfrak{f}_{4})_{-\sigma}$, we get $\sigma(X)E_{11}=\sigma X \sigma E_{11} = -E_{11}$. Thus $\rho(\mathfrak{f}_{4}) = T_{E_{11}} \mathcal{K}$ is contained in the $\sigma$-antiinvariant part $(\J)_{-\sigma}$ of $\J$. Since it is easy to see that $(\J)_{-\sigma}$ is equal to the right hand side of~\eqref{eq:tangsp} and $\dim (\J)_{-\sigma} = \dim \mathcal{K}$, we get the equality~\eqref{eq:tangsp}.
\end{proof}

We saw that $K_{P}$ is the isotropy group of the adjoint $K_{G}$-action on
\[
  (\J)_{-\sigma} = \left\{\, 
    \begin{pmatrix} 
      0 & x_{3} & \overline{x}_{2} \\ 
      \overline{x}_{3} & 0 & 0 \\ 
      x_{2} & 0 & 0 
    \end{pmatrix} 
  \, \Bigg| \; x_{2}, x_{3} \in \OO \,\right\}
\]
at $\left(\begin{smallmatrix} 0 & 0 & 1 \\ 0 & 0 & 0 \\ 1 & 0 & 0 \end{smallmatrix}\right)$. Thus Lemma~\ref{lem:tang} implies that $K_{P}$ is the isotropy group of the $K_{G}$-action on $T_{E_{11}} \mathcal{K}$. This proves Lemma~\ref{lem:sb}. Hence, according to~\cite{Hirsch1949} or~\cite{Yokota1955}, the Euler class of $\zeta_{F_{4}}$ is equal to $3$ times the fundamental class of $\OO P^{2}$ by the cell decomposition of $\mathbb{O}P^{2}$. So the assumption of Theorem~\ref{thm:finite} is satisfied in this case. 

\subsection{A remark on the center}\label{sec:center}
The $G$-actions on $G/P$ are not effective for some of the pairs $(G,G/P)$ considered in Corollary~\ref{cor:main}. In fact, in the case where $G$ is either $(\SU(n+1,1),S^{2n+1}_{\infty})$ or $(\Sp(n+1,1),S^{4n+3}_{\infty})$ for even $n$, the stabilizers of the $G$-action on $G/P$ are given by $\{ \, cI_{n+2}\mid c \in \CC^{\times}, c^{n+2} = 1 \}$ and $\{\pm I_{n+2}\}$, respectively, where they are equal to the centers $\Z(G)$ of $G$. In the other cases considered in Corollary~\ref{cor:main}, the $G$-actions on $G/P$ are effective. The quotient of $\SU(n+1,1)$ and $\Sp(n+1,1)$ by the centers are denoted by $\PSU(n+1,1)$ and $\PSp(n+1,1)$.

The finiteness of $\Sigma(\PSU(n+1,1),S^{2n+1}_{\infty})$ and $\Sigma(\PSp(n+1,1),S^{4n+3}_{\infty})$ is proved like in the cases $\Sigma(\SU(n+1,1),S^{2n+1}_{\infty})$ and $\Sigma(\Sp(n+1,1),S^{4n+3}_{\infty})$ of Theorem~\ref{thm:finite}. We only need to notice the following two facts. By the discreteness of $\Z(G)$, there is no difference when we consider their Lie algebras. Since $\Z(G)$ is contained in $\Z(\Gc)$ and $K_{P}$ in both cases, the canonical embedding $G/K_{P} \to \Gc/(K_{P})_{\CC}$ is not changed by taking quotient by $\Z(G)$.

\section{Bott-Thurston-Heitsch type formulas}\label{sec:ex}
\addtocontents{toc}{\protect\setcounter{tocdepth}{2}}

\subsection{Pittie's Bott connections}\label{sec:root}

The purpose of Section~\ref{sec:ex} is to prove Bott-Thurston-Heitsch type formulas (Theorem~\ref{thm:BT}). Section~\ref{sec:root} is devoted to recall the Pittie's construction of a Bott connection for the $P/K_{P}$-coset foliation $\F_{P}$ of $G/K_{P}$, where $G$ is semisimple and $P$ is parabolic. It will be used to calculate the Godbillon-Vey class of $\F_{P}$ in Lie algebra cohomology in Section~\ref{sec:Liecoh}. Since $(G,G/P)$-foliations are classified by $\F_{P}$ in the sense of Proposition~\ref{prop:h2}-\eqref{i: hol and widehat dev}, this computation can be applied to $(G,G/P)$-foliations (Section~\ref{sec:BT}). By using the computation in Section~\ref{sec:Liecoh}, we will also show that the Godbillon-Vey class is the essentially unique nontrivial secondary characteristic class for $(G,G/P)$-foliations in Section~\ref{sec:GVspan}.

Take a Cartan subgroup $H$ and a Borel subgroup $B$ of $G$ which contains $H$. By a classical result of Borel and Tits~\cite[Section 4]{BorelTits1965}, any parabolic subgroup of $G$ is conjugated to a standard one, which contains $B$. Thus we may assume that $P$ contains $B$. First we recall the decompositions of the semisimple $\gc$ and parabolic $\pc$ in this case. Let $\h=\Lie(H)$ and
\[
\gc = \h \oplus \bigoplus_{\alpha \in \Upsilon} (\gc)_{\alpha}
\]
be the root-space decomposition of $\gc$, where $\Upsilon$ is the set of roots. Fix a set $\Pi$ of simple roots which additively generate $\Upsilon$, and let $\Upsilon^{+}$ be the set of corresponding positive roots. Since $\pc$ contains $\Lie(B)=\bigoplus_{\alpha \in \Upsilon^{+}}(\gc)_{\alpha}$, there exists a subset $\Phi$ of $\Upsilon^{+}$ such that 
\begin{equation}\label{eq:Levi}
\pc =  \bigoplus_{\alpha \in -\Phi} (\gc)_{\alpha} \oplus \h \oplus \bigoplus_{\alpha \in \Upsilon^{+}} (\gc)_{\alpha}\;.
\end{equation}
Thus, with 
\[
\r = \hspace{-6pt} \bigoplus_{\alpha \in \Phi \cup (-\Phi)} (\gc)_{\alpha} \oplus \h\;, \quad \u = \hspace{-6pt} \bigoplus_{\alpha \in \Upsilon^{+} \setminus \Phi} (\gc)_{\alpha}\;, \quad \v = \hspace{-6pt} \bigoplus_{\alpha \in \Upsilon^{+} \setminus \Phi} (\gc)_{-\alpha}\;,
\]
we get a decomposition
\begin{equation}\label{eq:Lang}
\gc = \pc + \v = \r + \u + \v\;.
\end{equation}
Here, $\r$ is a reductive subalgebra of $\gc$ called the {\em Levi part\/} of $\pc$. Note that $\u$ and $\v$ are $\ad \r$-invariant and nilpotent. 

Let $\widehat{\F}_{\Pc}$ the right $\Pc$-coset foliation of $\Gc$. Left invariant complex connections on the normal bundle $\nu \widehat{\F}_{\Pc}$ of $\widehat{\F}_{\Pc}$ are in one-to-one correspondence with $\CC$-linear maps $\gc \to \gl(\gc/\pc;\CC)$. Let $\sigma : \gc \to \pc$ be the projection with respect to the decomposition~\eqref{eq:Lang}. Consider the connection $\wt{\nabla}^{\CC}$ on $\nu \widehat{\F}_{\Pc}$ determined by
\begin{equation}\label{eq:Pittieconn}
\wt{\nabla}^{\CC}_{X} Y = \pi \big([({\id}_{\g} - \sigma \pi)X, \sigma(Y)]\big)
\end{equation}
for $X\in\gc$ and $Y\in\gc/\pc$. The connection form $\Theta$ of $\wt{\nabla}^{\CC}$ is regarded as an element of $\gc^{*} \otimes \gl(\gc/\pc;\CC)$. Pittie observed that, if we identify $\gc/\pc$ to $\v$ via the canonical projection, then the connection form $\Theta$ of the connection given by~\eqref{eq:Pittieconn} is the Maurer-Cartan form of the adjoint action of $\pc$ on $\v$, which is given by
\begin{equation}\label{eq:Pittie1}
\theta_{ij}(X) = \eta_{i}([X,Y_{j}])  
\end{equation}
for $1 \leq i \leq q$, $1 \leq j \leq q$, where $\{Y_{j}\}$ is a basis of $\v$ and $\{\eta_{j}\}$ is the basis of $\v^{*}$ dual to $\{Y_{j}\}$. Let $p_{\u^{*} \wedge \v^{*}}$ denote the composite 
\[
\xymatrix{\bigwedge^{2} \gc^{*} = \bigwedge^{2} \pc^{*} \oplus \pc^{*} \wedge \v^{*} \oplus \bigwedge^{2} \v^{*} \ar[r] & \pc^{*} \wedge \v^{*} \ar[r] & \u^{*} \wedge \v^{*}} 
\]
of the projections with respect to the decompositions~\eqref{eq:Levi} and~\eqref{eq:Lang}. Let us denote the composite 
\begin{equation}\label{eq:widehatd}
\xymatrix{\gc^{*} \ar[r]^<<<<<{d} & \bigwedge^{2} \gc^{*} \ar[r]^{p_{\u^{*} \wedge \v^{*}}} & \u^{*} \wedge \v^{*}} 
\end{equation}
by $\hat{d}$. The curvature form $\Omega$ of $\Theta$ is the element of $\bigwedge^{2} \gc^{*} \otimes \gl(\gc/\pc;\CC)$ given by $\Omega = d\Theta - \Theta \wedge \Theta$. We will use the following observation of Pittie. 

\begin{prop}[{\cite[Proposition~2.1]{Pittie1979}}]\label{prop:PittieOmega}
$\hat{d}\Theta = \Omega$.
\end{prop}

This formula is a consequence of the $(\ad \r)$-invariance of $\u$ and $\v$.

Let $\Upsilon^{+} \setminus \Phi = \{\alpha_{i}\}_{1 \leq i \leq q}$. We will also use the following direct consequence of the formula~\eqref{eq:Pittie1}, observed by Pittie. 

\begin{prop}[{\cite[Theorem~2.3]{Pittie1979}}]\label{prop:Pittieh1}
$\Delta_{\F_{P}}(h_{1}) = -\frac{1}{2\pi}\sum_{i = 1}^{q} \alpha_{i}$.
\end{prop}

Pittie observed that $-\Delta_{\F_{P}}(c_{1})$ is a K\"{a}hler form of $\Gc/\Pc$ under the identification of $\bigwedge\u^{*} \otimes \bigwedge\v^{*}$ with the left invariant de Rham complex of $\Gc/\Pc$ in a standard way. By using the Lefschetz decomposition of the cohomology of K\"{a}hler manifolds, Pittie showed the following.

\begin{thm}[{\cite[Theorem~3.1]{Pittie1979}}]\label{thm:Pittie}
$\Delta_{\F_{P}}(H^{\bt}(WO_{q}))$ is linearly spanned by the Pontryagin classes and 
$\big\{\, \Delta_{\F_{P}}(h_{1}h_{I}c_{1}^{q}) \, \mid \, I \subseteq \{3,5, \ldots, [q] \} \,\big\}$.
\end{thm}

\subsection{Computation in Lie algebra cohomology}\label{sec:Liecoh}

\subsubsection{The case $(G,G/P) = (\SL(q+1;\RR),S^{q})$}\label{sec:trproj}

In this case, $\pc$ and $\v$ are the subalgebras of $\gc = \sl( q+1 ;\CC)$ consisting of the matrices of the form
\[
  \begin{pmatrix} 
    * & * & \cdots & * \\
    0 & * & \cdots & * \\
    \vdots & \vdots & \ddots & \vdots \\
    0 & * & \cdots & * 
  \end{pmatrix}
  \quad\text{and}\quad
  \begin{pmatrix} 
    0 & 0 & \cdots & 0 \\
    * & 0 & \cdots & 0 \\
    \vdots & \vdots & \ddots & \vdots \\
    * & 0 & \cdots & 0 
  \end{pmatrix}\;,
\]
respectively. Let $E_{ij}$ be the element of $\gc$ with $1$ at the $(i,j)$-th entry and $0$ at the other entries. Let $E^{\vee}_{ij}$ be the dual of $E_{ij}$. In this case, $\{E_{1j}\}_{2 \leq j \leq  q+1} $ is a basis of $\v$. Let $\Theta = (\theta_{ij})_{2 \leq i,j \leq  q+1} $ be the matrix presentation of $\Theta$ with respect to $\{E_{1j}\}_{2 \leq j \leq  q+1} $. From $[E_{kh},E_{lj}] = \delta_{hl} E_{kj} - \delta_{jk} E_{lh}$ and $E^{\vee}_{ji}(E_{kh}) = \delta_{jk} \delta_{ih}$, we get
\[
\theta_{ij}(E_{kh}) = E^{\vee}_{1i}([E_{kh},E_{1j}]) = \delta_{h1} \delta_{1k} \delta_{ij} - \delta_{jk} \delta_{ih}\;.
\]
Then
	\[
		\Theta = (\theta_{ij})_{2 \leq i,j \leq  q+1}= 
			\begin{pmatrix} 
				E^{\vee}_{22} - E^{\vee}_{11} & E^{\vee}_{32}  & \cdots & E^{\vee}_{q+1\,  2} \\
                              	E^{\vee}_{23} & E^{\vee}_{33} - E^{\vee}_{11} & \cdots & E^{\vee}_{q+1\,  3} \\
                              	\vdots & \vdots & \ddots & \vdots \\
                              	E^{\vee}_{2\, q+1}  & E^{\vee}_{3\, q+1}  & \cdots & E^{\vee}_{q+1\, q+1}  - E^{\vee}_{11} 
         			\end{pmatrix}\;.
	\]
By observing that $\sum_{i=1}^{q+1 }  E^{\vee}_{ii} = 0$ on $\gc^{*}$, we get
\begin{gather*}
\Delta_{\F_{P}}(h_{1}) = \frac{1}{2\pi} \tr \Theta = \frac{1}{2\pi} \sum_{i=2}^{q+1}  (E^{\vee}_{ii} - E^{\vee}_{11}) = - \frac{q+1 } {2\pi} E^{\vee}_{11}\;,\\
\Delta_{\F_{P}}(c_{1}) = d\Delta_{\F_{P}}(h_{1}) = \frac{q+1}{2\pi} \sum_{k=2}^{q+1}  E^{\vee}_{1k} \wedge E^{\vee}_{k1}\;.
\end{gather*}
Note that $\Theta$ equals the Maurer-Cartan form $\Theta_{MC} = (E^{\vee}_{ij})_{2 \leq i,j \leq  q+1} $ of $\sl(q;\CC)$ modulo $\Delta_{\F_{P}}(h_{1})$. Thus
\begin{align}
\GV(\F_{P}) = (2\pi)^{q+1} \Delta_{\F_{P}}(h_{1}c_{1}^{q}) & = -(q+1)^{q+1} (q+1)! \, E^{\vee}_{11} \wedge \bigwedge_{k=2}^{q+1}  E^{\vee}_{1k} \wedge E^{\vee}_{k1} \;,\label{eq:GVtrproj1}\\
 \Delta_{\F_{P}}(h_{1}h_{I}c_{1}^{q})  & = \frac{1}{(2\pi)^{q+1}} \GV(\F_{P})\,h_{I}(\Theta_{MC})\;. \label{eq:GVtrproj2}
\end{align}
We will use these formulas to give an alternative proof of Theorem~\ref{thm:BGH}. Heitsch obtained more general formulas of this type for secondary characteristic classes of the form $h_{I}c_{J}$ by the application of his residues formulas~\cite[Theorem~4.2]{Heitsch1978}, \cite[Theorem~2.3]{Heitsch1983}.

\subsubsection{The case $(G,G/P) = (\SO_{0}(n+1,1),S^{n}_{\infty})$}\label{sec:trconflie}

Yamato~\cite{Yamato1975} also computed the characteristic classes of this case in a different way. Let $n'=n+1$ and $n''=n+2$. Let
\begin{equation}\label{eq:Iq2}
I'_{n''} =  \begin{pmatrix} 0  & 0 & -1 \\
                                  0 & I_{n} & 0 \\
                                  -1 & 0 & 0 \end{pmatrix} \in \gl(n'';\RR)\;, 
\end{equation}
where $I_{n}$ is the $n \times n$ identity matrix. We use the following description of $\g=\so(n',1)$: 
\begin{align*}
\g & = \left\{\, X \in \gl(n'';\RR) \mid \text{}^{t}X I'_{n''} +  I'_{n''} X = 0 \,\right\} \\
 & = \left\{\, \begin{pmatrix} a  & u & 0 \\
                              {}^{t}v & A & {}^{t}u \\
                              0 & v & -a \end{pmatrix} \in \gl(n'';\RR) \, \Bigg| \, a \in \RR,\ A \in \so(n;\RR),\ u,v \in \RR^{n} \,\right\}\;.
\end{align*}
Since $e_{1}=\left(\begin{smallmatrix} 1 \\ 0 \\ \vdots \\ 0 \end{smallmatrix}\right)$ is a vector in the light cone, we get
\begin{align*}
\p & = \left\{\, X \in \g \mid \exists a \in \RR\ \text{so that}\ X e_{1} = ae_{1} \,\right\} \\
 & = \left\{\, \begin{pmatrix} a  & u & 0 \\
                              0 & A & {}^{t}u \\
                              0 & 0 & -a \end{pmatrix} \in \gl(n'';\RR) \, \Bigg| \, a \in \RR,\ A \in \so(n;\RR),\ u \in \RR^{n} \,\right\}\;.
\end{align*}
Then
\begin{align*}
\gc &= \left\{\, \begin{pmatrix} a  & u & 0 \\
                              {}^{t}v & A & {}^{t}u \\
                              0 & v & -a \end{pmatrix} \in \gl(n'';\CC) \, \Bigg| \, a \in \CC,\ A \in \so(n;\CC),\ u,v \in \CC^{n} \,\right\}\;,\\
\pc &= \left\{\, \begin{pmatrix} a  & u & 0 \\
                              0 & A & {}^{t}u \\
                              0 & 0 & -a \end{pmatrix} \in \gl(n'';\CC) \, \Bigg| \, a \in \CC,\ A \in \so(n;\CC),\ u \in \CC^{n} \,\right\}\;,\\
\v &= \left\{\, \begin{pmatrix} 0 & 0 & 0 \\
                              {}^{t}v & 0 & 0 \\
                              0 & v & 0 \end{pmatrix} \in \gl(n'';\CC) \, \Bigg| \, v \in \CC^{n} \,\right\}\;.
\end{align*}
Let 
\[
a = E_{11} - E_{n''n''}\;, \quad v_{j} = E_{j1} + E_{n''j}\;, \quad A_{kh} = E_{kh} -E_{hk}\;.
\]
Then we get a basis 
\begin{equation}\label{eq:basis}
\{v_{j}\}_{2 \leq j \leq n'} \cup \{a\} \cup \{ {}^{t}v_{j}\}_{2 \leq j \leq n'} \cup \{A_{kh}\}_{2 \leq k < h \leq n'}
\end{equation}
of $\gc$. Here $\{v_{j}\}_{2 \leq j \leq n'}$ is a basis of $\v$ and $\{a\} \cup \{ {}^{t}v_{j}\}_{2 \leq j \leq n'} \cup \{A_{kh}\}_{2 \leq k < h \leq n'}$ is a basis of $\pc$. We get
\[
[ a, v_{j}] = -v_{j}\;,\quad [ \text{}^{t}v_{i}, v_{j}] = \delta_{ij} a\;,\quad
[ A_{kh}, v_{j}] = \delta_{jh} v_{k} - \delta_{jk} v_{h}\;.
\]
For $z\in\gc$, let $z^{\vee}\in \gc^{*}$ denote the dual of $z$ with respect to the basis~\eqref{eq:basis}. Since $\theta_{ij}(X) = v^{\vee}_{i}([X, v_{j}])$ for $X\in\pc$, it follows that
  \[
    \theta_{ij}(a) = -\delta_{ij}\;, \quad \theta_{ij}( {}^{t}v_{l}) = 0\;, \quad
    \theta_{ij}(A_{kh}) = \delta_{jh} \delta_{ik} - \delta_{jk} \delta_{ih}\;.
  \]
Thus 
\begin{equation*}
\Theta = (\theta_{ij}) = 
\begin{pmatrix} -a^{\vee} & A^{\vee}_{32} & \cdots & A^{\vee}_{n' 2} \\
                    A^{\vee}_{23}     & -a^{\vee} &  & \vdots \\
                    \vdots       &      &  \ddots & A^{\vee}_{n' n} \\
                    A^{\vee}_{2 n'}              & \cdots      & A^{\vee}_{n n'}          & -a^{\vee} \end{pmatrix}\;.
\end{equation*}
Since $\hat{d}(a^{\vee}) = -\sum_{k=2}^{n} {}^{t}v_{k}^{\vee} \wedge v^{\vee}_{k}$ and $\hat{d}A^{\vee}_{kh} = 0$ (see~\eqref{eq:widehatd} for the definition of $\hat{d}$), Proposition~\ref{prop:PittieOmega} implies
\begin{equation*}
\Omega = 
\begin{pmatrix} \sum_{k=2}^{n'} {}^{t}v_{k}^{\vee} \wedge v^{\vee}_{k}  & 0 & \cdots & 0 \\
                    0     & \sum_{k=2}^{n'} {}^{t}v_{k}^{\vee} \wedge v^{\vee}_{k} &  & \vdots \\
                    \vdots       &       &  \ddots & 0 \\
                    0              & \cdots      & 0          & \sum_{k=2}^{n'} {}^{t}v_{k}^{\vee} \wedge v^{\vee}_{k} \end{pmatrix}\;.
\end{equation*}
We get
\begin{gather}
\Delta_{\F_{P}}(h_{1}) =  -n a^{\vee}\;, \label{eq:SOh1}\\
\GV(\F_{P}) = (2\pi)^{n+1} \Delta_{\F_{P}}(h_{1}c_{1}^{n}) = -n^{n+1} n!\, a^{\vee} \wedge \bigwedge_{k=2}^{n+1} {}^{t}v_{k}^{\vee} \wedge v^{\vee}_{k}\;.\label{eq:GVtrconf2}
\end{gather}
Later, in Proposition~\ref{prop:onlyGV3}, we will show that any other nontrivial secondary characteristic class is a multiple of the Godbillon-Vey class by using~\eqref{eq:GVtrconf2}. 

Let $K_{G}$ be the Lie subgroup of $G$ consisting of all orthogonal matrices in $G$. Here $K_{G}$ is a maximal compact subgroup of $G$. Let $\k_{G} = \Lie(K_{G})$. To be used in the proof of Theorem~\ref{thm:BT} later, we will compute the integration of $\GV(\F_{P})$ along the fiber of $G/K_{P} \to G/K_{G}$. Letting
\[
Q_{1} = (-1)^{n+1} a^{\vee} \wedge \bigwedge_{k=2}^{n+1} ({}^{t}v_{k}^{\vee} + v^{\vee}_{k})\;, \quad
Q_{2} = \bigwedge_{k=2}^{n+1} (v^{\vee}_{k} - {}^{t}v_{k}^{\vee})\;,
\]
we have the following direct consequence of~\eqref{eq:GVtrconf2}:
\begin{equation}\label{eq:GVtrconf3} 
\GV(\F_{P}) = \frac{(-1)^{\frac{n(n+1)}{2}}n^{n+1} n!}{2^{n}}\, Q_{1} \wedge Q_{2}\;.
\end{equation}
To derive~\eqref{eq:GVtrconf3} from~\eqref{eq:GVtrconf2}, we note that
\begin{equation}\label{eq:sign}
\sign \begin{pmatrix} 1 & 2 & 3 & \cdots & n     & n+1 & n+2 & \cdots & 2n-1 & 2n \\
                           1 & 3 & 5 & \cdots & 2n-1 & 2 & 4 & \cdots & 2n-2 & 2n \end{pmatrix} 
                            = (-1)^{\frac{n(n-1)}{2}}\;.
\end{equation}
Note that $Q_{1}$ is a wedge product of symmetric matrices and $Q_{2}$ is a wedge product of anti-symmetric matrices. By
\begin{align*}
\k_{G} & = \{\, A \in \gl (n+2;\RR) \mid A \in \g, A = -{}^{t}A \,\} \notag \\ 
         & = \left\{\, \begin{pmatrix} 0  & -v & 0 \\
                              {}^{t}v & A & -{}^{t}v \\
                              0 & v & 0 \end{pmatrix} \in \gl(n+2;\RR) \, \Bigg| \, A \in \so(n;\RR),\ v \in \RR^{n} \,\right\}\;, \label{eq:kg} 
\end{align*}
it is easy to see that $Q_{1}$ is $K_{G}$-basic; namely, since $\deg Q_{1}= \dim G/K_{G}$, $Q_{1}$ is the pull-back of a volume form on $G/K_{G}$ by the projection $\phi_{K_{G}} : G/K_{P} \to G/K_{G}$. Since $Q_{1} \wedge Q_{2}$ is a volume form on $G/K_{P}$, by using the exact sequence 
\[
\xymatrix{ 0 \ar[r] & \k_{G}/\k_{P} \ar[r] & \g/\k_{P} \ar[r] & \g/\k_{G} \ar[r] & 0}\;,
\]
we see that $Q_{2}$ restricted to each fiber of $\phi_{K_{G}}$ is a volume form. We orient the fibers of $\phi_{K_{G}}$ with this volume form. We will prove \eqref{eq:lat} by computing the integration $\int_{K_{G}/K_{P}} Q_{2}$ and the norm of $Q_{1}$ with respect to the metric on $\g/\k_{G}$ induced from the Killing form of $\g$.

First we compute the norm of $Q_{1}$. By identifying $\g/\k_{G}$ with the space of symmetric matrices in $\g$, the norm $\|\cdot\|$. on $\g/\k_{G}$ induced from the Killing form is given by $\|X\| = \sqrt{n \tr (X^{2})}$. Then we have
\begin{alignat*}{2}
\|a\| & =\sqrt{2n}\;, &\quad \|{}^{t}v_{k} + v_{k}\| & =2\sqrt{n}\;. 
\intertext{It follows that}
\|a^{\vee}\| & =\frac{1}{\sqrt{2n}}\;, &\quad \|{}^{t}v_{k}^{\vee} + v^{\vee}_{k}\| & =\frac{1}{\sqrt{n}}\;,
\end{alignat*}
where the norm on $(\g/\k_{G})^{*}$ induced by the dual metric is denoted by the same symbol $\|\cdot\|$. Thus, letting $\omega_{G/K_{G}}$ be the volume form on $G/K_{G}$  of norm one defining the same orientation as $Q_{1}$, we get 
\begin{equation}\label{eq:intt}
Q_{1} = \frac{1}{\sqrt{2}\,n^{\frac{n+1}{2}}}\, \phi_{K_{G}}^{*} \omega_{G/K_{G}}\;.
\end{equation}
Let $\k_{G}=\k_{P} \oplus \m$ be the orthogonal decomposition with respect to the Killing form. Then $\m$ is realized as 
\[
\m = \left\{\, \begin{pmatrix} 0 & -v & 0 \\ {}^{t}v & 0 & -{}^{t}v \\ 0 & v & 0 \end{pmatrix} \in \gl(n+2;\RR) \, \Bigg| \, v \in \RR^{n} \,\right\}\;.
\]
Here $K_{G}/K_{P} \approx S^{n}$, and $\m$ is naturally identified with the tangent space of $K_{G}/K_{P}$ at $eK_{P}$. Since $\left(\begin{smallmatrix} 0 & -v & 0 \\ {}^{t}v & 0 & -{}^{t}v \\ 0 & v & 0 \end{smallmatrix}\right)$ in $\m$ is conjugated to $\left(\begin{smallmatrix} 0 & 0 & 0 \\ 0 & 0 & -\sqrt{2}{}^{t}v \\ 0 & \sqrt{2}v & 0 \end{smallmatrix}\right)$ by $\left(\begin{smallmatrix} 1/\sqrt{2} & 0 & 1/\sqrt{2} \\ 0 & I_{n} & 0 \\ -1/\sqrt{2} & 0 & 1/\sqrt{2} \end{smallmatrix}\right)$, we can compare $Q_{2}$ with the volume form of norm one on the unit sphere of $\RR^{n+1}$. Then we have 
\[
\int_{K_{G}/K_{P}} Q_{2} =2^{n/2}\vol(S^{n})\;.
\]
 Then, combining~\eqref{eq:GVtrconf3} and~\eqref{eq:intt} with this computation, we get
\begin{equation}\label{eq:inttt}
\fint_{\phi_{K_{G}}} \GV(\F_{P}) = (-1)^{\frac{n(n+1)}{2}}\frac{n^{\frac{n+1}{2}}\vol (S^{n})}{2^{\frac{n+1}{2}}}\, \omega_{G/K_{G}}\;.
\end{equation}

\subsubsection{The case $(G,G/P) = (\SU(n+1,1),S^{2n+1}_{\infty})$}\label{sec:trCRlie}

Let $n'=n+1$ and $n''=n+2$. Let $I'_{n''}$ be the matrix given by~\eqref{eq:Iq2}. We use the following description of $\g=\su(n',1)$:
\begin{align}
  \g & = \left\{\, X \in \sl(n'';\CC) \, \mid \, {}^{t}\overline{X} I'_{n''} +  I'_{n''} X = 0 \,\right\} \label{eq:susl} \\
  & = \left\{\, 
    \begin{pmatrix} 
      a  & u &  \sqrt{-1}c \\
      {}^{t}\overline{v} & A & {}^{t}\overline{u} \\
      \sqrt{-1}g & v & -\overline{a} 
    \end{pmatrix} 
  \in \sl(n'';\CC) \, \Bigg| \, 
    \begin{matrix}
      a \in \CC,\ c,g \in \RR,\\
      A \in \u(n),\ u,v \in \CC^{n}
    \end{matrix}
  \,\right\}\;. \notag
\end{align}
Since $e_{1}=\left(\begin{smallmatrix} 1 \\ 0 \\ \vdots \\ 0 \end{smallmatrix}\right)$ is a vector in the light cone, we get
\begin{align*}
  \p & = \left\{\, X \in \g \mid \exists a \in \CC\ \text{so that}\ Xe_{1} = ae_{1}\, \right\} \\
  & = \Biggl\{\, 
    \begin{pmatrix} 
      a  & u &  \sqrt{-1}c \\
      0 & A & {}^{t}\overline{u} \\
      0 & 0 & -\overline{a} 
    \end{pmatrix} 
  \in \sl(n'';\CC) \, \Bigg| \,
    \begin{matrix}
      a \in \CC,\ c \in \RR,\\
      A \in \u(n),\ u \in \CC^{n}
    \end{matrix}
  \,\Biggr\}\;.
\end{align*}
Then $\gc = \sl(n'';\CC)$, and
\begin{align*}
  \pc & = \left\{\, 
    \begin{pmatrix} 
      a_{1}  & u_{1} &  c \\
      0 & A & {}^{t}u_{2} \\
      0 & 0 & a_{2} 
    \end{pmatrix} 
  \in \sl(n'';\CC) \, \Bigg| \, 
    \begin{matrix}
      a_{1}, a_{2}, c \in \CC,\ u_{1}, u_{2} \in \CC^{n},\\
      A \in \gl(n;\CC)
    \end{matrix}
  \,\right\}\;,\\
\v & = \left\{\, \begin{pmatrix} 0 & 0 & 0 \\
                              {}^{t}v_{1} & 0 & 0 \\
                              g & v_{2} & 0 \end{pmatrix} \in \sl(n'';\CC) \, \Bigg| \, g \in \CC,\ v_{1}, v_{2} \in \CC^{n} \,\right\}\;. 
\end{align*}
We can compute $\Theta$ and $\Omega$ like in the last case. But here we compute only the Godbillon-Vey class of $\F_{P}$. By using the computation, we will see that any other nontrivial secondary characteristic classes are multiples of the Godbillon-Vey class (Proposition~\ref{prop:onlyGV3}). We will apply Proposition~\ref{prop:Pittieh1} to compute $\Delta_{\F_{P}}(h_{1})$. As a Cartan subalgebra $\h$, we take the Lie subalgebra of $\gc$ consisting of diagonal matrices. As a basis of $\v$ consisting of root vectors, we can take $\{E_{k1}\}_{2 \leq k \leq n'} \cup \{E_{n''1}\} \cup \{E_{n''k}\}_{2 \leq k \leq n'}$. For a root vector $z\in\gc$, let $z^{\vee}\in \gc^{*}$ be the element such that $z^{\vee}(z)=1$ and $z^{\vee}(z')=0$ for any $z' \in \h$ and any root vector $z'$ which is linearly independent of $z$. The root of $E_{ij}$ is given by $E^{\vee}_{ii} - E^{\vee}_{jj}$. Thus Proposition~\ref{prop:Pittieh1} implies
\begin{align}
\Delta_{\F_{P}}(h_{1}) &= \frac{1}{2\pi}\left( E^{\vee}_{n''n''} - E^{\vee}_{11} + \sum_{k = 2}^{n'} (E^{\vee}_{kk} - E^{\vee}_{11}) + \sum_{k = 2}^{n'} (E^{\vee}_{n''n''} - E^{\vee}_{kk})\right) \notag \\
&= -\frac{n'}{2\pi}(E^{\vee}_{11} - E^{\vee}_{n''n''})\;. \label{eq:SUh1}
\end{align}
So
\[
\Delta_{\F_{P}}(c_{1}) = d\big(\Delta_{\F_{P}}(h_{1})\big) =  \frac{n'}{2\pi}\left(2E^{\vee}_{1n''} \wedge E^{\vee}_{n''1} + \sum_{k=2}^{n'} E^{\vee}_{1k} \wedge E^{\vee}_{k1} + \sum_{k=2}^{n'} E^{\vee}_{kn''} \wedge E^{\vee}_{n''k}\right)\;.
\]
Thus we get the following formula:
\begin{multline}\label{eq:GVSU2}
\GV(\F_{P}) = (2\pi)^{2n+2} \Delta_{\F_{P}}(h_{1}c_{1}^{2n+1})= -2 (n')^{2n+2} (2n+1)! \\
\times(E^{\vee}_{11} - E^{\vee}_{n''n''}) \wedge \bigwedge_{k=2}^{n''} (E^{\vee}_{1k} \wedge E^{\vee}_{k1}) \wedge \bigwedge_{k=2}^{n'} (E^{\vee}_{kn''} \wedge E^{\vee}_{n''k})\;.
\end{multline}
Later, in Proposition~\ref{prop:onlyGV3}, we will show that any other nontrivial secondary characteristic class of transversely spherical CR foliations is a multiple of the Godbillon-Vey class by using~\eqref{eq:GVSU2}. 

Let $K_{G}$ be the Lie subgroup of $G$ consisting of unitary matrices, which is a maximal compact subgroup of $G$. Let $\k_{G} = \Lie(K_{G})$. To be used later for the proof of Theorem~\ref{thm:BT}, we will compute the integration of $\GV(\F_{P})$ along the fiber of the canonical projection $\phi_{K_{G}} : G/K_{P} \to G/K_{G}$. Letting 
\begin{multline*}
L_{1} = (E^{\vee}_{11} - E^{\vee}_{n''n''}) 
\wedge \sqrt{-1} (E^{\vee}_{n''1} - E^{\vee}_{1n''}) \wedge \text{} \\ 
\bigwedge_{k=2}^{n'} \sqrt{-1} (E_{k1}^{\vee} - E_{1k}^{\vee} + E_{n''k}^{\vee} - E_{kn''}^{\vee}) \wedge \bigwedge_{k=2}^{n'} (E_{1k}^{\vee} + E_{k1}^{\vee} + E_{n''k}^{\vee} + E_{kn''}^{\vee})\;, 
\end{multline*}
\begin{multline*}
L_{2} =  \sqrt{-1} (E^{\vee}_{n''1} + E^{\vee}_{1n''}) \wedge \text{} \\ 
\bigwedge_{k=2}^{n'} (E_{k1}^{\vee} - E_{1k}^{\vee} - E_{n''k}^{\vee} + E_{kn''}^{\vee}) \wedge \bigwedge_{k=2}^{n'} \sqrt{-1} (E_{1k}^{\vee} + E_{k1}^{\vee} - E_{n''k}^{\vee} - E_{kn''}^{\vee})\;,
\end{multline*}
we have the following direct consequence of~\eqref{eq:GVSU2}: 
\begin{equation}\label{eq:GVSU3}
\GV(\F_{P})  = (-1)^{n+1}\frac{(n')^{2n+2} (2n+1)!}{2^{4n}}\, L_{1} \wedge L_{2}\;.
\end{equation}
Here $L_{1}$ is a wedge product of Hermitian matrices and $L_{2}$ is a wedge product of skew-Hermitian matrices. Since $\k_{G} = \{\, A \in \g \mid A = -{}^{t}\overline{A} \,\}$, it is easy to see that $L_{1}$ is $K_{G}$-basic, and hence $L_{1}$ is the pull-back of a volume form on $G/K_{G}$ by the projection $\phi_{K_{G}} : G/K_{P} \to G/K_{G}$. Identifying $\g/\k_{G}$ with the space of Hermitian matrices in $\g$, the norm $\|\cdot\|$ on $\g/\k_{G}$ induced from the Killing form of $\g$ is given by $\|X\|=\sqrt{2(n+2)\tr (X^{2})}$. Thus, denoting the norm on $(\g/\k_{G})^{*}$ of the dual metric by the same symbol $\|\cdot\|$, we have that 
\begin{align*}
\|E_{11}^{\vee} - E_{n''n''}^{\vee}\| = \|\sqrt{-1}(E_{1n''}^{\vee} - E_{n''1}^{\vee})\| & = 1/\sqrt{n+2}\;, \\
\|\sqrt{-1}(E_{1k}^{\vee} - E_{k1}^{\vee} + E_{n''k}^{\vee} - E_{kn''}^{\vee})\| & = \sqrt{2/(n+2)}\;, \\
\|E_{1k}^{\vee} + E_{k1}^{\vee} + E_{n''k}^{\vee} + E_{kn''}^{\vee}\| & = \sqrt{2/(n+2)}\;.
\end{align*}
Thus, letting $\omega_{G/K_{G}}$ be the oriented volume form on $G/K_{G}$ of norm one defining the same orientation as $L_{1}$, we get
\begin{equation}\label{eq:SUcg0}
L_{1} = \frac{2^{n}}{(n+2)^{n+1}}\, \phi_{K_{G}}^{*} \omega_{G/K_{G}}\;.
\end{equation}
Since $L_{1} \wedge L_{2}$ is a volume form on $G/K_{P}$, the restriction of $L_{2}$ to each fiber of $\phi_{K_{G}}$ is a volume form, which is used to orient it. Let $\k_{G}=\k_{P} \oplus \m$ be the orthogonal decomposition with respect to the Killing form. Here $\m$ is realized as 
\[
\m = \left\{\, \begin{pmatrix} 0 & -v & \sqrt{-1}g \\ {}^{t}\overline{v} & 0 & -{}^{t}\overline{v} \\ \sqrt{-1}g & v & 0 \end{pmatrix} \in \gl(n'';\CC) \, \Bigg| \, g \in \RR, v \in \CC^{n} \,\right\}\;,
\]
identified with the tangent space of $K_{G}/K_{P}$ at $eK_{P}$. Since $\left(\begin{smallmatrix} 
0 & -v & \sqrt{-1}g \\ 
{}^{t}\overline{v} & 0 & -{}^{t}\overline{v} \\ 
\sqrt{-1}g & v & 0 \end{smallmatrix}
\right)$ is conjugated to $\left(\begin{smallmatrix} 
\sqrt{-1}g & 0 & 0 \\ 
0 & 0 & -\sqrt{2}{}^{t}\overline{v} \\ 
0 & \sqrt{2}v & -\sqrt{-1}g 
\end{smallmatrix}\right)$ by $\left(\begin{smallmatrix} 1/\sqrt{2} & 0 & 1/\sqrt{2} \\ 0 & I_{n} & 0 \\ -1/\sqrt{2} & 0 & 1/\sqrt{2} \end{smallmatrix}\right)$ , we may compare $L_{2}$ with the volume form of norm one on the unit sphere of $\CC^{n+1}$ to get
\begin{equation}\label{eq:SUcg1}
\int_{K_{G}/K_{P}} L_{2} = 2^{3n+1} \vol(S^{2n+1})\;.
\end{equation}
By~\eqref{eq:GVSU3},~\eqref{eq:SUcg0} and~\eqref{eq:SUcg1}, we obtain
\begin{equation}\label{eq:SUcg}
\fint_{\phi_{K_{G}}} \GV(\F_{P}) =  (-1)^{n+1}\frac{2(n')^{2n+2} (2n+1)!\vol (S^{2n+1})}{(n+2)^{n+1}} \, \omega_{G/K_{G}}\;.
\end{equation}

\subsubsection{The case $(G,G/P) = (\Sp(n+1,1),S^{4n+3}_{\infty})$}\label{sec:trSplie}

Let $n'=n+1$ and $n''=n+2$. Let
\begin{equation*}
J' =  \begin{pmatrix} 0      & I_{n''}' \\
                           -I_{n''}' & 0    \end{pmatrix}\;,
\end{equation*}
where $I'_{n''}$ is the matrix given by~\eqref{eq:Iq2}. We use the following description of $\g=\sp(n',1)$:
\begin{equation}\label{eq:spsl}
  \g = \left\{\, X = 
    \begin{pmatrix} 
      Z_{1} & Z_{2} \\
      -\overline{Z}_{2} & \overline{Z}_{1}
    \end{pmatrix} 
  \in \gl(2n'';\CC) \, \Bigg| \, 
    \begin{matrix}
      Z_{1},\ Z_{2} \in \gl(n'';\CC),\\
      {}^{t}X J' + J'X = 0
    \end{matrix}
  \,\right\}\;.
\end{equation}
Note that, for a matrix of the form $X = 
    \left(\begin{smallmatrix} 
      Z_{1} & Z_{2} \\
      -\overline{Z}_{2} & \overline{Z}_{1}
    \end{smallmatrix}\right)$, where $Z_{1}$, $Z_{2}\in \gl(n'';\CC)$, the condition ${}^{t}X J' + J'X = 0$ is equivalent to ${}^{t}\overline{X} J'' + J''X = 0$, where $J'' =  \left(\begin{smallmatrix} I_{n''}' & 0 \\
                            0 & I_{n''}' \end{smallmatrix}\right)$. Hence the above presentation of $\sp(n',1)$ is equivalent to the well known presentation. Here,
\[
\p = \{\, X \in \g \mid \exists s, t \in \CC\ \text{so that}\ Xe_{1} = se_{1} + te_{n''+1} \,\}\;,
\]
where $e_{i}$ is the $i$-th standard unit vector of $\CC^{2n''}$. Thus $\p$ consists of the matrices of the form: 
\begin{equation*}
 \begin{pmatrix} a  & b & \sqrt{-1}\,c & d & f & g \\
                        0   & A & {}^{t}\overline{b}  &  0  & B  & -{}^{t}f \\
                          0 & 0  & -\overline{a}  & 0 & 0 & d \\
                          -\overline{d}  & -\overline{f} & -\overline{g} & \overline{a}  & \overline{b} & -\sqrt{-1}\,c \\
                          0 & -\overline{B}   & {}^{t}\overline{f}     & 0      & \overline{A} & {}^{t}b \\
                          0 & 0    & -\overline{d}             & 0       & 0      & -a 
\end{pmatrix}\;,
\end{equation*}
where $c \in \RR$, $a,d,g \in \CC$, $b,f \in \CC^{n}$, $A \in \u(n)$, and $B \in \gl(n;\CC)$ with $B={}^{t}B$. We get
\[
\gc  = \sp (n'';\CC) = \left\{\, X \in \gl(2n'';\CC) \, \mid \, {}^{t}X J' + J'X = 0 \,\right\} \;,
\]
which consists of the matrices of the form
\[
X = \begin{pmatrix} Z_{1} & Z_{2} \\
                                          Z_{3} & Z_{4}    \end{pmatrix} \in \gl(2n'';\CC)
\]
such that
\[
{}^{t}Z_{1} I'_{n''} + I'_{n''} Z_{4}=
- {}^{t}Z_{3} I'_{n''} + I'_{n''} Z_{3}=
- {}^{t}Z_{2} I'_{n''} + I'_{n''} Z_{2}=0\;. 
\]
Then $\pc$ is the subalgebra of $\gc$ consisting of the matrices 
\begin{equation*}
 \begin{pmatrix} a_{1}  & b_{1} & c & d_{1}  & f_{1} & g_{1} \\
                         0  & A & {}^{t}b_{2} &  0 & B_{1}  & -{}^{t}f_{1} \\
                          0 & 0  & a_{2} & 0  & 0  & d_{1} \\
                          d_{2} & f_{2}  & g_{2} & -a_{2}  & b_{2} & -c \\
                          0 & B_{2}  & -{}^{t}f_{2} & 0 & -{}^{t}A & {}^{t}b_{1} \\
                          0 & 0    & d_{2} & 0  & 0      & -a_{1} 
\end{pmatrix}\;,
\end{equation*}
where $a_{1},a_{2},c,d_{1},d_{2},g_{1},g_{2} \in \CC$, $b_{1},b_{2},f_{1},f_{2} \in \CC^{n}$, $A \in \u(n)$, and $B_{1},B_{2} \in  \gl(n;\CC)$ with $B_{1}={}^{t}B_{1}$ and $B_{2}={}^{t}B_{2}$. Let $\h$ be a Cartan subalgebra of $\gc$ with the following basis:
\begin{equation}\label{eq:bas}
\{E_{11} - E_{2n''\, 2n''}\} \cup \{E_{kk} - E_{n''+k\, n''+k}\}_{2 \leq k \leq n'} \cup \{E_{n''\,n''} - E_{n''+1\, n''+1}\}\;. 
\end{equation} 
Thus $\v$ consists of the matrices
\begin{equation*}
 \begin{pmatrix}  0  & 0 & 0  & 0 & 0 & 0 \\
                       {}^{t}u_{1} & 0 & 0  & -{}^{t}x_{1} & 0 & 0 \\
                       v & u_{2} & 0  & y_{1} & x_{1} & 0 \\
                       0 & 0 & 0 & 0 & 0 & 0 \\
                       -{}^{t}x_{2} & 0 & 0 & {}^{t}u_{2} & 0 & 0 \\
                       y_{2} & x_{2} & 0 & -v & u_{1} & 0  
\end{pmatrix}\;,
\end{equation*}
where $v,y_{1},y_{2} \in \CC$ and $u_{1},u_{2},x_{1},x_{2} \in \CC^{n}$.

Here, we compute the Godbillon-Vey class like in the last example by using Proposition~\ref{prop:Pittieh1}. Later, by using the computation, we will see that any other nontrivial secondary characteristic class is a multiple of the Godbillon-Vey class (see Proposition~\ref{prop:onlyGV3}). As a basis of $\v$ consisting of root vectors, take
\begin{alignat*}{2}
u_{1,k} & = E_{k\, 1} + E_{2n''\, n''+k}\;, \quad &2 \leq k \leq n'\;, \\
u_{2,k} & = E_{n''\, k} + E_{n''+k\, n''+1}\;, \quad &2 \leq k \leq n'\;, \\
v & = E_{n''\, 1} - E_{2n''\, n''+1}\;, &\\
x_{1,k} & = -E_{k\, n''+1} + E_{n''\, n''+k}\;, \quad &2 \leq k \leq n'\;, \\
y_{1} & = E_{n''\, n''+1}\;, & \\
x_{2,k} & = -E_{n''+k\, 1} + E_{2n''\, k}\;, \quad &2 \leq k \leq n'\;, \\
y_{2} & = E_{2n''\, 1}\;. &
\end{alignat*}
Let $\{\gamma_{i}\}_{1 \leq i \leq n''}$ be the basis of $\h^{*}$ dual to~\eqref{eq:bas}. The roots corresponding to these vectors are given in Table~\ref{table: roots}. Thus, by Proposition~\ref{prop:Pittieh1},
\begin{equation}\label{eq:h1Sp}
\Delta_{\F_{P}}(h_{1}) = -\frac{2n+3}{2\pi}\,(\gamma_{1} - \gamma_{n''})\;.
\end{equation}

\begin{table}[h]
\[
\begin{array}{|c|c|c|c|c|c|c|} \hline
u_{1,k} & u_{2,k}  & v & x_{1,k} & y_{1} & x_{2,k} & y_{2} \\ \hline
-\gamma_{1} + \gamma_{k} & -\gamma_{k} + \gamma_{n''} & -\gamma_{1} + \gamma_{n''} & \gamma_{k} + \gamma_{n''} & 2\gamma_{n''} & -\gamma_{1}-\gamma_{k} & -2\gamma_{1} \\ \hline
\end{array}
\]
\caption{Roots of root vectors in $\v$, where $2 \leq k \leq n'$.}
\label{table: roots}
\end{table}

For a root vector $z\in\gc$, let $z^{\vee}\in\gc^{*}$ be determined by $z^{\vee}(z)=1$ and $z^{\vee}(z')=0$ for any $z' \in \h$ and any root vector $z'$ which is linearly independent of $z$. We have
\begin{align*}
\hat{d}\gamma_{1} & =  -\sum_{k=2}^{n'} ({}^{t}u_{1,k}^{\vee} \wedge u_{1,k}^{\vee}) - ({}^{t}v^{\vee} \wedge v^{\vee}) -\sum_{k=2}^{n'} ({}^{t}x_{2,k}^{\vee} \wedge x_{2,k}^{\vee}) - ({}^{t}y_{2}^{\vee} \wedge y_{2}^{\vee})\;, \\
\hat{d}\gamma_{n''} & = ({}^{t}v^{\vee} \wedge v^{\vee}) + \sum_{k=2}^{n'} ({}^{t}u_{2,k}^{\vee} \wedge u_{2,k}^{\vee}) + ({}^{t}y_{1}^{\vee} \wedge y_{1}^{\vee}) +\sum_{k=2}^{n'} ({}^{t}x_{1,k}^{\vee} \wedge x_{1,k}^{\vee})
\end{align*}
(see~\eqref{eq:widehatd} for the definition of $\hat{d}$). Let $\zeta$ be a symplectic form on $\u \oplus \v$ defined by
\begin{align*}
\zeta & = \sum_{k=2}^{n'} ({}^{t}u_{1,k}^{\vee} \wedge u_{1,k}^{\vee}) + 2({}^{t}v^{\vee} \wedge v^{\vee}) + \sum_{k=2}^{n'} ({}^{t}x_{2,k}^{\vee} \wedge x_{2,k}^{\vee}) + ({}^{t}y_{2}^{\vee} \wedge y_{2}^{\vee}) \\
&\phantom{=\text{}\text{} + {}^{t}v^{\vee}} + \sum_{k=2}^{n'} ({}^{t}u_{2,k}^{\vee} \wedge u_{2,k}^{\vee}) + ({}^{t}y_{1}^{\vee} \wedge y_{1}^{\vee}) + \sum_{k=2}^{n'} ({}^{t}x_{1,k}^{\vee} \wedge x_{1,k}^{\vee})\;.
\end{align*}
Then
\begin{equation}\label{eq:c1Sp}
\Delta_{\F_{P}}(c_{1}) = d\big(\Delta_{\F_{P}}(h_{1})\big) = \frac{2n+3}{2\pi}\,\zeta\;.
\end{equation}
By~\eqref{eq:h1Sp} and~\eqref{eq:c1Sp}, we obtain the following formula of the Godbillon-Vey class:
\begin{equation}\label{eq:GVSp2}
\GV(\F_{P}) = (2\pi)^{4n+4} \Delta_{\F_{P}}(h_{1}c_{1}^{4n+3}) = -(2n+3)^{4n+4} \,(\gamma_{1} - \gamma_{n''})\wedge \zeta^{4n+3}\;.
\end{equation}
Later, in Proposition~\ref{prop:onlyGV3}, we will show that any other nontrivial secondary characteristic class of $(\Sp(n+1,1),S^{4n+3}_{\infty})$-foliations is a multiple of the Godbillon-Vey class by using~\eqref{eq:GVSp2}. To be used below, we also state the following direct consequence of~\eqref{eq:GVSp2}:

\begin{equation}\label{eq:GVSp3}
\GV(\F_{P}) = \frac{(2n+3)^{4n+4}(4n+3)!}{2^{8n+4}}\, W_{1} \wedge W_{2}\;,
\end{equation}
where 
\begin{align}
W_{1} &= (\gamma_{1} - \gamma_{n''}) \wedge \sqrt{-1} (v^{\vee} - {}^{t}v^{\vee}) \notag \\
& \phantom{=\text{}}\text{} \wedge (y^{\vee}_{1} + {}^{t}y^{\vee}_{1} - y^{\vee}_{2} - {}^{t}y^{\vee}_{2}) \wedge \sqrt{-1}  (y^{\vee}_{1} - {}^{t}y^{\vee}_{1} + y^{\vee}_{2} - {}^{t}y^{\vee}_{2}) \notag \\
& \phantom{=\text{}}\text{} \wedge \bigwedge_{k=2}^{n'} \sqrt{-1} (u^{\vee}_{1,k} - {}^{t}u^{\vee}_{1,k} - u^{\vee}_{2,k} + {}^{t}u^{\vee}_{2,k}) \wedge \bigwedge_{k=2}^{n'} (u^{\vee}_{1,k} + {}^{t}u^{\vee}_{1,k} + u^{\vee}_{2,k} + {}^{t}u^{\vee}_{2,k}) \notag \\
& \phantom{=\text{}}\wedge \bigwedge_{k=2}^{n'} (x^{\vee}_{1,k} + {}^{t}x^{\vee}_{1,k} - x^{\vee}_{2,k} - {}^{t}x^{\vee}_{2,k}) \wedge \bigwedge_{k=2}^{n'} (x^{\vee}_{1,k} + {}^{t}x^{\vee}_{1,k} + x^{\vee}_{2,k} + {}^{t}x^{\vee}_{2,k})\;, \label{eq:Omega1} \\
W_{2} &= \sqrt{-1} (v^{\vee} + {}^{t}v^{\vee}) \wedge (y^{\vee}_{1} - {}^{t}y^{\vee}_{1} - y^{\vee}_{2} + {}^{t}y^{\vee}_{2}) \wedge \sqrt{-1}  (y^{\vee}_{1} + {}^{t}y^{\vee}_{1} + y^{\vee}_{2} + {}^{t}y^{\vee}_{2}) \notag \\
& \phantom{=\text{}}\text{} \wedge \bigwedge_{k=2}^{n'} \sqrt{-1} (u^{\vee}_{1,k} + {}^{t}u^{\vee}_{1,k} - u^{\vee}_{2,k} - {}^{t}u^{\vee}_{2,k}) \wedge \bigwedge_{k=2}^{n'} (u^{\vee}_{1,k} - {}^{t}u^{\vee}_{1,k} + u^{\vee}_{2,k} - {}^{t}u^{\vee}_{2,k}) \notag \\
& \phantom{=\text{}}\text{}\wedge \bigwedge_{k=2}^{n'} (x^{\vee}_{1,k} - {}^{t}x^{\vee}_{1,k} + x^{\vee}_{2,k} - {}^{t}x^{\vee}_{2,k}) \wedge \bigwedge_{k=2}^{n'} (x^{\vee}_{1,k} - {}^{t}x^{\vee}_{1,k} - x^{\vee}_{2,k} + {}^{t}x^{\vee}_{2,k})\;. \label{eq:Omega2}
\end{align}

Let $K_{G}$ be the Lie subgroup of $G$ consisting of unitary matrices, which is a maximal compact subgroup of $G$. Let $\k_{G} = \Lie(K_{G})$. To be used later for the proof of Theorem~\ref{thm:BT}, we will compute the integration of $\GV(\F_{P})$ along the fibers of the canonical projection $\phi_{K_{G}} : G/K_{P} \to G/K_{G}$. Here $W_{1}$ is a wedge product of Hermitian matrices and $W_{2}$ is a wedge product of skew-Hermitian matrices. By $\k_{G} = \{\, A \in \g \mid A = -{}^{t}\overline{A} \,\}$, it is easy to see that $W_{1}$ is $K_{G}$-basic, and hence $W_{1}$ is the pull-back of a volume form on $G/K_{G}$ by the projection $\phi_{K_{G}} : G/K_{P} \to G/K_{G}$. Identifying $\g/\k_{G}$ with the space of Hermitian matrices in $\g$, the norm $\|\cdot\|$ of $\g/\k_{G}$ induced by the Killing form of $\g$ is given by $\|X\|=\sqrt{2(n+3) \tr (X^{2})}$. Thus, denoting the norm on $(\g/\k_{G})^{*}$ induced by the dual metric by the same symbol $\|\cdot\|$, we have that
\begin{align*}
\|\gamma_{1} - \gamma_{n''}\| = \|\sqrt{-1}(v^{\vee}-{}^{t}v^{\vee})\| & = 1/\sqrt{2(n+3)}\;, \\
\|y_{1}^{\vee} \pm {}^{t}y_{1}^{\vee} \pm y_{2}^{\vee} \pm {}^{t}y_{2}^{\vee}\| & = \sqrt{2/(n+3)}\;, \\
\|z_{1}^{\vee} + {}^{t}z_{1}^{\vee} + z_{2}^{\vee} + {}^{t}z_{2}^{\vee}\| & = 1/\sqrt{n+3}\;,
\end{align*}
for any
\begin{multline*}
(z_{1},z_{2})\in\left\{(\pm u_{1,k}, \pm u_{2,k}), \left(\pm \sqrt{-1}\,u_{1,k}, \pm \sqrt{-1}\,u_{2,k}\right), \right.\\
\left.(\pm x_{1,k}, \pm x_{2,k}), \left(\pm \sqrt{-1}\,x_{1,k}, \pm \sqrt{-1}\,x_{2,k}\right)\right\}_{2 \leq k \leq n'}
\end{multline*}
such that $z_{1}^{\vee} + {}^{t}z_{1}^{\vee} + z_{2}^{\vee} + {}^{t}z_{2}^{\vee}$ is Hermitian. Thus, letting $\omega_{G/K_{G}}$ be the oriented volume form on $G/K_{G}$ of norm $1$, we get
\begin{equation}\label{eq:Spcg0}
W_{1}  = \frac{1}{(n+3)^{2n+2}} \,\phi_{K_{G}}^{*} \omega_{G/K_{G}}\;.
\end{equation}
Consider the decomposition $\k_{G}=\k_{P} \oplus \m$. We have
\begin{equation} 
\m = \left\{\,  \left.\begin{pmatrix} 
0  & -u & \sqrt{-1}\,v & 0 & \overline{x} & -\overline{y} \\
{}^{t}\overline{u}  & 0 & -{}^{t}\overline{u}  & {}^{t}\overline{x} & 0 & -{}^{t}\overline{x}  \\
\sqrt{-1}\,v & u  & 0  & -\overline{y} & -\overline{x} & 0 \\
0  & -x & y & 0  & -\overline{u} & -\sqrt{-1}\,v  \\
-{}^{t}x & 0  & {}^{t}x  & {}^{t}u  & 0 & -{}^{t}u \\
y & x  & 0 & -\sqrt{-1}\,v & \overline{u} & 0 
\end{pmatrix} \, \right| \, 
\begin{matrix}
v \in \RR,\ y \in \CC, \\
u, x \in \CC^{n}
\end{matrix}
 \,\right\}
\end{equation}
in $\gl(2n'';\CC)$. Since the matrices 
\[
\begin{pmatrix} 
0  & -u & \sqrt{-1}\,v & 0 & \overline{x} & -\overline{y} \\
{}^{t}\overline{u}  & 0 & -{}^{t}\overline{u}  & {}^{t}\overline{x} & 0 & -{}^{t}\overline{x}  \\
\sqrt{-1}\,v & u  & 0  & -\overline{y} & -\overline{x} & 0 \\
0  & -x & y & 0  & -\overline{u} & -\sqrt{-1}\,v  \\
-{}^{t}x & 0  & {}^{t}x  & {}^{t}u  & 0 & -{}^{t}u \\
y & x  & 0 & -\sqrt{-1}\,v & \overline{u} & 0 
\end{pmatrix}
\]
and
\[
\begin{pmatrix} 
\sqrt{-1}v  & 0 & 0 & -\overline{y} & 0 & 0 \\
0  & 0 & -\sqrt{2}{}^{t}\overline{u}  & 0 & 0 & \sqrt{2}{}^{t}\overline{x}  \\
0 & -\sqrt{2}u  & -\sqrt{-1}v & 0 & \sqrt{2}\overline{x} & \overline{y} \\
y  & 0 & 0 & -\sqrt{-1}v  & 0 & 0  \\
0 & 0  & -\sqrt{2}{}^{t}x  & 0  & 0 & -\sqrt{2}{}^{t}u \\
0 & -\sqrt{2}x  & -y & 0 & -\sqrt{2}\overline{u} & \sqrt{-1}v 
\end{pmatrix}
\]
are conjugated by $\left(\begin{smallmatrix} A & 0 \\ 0 & A \end{smallmatrix}\right)$, where $A=\left(\begin{smallmatrix} 1/\sqrt{2} & 0 & 1/\sqrt{2} \\ 0 & I_{n} & 0 \\ -1/\sqrt{2} & 0 & 1/\sqrt{2} \end{smallmatrix}\right)$, we compare $W_{2}$ with the volume form of norm one of the unit sphere of $\HH^{n+1}$ to get
\begin{equation}\label{eq:Spcg1}
\int_{K_{G}/K_{P}} W_{2} = 2^{5n+3}\, \vol (S^{4n+3})\;.
\end{equation}
By~\eqref{eq:GVSp3},~\eqref{eq:Spcg0} and~\eqref{eq:Spcg1}, we obtain
\begin{equation}\label{eq:Spcg}
\fint_{\phi_{K_{G}}} \GV(\F_{P}) = \frac{(2n+3)^{4n+4}(4n+3)! \vol(S^{4n+3})}{2^{3n+1}(n+3)^{2n+2}} \, \omega_{G/K_{G}}\;.
\end{equation}

\subsubsection{The case $(G,G/P) = (F_{4(-20)},S^{15}_{\infty})$}\label{sec:f4}

Here, we refer to~\cite[Section~2.6]{Yokota2009} for the structure of $\f_{4}^{\CC} = \Lie(F_{4}^{\CC})$. For a standard choice of a Cartan subalgebra $\h = \bigoplus_{i=0}^{3} \CC H_{i}$, a system of simple roots $\{\alpha_{i}\}_{0 \leq i \leq 3}$ is given by
\[
\alpha_{0} = \lambda_{0} - \lambda_{1}\;, \quad \alpha_{1} = \lambda_{1} - \lambda_{2}\;, \quad \alpha_{2}=\lambda_{2}\;, \quad \alpha_{3}=\frac{1}{2}(-\lambda_{0} -\lambda_{1}-\lambda_{2} + \lambda_{3})\;,
\]
where $\lambda_i=B(\cdot,H_i)$ with respect to the Killing form $B$ of $\f_{4}^{\CC}$. One may choose positive roots of $\f_{4}^{\CC}$ listed in Table~\ref{table: positive roots} for this simple root system.

\begin{table}[h]
\begin{gather*}
  \begin{alignedat}{3}
    \lambda_{0} & = \phantom{2\text{}}\alpha_{0} + \phantom{2\text{}}\alpha_{1} +\phantom{2\text{}}\alpha_{2}\;,\phantom{\text{}+2\alpha_4}&\quad
    \lambda_{1} & = \phantom{\alpha_1+2\text{}}\alpha_{1} + \phantom{2\text{}}\alpha_{2}\;,\phantom{\text{}+2\alpha_4} \\
    \lambda_{2} & =\phantom{2\alpha_1+2\alpha_2+2\text{}}\alpha_{2}\;,\phantom{\text{}+2\alpha_4}&\quad
    \lambda_{3} & = \alpha_{0} + 2\alpha_{1} + 3\alpha_{2} + 2\alpha_{3}\;,  \\
    \lambda_{0} - \lambda_{1} & = \phantom{2\text{}}\alpha_{0}\;,\phantom{\text{}+2\alpha_2+2\alpha_3+2\alpha_4}&\quad 
    \lambda_{0} - \lambda_{2} & = \alpha_{0} + \phantom{2\text{}}\alpha_{1}\;,\phantom{\text{}+2\alpha_3+2\alpha_4}\; \\
    -\lambda_{0} + \lambda_{3} & = \phantom{2\alpha_1+2\text{}}\alpha_{1} + 2\alpha_{2} + 2\alpha_{3}\;, &\quad
    \lambda_{1} - \lambda_{2} & = \phantom{\alpha_1+2\text{}}\alpha_{1}\;,\phantom{\text{}+2\alpha_3+2\alpha_4}  \\
    -\lambda_{1} + \lambda_{3} & = \phantom{2\text{}}\alpha_1+ \phantom{2\text{}}\alpha_{1} + 2\alpha_{2} + 2\alpha_{3}\;, &\quad
    -\lambda_{2} + \lambda_{3} & = \alpha_1+ 2\alpha_{1} + 2\alpha_{2} + 2\alpha_{3}\;, \\
    \lambda_{0} + \lambda_{1} & = \phantom{2\text{}}\alpha_{0} + 2\alpha_{1} + 2\alpha_{2}\;,\phantom{\text{}+2\alpha_4}&\quad 
    \lambda_{0} + \lambda_{2} & = \alpha_{0} + \phantom{2\text{}}\alpha_{1} + 2\alpha_{2}\;,\phantom{\text{}+2\alpha_4} \\
    \lambda_{0} + \lambda_{3} & = 2\alpha_{0} + 3\alpha_{1} + 4\alpha_{2} + 2\alpha_{3}\;, &\quad
    \lambda_{1} + \lambda_{2} & = \phantom{\alpha_1+2\text{}}\alpha_{1} + 2\alpha_{2}\;,\phantom{\text{}+2\alpha_4}  \\
    \lambda_{1} + \lambda_{3} & = \phantom{2\text{}}\alpha_{0}+3\alpha_{1}+4\alpha_{2}+2\alpha_{3}\;,&\quad
    \lambda_{2} + \lambda_{3} & = \alpha_{0} + 2\alpha_{1} + 4\alpha_{2} + 2\alpha_{3}\;, 
  \end{alignedat} \\
  \begin{aligned}
    \textstyle\frac{1}{2}(\lambda_{0} + \lambda_{1} + \lambda_{2} + \lambda_{3}) & = \alpha_{0} 
    + 2\alpha_{1} + 3\alpha_{2} + \alpha_{3}\;, \\
    \textstyle\frac{1}{2}(-\lambda_{0} - \lambda_{1} - \lambda_{2} + \lambda_{3}) & =  \phantom{\alpha_{0} + 2\alpha_{1} + 3\alpha_{2}} + \alpha_{3}\;, \\
    \textstyle\frac{1}{2}(\lambda_{0} + \lambda_{1} - \lambda_{2} + \lambda_{3}) & = \alpha_{0} 
    + 2\alpha_{1} + 2\alpha_{2} + \alpha_{3}\;, \\
    \textstyle\frac{1}{2}(\lambda_{0} - \lambda_{1} + \lambda_{2} + \lambda_{3}) & =  \alpha_{0} 
    + \phantom{2\text{}}\alpha_{1} + 2\alpha_{2} + \alpha_{3}\;, \\
    \textstyle\frac{1}{2}(-\lambda_{0} - \lambda_{1} + \lambda_{2} + \lambda_{3}) & = \phantom{\alpha_1+2\alpha_2+2\text{}} \alpha_{2} + \alpha_{3}\;, \\
    \textstyle\frac{1}{2}(-\lambda_{0} + \lambda_{1} - \lambda_{2} + \lambda_{3}) & = \phantom{\alpha_1+2\text{}}\alpha_{1} 
    + \phantom{2\text{}}\alpha_{2} + \alpha_{3}\;, \\
    \textstyle\frac{1}{2}(\lambda_{0} - \lambda_{1} - \lambda_{2} + \lambda_{3}) & =  \alpha_{0} 
    + \phantom{2\text{}} \alpha_{1} + \phantom{2\text{}} \alpha_{2} + \alpha_{3}\;, \\
    \textstyle\frac{1}{2}(-\lambda_{0} + \lambda_{1} + \lambda_{2} + \lambda_{3}) & = \phantom{\alpha_1+2\text{}} \alpha_{1} 
    + 2\alpha_{2} + \alpha_{3}\;.
  \end{aligned}
\end{gather*}
\caption{Positive roots of $\f_{4}^{\CC}$}
\label{table: positive roots}
\end{table}

Here, $\alpha_{3}$ is the unique noncompact simple root of $\g$. Thus, the unique standard parabolic subalgebra $\pc$ of $\gc$ is of the form $\bigoplus_{\alpha \in -\Phi} (\gc)_{\alpha} \oplus \h \oplus \bigoplus_{\alpha \in \Upsilon^{+}} (\gc)_{\alpha}$, where $\Phi$ is a subset of $\Upsilon^{+}$ generated by $\{\alpha_{0}, \alpha_{1}, \alpha_{2}\}$. Then $\v=\bigoplus_{\alpha \in \Upsilon^{+} \setminus \Phi} (\gc)_{-\alpha}$ is spanned by the root vectors of the $15$ negative roots in
\[
\Lambda_{1} = \Big\{\frac{1}{2}(\pm \lambda_{0} \pm \lambda_{1} \pm \lambda_{2} - \lambda_{3})\Big\}
\quad\text{and}\quad
\Lambda_{2} =  \{-\lambda_{3}, \pm \lambda_{0} - \lambda_{3}, \pm \lambda_{1} - \lambda_{3}, \pm \lambda_{2} - \lambda_{3}\}\;,
\]
which are not generated by $\{-\alpha_{0}, -\alpha_{1}, -\alpha_{2}\}$, whose sum is $-11\lambda_{3}$. By Proposition~\ref{prop:Pittieh1}, we get $\Delta_{\F_{P}}(h_{1}) = -\frac{11}{2\pi}\lambda_{3}$. Take a set of root vectors $\{E_{\alpha}\}_{\alpha \in \Upsilon}$ so that $B(E_{\alpha},E_{-\alpha})=1$ for the Killing form $B$ of $\f_{4}^{\CC}$. For $\alpha\in \Upsilon$, let $H_{\alpha}$ be the element of $\h$ such that $B(H,H_{\alpha}) = \alpha(H)$ for any $H\in\h$. By Proposition~\ref{prop:PittieOmega} and since 
\[
d\lambda_{3}(E_{-\alpha},E_{\alpha}) = -B(E_{-\alpha},E_{\alpha})\, \lambda_{3}(H_{-\alpha})=-\lambda_{3}(H_{-\alpha})\;,
\] 
we get
\begin{equation*}
\Delta_{\F_{P}}(c_{1}) = d\Delta_{\F_{P}}(h_{1}) = \frac{11}{2\pi} \sum_{\alpha \in \Lambda_{1} \cup \Lambda_{2}} \lambda_{3}(H_{-\alpha}) \, E^{\vee}_{-\alpha} \wedge E^{\vee}_{\alpha}\;.
\end{equation*}
By using $B(\sum_{i=0}^{3} \lambda_{i}H_{i}, \sum_{j=0}^{3} \lambda'_{j}H_{j}) = 18 \sum_{i=0}^{3} \lambda_{i} \lambda_{i}'$, we obtain 
\[
\lambda_{3}(H_{-\alpha}) = \begin{cases}
9\;, & \alpha \in \Lambda_{1}\;, \\
18\;, & \alpha \in \Lambda_{2}\;.
\end{cases}
\]
Then we get the following formula of the Godbillon-Vey class:
\begin{equation}\label{eq:GVF42}
\GV(\F_{P}) = (2\pi)^{16} \Delta_{\F_{P}}(h_{1}c_{1}^{15}) = -2^{7} 3^{30} 11^{16} 15! \, \lambda_{3} \wedge \bigwedge_{\alpha \in \Lambda_{1} \cup \Lambda_{2}} E_{-\alpha}^{\vee} \wedge E_{\alpha}^{\vee}\;.
\end{equation}
By integrating $\GV(\F_{P})$ along the fibers of the canonical projection $\phi_{K_{G}} : G/K_{P} \to G/K_{G}$, we obtain
\begin{equation}\label{eq:F4cg}
\fint_{\phi_{K_{G}}} \GV(\F_{P}) = c_{G}\, \omega_{G/K_{G}}
\end{equation}
for a nonzero constant $c_{G}$.

\subsection{The Godbillon-Vey class spans the secondary characteristic classes}\label{sec:GVspan}

We assume that $(G,G/P)$ is equal to $(\SO_{0}(n+1,1),S^{n}_{\infty})$, $(\SU(n+1,1),S^{2n+1}_{\infty})$, $(\Sp(n+1,1),S^{4n+3}_{\infty})$ or $(F_{4(-20)},S^{15}_{\infty})$. In the last section, we saw that the Godbillon-Vey class of $\F_{P}$ is nontrivial, being given by a volume form on $G/K_{P}$. By using the computation, we will prove the following result in this section. 

\begin{prop}\label{prop:onlyGV3}
$\Delta_{\F}(H^{\bt}(WO_{q}))$ is spanned by the Godbillon-Vey class $\Delta_{\F}(h_{1}c_{1}^{q})$ for any $(G,G/P)$-foliation $\F$ of $M$.
\end{prop}

Recall that the secondary characteristic classes of the form $\Delta_{\F_{P}}(h_{I}c_{J})$ with nonempty $I$ are called {\em exotic}. First, we observe the following.

\begin{lemma}\label{lem:onlyGV}
Every nontrivial exotic secondary characteristic class of $\F_{P}$ is a multiple of the Godbillon-Vey class $\Delta_{\F_{P}}(h_{1}c_{1}^{q})$ in $H^{\bt}(\g,K_{P})$.
\end{lemma}

\begin{proof}
Note that $\deg h_{I}c_{J} \geq 2q+1$ for any $h_{I}c_{J}$ in $WO_{q}$ with nonempty $I$. Since $(G,G/P)$ is $(\SO_{0}(n+1,1),S^{n}_{\infty})$, $(\SU(n+1,1),S^{2n+1}_{\infty})$, $(\Sp(n+1,1),S^{4n+3}_{\infty})$ or $(F_{4(-20)},S^{15}_{\infty})$, we have $\dim G/K_{P} = 1 + 2 \dim G/P$. Then $\Delta_{\F_{P}}(h_{I}c_{J})=0$ for any $h_{I}c_{J}$ in $WO_{q}$ with $\deg h_{I}c_{J} > 2q+1$, and $\Delta_{\F}(h_{I}c_{J})$ is a multiple of a volume form on $G/K_{P}$ for any $h_{I}c_{J}$ in $WO_{q}$ with $\deg h_{I}c_{J} = 2q+1$. Since the Godbillon-Vey class is represented by a volume form on $G/K_{P}$ by~\eqref{eq:GVtrconf2},~\eqref{eq:GVSU2},~\eqref{eq:GVSp2} and~\eqref{eq:GVF42}, the proof is concluded.
\end{proof}

For the Pontryagin classes, an argument similar to Heitsch~\cite[Section~4]{Heitsch1986} for transversely projective foliations implies the following.

\begin{lemma}\label{lem:onlyGV2}
For any $(G,G/P)$-foliation $\F$ of $M$, the Pontryagin classes of $\nu \F$ are zero in $H^{\bt}(M)$.
\end{lemma}

\begin{proof}
Let $T_{0} (G/K_{G})$ be the complement of the zero section of the total space of the tangent bundle of $G/K_{G}$. Since $G/K_{P}$ is $G$-equivariantly diffeomorphic to the total space of the unit tangent bundle of $G/K_{G}$ in these cases, as mentioned in Section~\ref{sec:example}, we identify $G/K_{P}$ to a submanifold of $T_{0} (G/K_{G})$. We have a $G$-equivariant contraction $\gamma : T_{0} (G/K_{G}) \to G/K_{P}$. Let $\rho : T_{0} (G/K_{G}) \to G/K_{G}$ be the projection. Consider the vector bundle $\ker \rho_{*}$ over $T_{0} (G/K_{G})$ consisting of vertical vectors. Let $E= (\ker \rho_{*})|_{G/K_{P}}$. We have $\nu \F_{P} \oplus \RR_{\gamma} = E$, where $\RR_{\gamma}$ is the trivial vector bundle of rank one over $G/K_{P}$ spanned by vectors tangent to the fibers of $\gamma$. Here, $E$ has a $G$-invariant flat connection $\nabla'$ induced by the vector bundle structure of $\ker \rho_{*}$. Thus the total Pontryagin form $p(E,\nabla')$ of $(E,\nabla')$ is zero. 

Let $\widetilde{M}$ be the universal cover of $M$ and $\widehat{\dev} : \widetilde{M} \to G/K_{P}$ be a $\pi_{1}M$-equivariant map such that $\widetilde{\F} = \widehat{\dev}^{*}\F_{P}$, where $\widetilde{\F}$ is the lift of $\F$ to $\widetilde{M}$ (see Proposition~\ref{prop:h2}). By the $\pi_{1}M$-equivariance of $\widehat{\dev}$, the vector bundles $\widehat{\dev}^{*}\RR_{\gamma}$ and $\widehat{\dev}^{*}E$ over $\widetilde{M}$ descend to vector bundles over $M$, which are denoted by $\RR_{M}$ and $E_{M}$, respectively. Since $E_{M}$ admits a flat connection by construction, the total Pontryagin class $p(E_{M})$ of $E_{M}$ is $0$. Since $\nu \F \oplus \RR_{M} = E_{M}$ and by the product formula of total Pontryagin classes, we get $p(\nu \F) = p(E_{M})=0$. 
\end{proof}

Proposition~\ref{prop:onlyGV3} is a consequence of Lemmas~\ref{lem:onlyGV} and~\ref{lem:onlyGV2} and Theorem~\ref{thm:factrization}.

\subsection{Proof of Bott-Thurston-Heitsch type formulas}\label{sec:BT}

\subsubsection{The volume of flat $G/K_{G}$-bundles}\label{sec:vol}

Here, we recall the definition of the characteristic classes of $G/K_{G}$-bundles with flat $G$-connections. For a $G/K_{G}$-bundle $p_{Q} : Q \to N$ with a flat $G$-connection whose holonomy homomorphism is $h : \pi_{1}N \to G$, we have the Chern-Weil homomorphism $H^{\bt}(\g,K_{G}) \to H^{\bt}(Q;\RR)$. The sections $s$ of $p_{Q}$ are unique up to isotopy because of the contractibility of $G/K_{G}$. By composing the pull-back by $s$ with the Chern-Weil homomorphism, we get a map $H^{\bt}(\g,K_{G}) \to H^{\bt}(N;\RR)$. Since this map depends only on $h$, we denote it by $\Xi_{h}$. 

We fix an orientation on $G/K_{G}$. Let $\omega_{G/K_{G}}$ be the corresponding left invariant volume form on $G/K_{G}$ of norm $1$ with respect to the metric induced by the Killing form of $\g$. Let $\vol_{G/K_{G}} = [\omega_{G/K_{G}}]$ and $\vol(h) = \Xi_{h}\left( {\vol}_{G/K_{G}} \right) \in H^{m}(N;\RR)$, where $m = \dim G/K_{G}$. The class $\vol(h)$ is called the {\em volume} of $Q$ or of the holonomy presentation $h$. 

\begin{ex}\label{ex:volumeG} For the case where $N=\G \backslash G/K_{G}$ for a torsion-free uniform lattice $\G$ of $G$, the volume of $\G \hookrightarrow G$ is denoted by $\vol(\G)$, which is represented by the volume form on $N$ induced from $\omega_{G/K_{G}}$.
\end{ex}

\begin{rem}
$\Xi_{h}$ is called the {\em Borel regulator map\/} by algebraic geometers. For the importance of the volume in algebraic geometry, see~\cite{Reznikov1996} and the references therein.
\end{rem}

\subsubsection{Bott-Thurston-Heitsch type formulas for homogeneous foliations}\label{sec:BTHhom}

We apply the computation of the last section to calculate the Godbillon-Vey classes of homogeneous foliations $\F_{\G}$ in Example~\ref{ex:mob}. We consider the $K_{G}/K_{P}$-bundle $\phi_{K_{G}} : \G \backslash G/K_{P} \to \G \backslash G/K_{G}$. We may identify $\G \backslash G/K_{P}$ to the unit tangent sphere bundle of $\G \backslash G/K_{G}$. Fix an orientation of $\G \backslash G/K_{G}$ and let $\omega_{G/K_{G}}$ be the oriented volume form of norm one. We have an induced orientation on $\G \backslash G/K_{P}$ as the boundary of the unit tangent disk bundle of $\G \backslash G/K_{G}$. Fix a compatible orientation on the fibers of $\phi_{K_{G}}$.

\begin{prop}\label{prop:volGV}
Let $(G,G/P)$ be one of $(\SO_{0}(n+1,1),S^{n}_{\infty})$, $(\SU(n+1,1),S^{2n+1}_{\infty})$, $(\Sp(n+1,1),S^{4n+3}_{\infty})$ or $(F_{4(-20)},S^{15}_{\infty})$. Let $q=\dim G/P$ {\rm(}the codimension of $\F_{\G}${\rm)}. We have 
\begin{align}\label{eq:lat}
\frac{1}{(2\pi)^{q+1}}\fint_{\phi_{K_{G}}} \GV(\F_{\G}) = c_{G}\, \omega_{G/K_{G}}
\end{align}
in $\Omega^{q+1}(\G \backslash G/K_{G})$, where $c_{G}$ is a nonzero constant depending on $(G,G/P)$ given by Table~\ref{table: c_G} in the case of classical groups, where $\vol(S^{q})$ is the volume of the unit sphere in $\RR^{q+1}$, given by $\vol(S^{q}) = \frac{(2\pi)^{(q+1)/2}}{2 \cdot 4 \cdots (q-1)}$ for odd $q$ and $\vol(S^{q}) = \frac{2(2\pi)^{q/2}}{1 \cdot 3 \cdots (q-1)}$ for even $q$.
\end{prop}

\renewcommand{\arraystretch}{2.5}
\begin{table}[h]
\[
\begin{array}{| c | c |} 
\hline
(G,G/P) & c_{G} \\ 
\hline
(\SO_{0}(n+1,1),S^{n}_{\infty}) & \displaystyle (-1)^{\frac{n(n+1)}{2}}\frac{n^{\frac{n+1}{2}}\vol (S^{n})}{2^{\frac{3n+3}{2}}\pi^{n+1}} \\
(\SU(n+1,1),S^{2n+1}_{\infty}) & \displaystyle (-1)^{n+1}\frac{2(n+1)^{2n+2} (2n+1)!\vol(S^{2n+1})}{(2\pi)^{2n+2}(n+2)^{n+1}} \\
(\Sp(n+1,1),S^{4n+3}_{\infty}) & \displaystyle \frac{(2n+3)^{4n+4}(4n+3)! \vol(S^{4n+3})}{2^{3n+1}(2\pi)^{4n+4}(n+3)^{2n+2}} \\ 
\hline
\end{array}
\]
\caption{The constant $c_G$}
\label{table: c_G}
\end{table}
\renewcommand{\arraystretch}{1}

\begin{proof}
As explained in Section~\ref{sec:example}, we can identify $G/K_{P}$ to the unit tangent sphere bundle of $G/K_{G}$. Let $X=\G \backslash G/K_{G}$ and $Y=\G \backslash G/K_{P}$. Since the $\G$-action on $G/K_{G}$ is isometric, we can identify $Y$ to the unit tangent sphere bundle of $X$. Here $\GV(\F_{P})$ is $\G$-invariant and induces $\GV(\F_{\G})$ on $Y$. Then the proposition in the case where $G=F_{4(-20)}$ directly follows from the computation~\eqref{eq:F4cg}. In the other cases, we can determine the constant $c_{G}$ up to sign by \eqref{eq:inttt},~\eqref{eq:SUcg} and~\eqref{eq:Spcg}. Thus, to compute $c_{G}$, it is sufficient to compute the sign of $c_{G}$. Let $\nu$ be a $q$-form on $Y$ whose restriction to each fiber of the sphere bundle $\phi_{K_{G}} : Y \to X$ is an oriented volume form of volume one. Then we have
\[
\GV(\F_{\G}) = c_{G}\, \phi_{K_{G}}^{*}\omega_{G/K_{G}} \wedge \nu
\]
Let $D$ be the tangent unit disk bundle of $X$ with the induced orientation from $X$. We extend $\nu$ to a $q$-form $\tilde\nu$ on $D$ so that the restriction of $d\tilde\nu$ to each fiber of the projection $D \to X$ is an oriented volume form. For instance, take $\tilde\nu=\frac{1}{q}\sum_{i=1}^q(-1)^{i-1}x_i\,dx_1\wedge\dots\wedge\widehat{dx_i}\wedge\dots\wedge dx_n$, using the fiberwise coordinates $(x_1,\dots,x_n)$ given by normal charts depending smoothly on their center points locally in $X$ (they are defined by a local section of the corresponding bundle of oriented orthonormal frames). Then we have
\[
\int_{Y} \GV(\F_{\G}) = \int_{D}d\GV(\F_{\G}) = (-1)^{q+1} c_{G} \int_{D} \phi_{K_{G}}^{*}\omega_{G/K_{G}} \wedge d\tilde\nu\;.
\]
We see that $\int_{D} \phi_{K_{G}}^{*}\omega_{G/K_{G}} \wedge d\tilde\nu$ is positive. Then, to compute the sign of $c_{G}$, it suffices to compute the sign of $\int_{Y} \GV(\F_{\G})$. By the description of $\Delta_{\F_{P}}(h_{1})$ in \eqref{eq:SOh1}, \eqref{eq:SUh1} and \eqref{eq:h1Sp}, we see that $-\Delta_{\F_{\G}}(h_{1})$ is equal, up to the product of a positive constant, to the standard contact form whose Reeb flow is the geodesic flow on $X$ (for a description of geodesic flows on $G/K_{G}$, see \cite{Quint2006} for example). Let $\eta=-\Delta_{\F_{\G}}(h_{1})$. Then $\GV(\F_{\G})$ is equal to $-\eta \wedge (-d\eta)^{q}$ up to the product of a positive constant. With $D$ as before, let $\tilde\eta$ be the $1$-form on $D$ which is the pull back of the Liouville $1$-form on $T^{*}X$ by the mapping $\flat : D \to T^{*}X; v \mapsto (w \mapsto g_{X}(v, w))$, where $g_{X}$ is the metric on $X$ induced from the Killing form on $\Lie(G)$. It is known that $\tilde\eta|_{Y} = \eta$ (see \cite[Lemma~1.37]{Paternain1999}). Since $Y=\partial D$, we have
\begin{align*}
\int_{Y} -\eta \wedge (-d\eta)^{q}= (-1)^{q+1} \int_{D} (d\tilde\eta)^{q+1}\;.
\end{align*}
We fix a point $x$ in $X$. Take a point $y \in \operatorname{pr}^{-1}(x)$, where $\operatorname{pr} : D \to X$ is the projection. Here, by using the Levi-Civita connection on $TX$, we see that $T_{y}D$ is splitted into the horizontal part $\mathcal{H}$ and the vertical part $\mathcal{V}=\ker \operatorname{pr}_{*}$. Here $\mathcal{H}$ is identified to $T_{x}X$ by $\operatorname{pr}_{*}$, and $\mathcal{V}$ is naturally identified to $T_{x}X$ by using the affine structure of $T_{x}X$. Let $\{v_{1}, \ldots, v_{q+1}\}$ be a basis of $T_{x}X$. Take a basis $\{a_{1}, \ldots, a_{q+1}\}$ of $\mathcal{H}$ and a basis of $\{b_{1}, \ldots, b_{q+1}\}$ of $\mathcal{V}$ so that both $a_{i}$ and $b_{i}$ are identified to $v_{i}$. Then, it is well known that $d\tilde\eta$ is equal to $\sum_{i=1}^{q+1} (a_{i} \wedge b_{i})$ up to the product of a positive constant (see \cite[Section~1.3.2]{Paternain1999}). Here the oriented volume form on $D$ induced from $X$ is equal to $a_{1} \wedge \ldots \wedge a_{q+1} \wedge b_{1} \wedge \ldots \wedge b_{q+1}$. Therefore, by~\eqref{eq:sign}, we have that the sign of $\int_{Y} \GV(\F_{\G})$ is equal to $(-1)^{\frac{(q+1)(q+2)}{2}}$, which implies that the sign of $c_{G}$ is equal to $(-1)^{\frac{q(q+1)}{2}}$. Thus we obtain the constants $c_{G}$ in Table~\ref{table: c_G}.
\end{proof}

The same computation gives the following relation of the Godbillon-Vey class and the volume in the level of Lie algebra cohomology.

\begin{prop}\label{prop:volGV2}
Let $(G,G/P)$ be one of $(\SO_{0}(n+1,1),S^{n}_{\infty})$, $(\SU(n+1,1),S^{2n+1}_{\infty})$, $(\Sp(n+1,1),S^{4n+3}_{\infty})$ or $(F_{4(-20)},S^{15}_{\infty})$. Let $q=\dim G/P$ {\rm(}the codimension of $\F_{P}${\rm)}. We have
\begin{equation}
\frac{1}{(2\pi)^{q+1}}\fint_{\phi_{K_{G}}} \GV(\F_{P})= c_{G}\, \omega_{G/K_{G}}
\end{equation}
in $\big(\bigwedge^{q+1}\g^{*}\big)_{K_{G}}$ for some orientations of $G/K_{G}$ and the fibers of $\phi_{K_{G}} : G/K_{P} \to G/K_{G}$, where $c_{G}$ is a nonzero constant depending on $(G,G/P)$ given in Table~\ref{table: c_G} in the case of classical groups.
\end{prop}

\begin{rem}
By~\cite[Theorem~7.83]{KamberTondeur1975b}, the following diagram commutes: 
\begin{equation}
\xymatrix{H^{\bt}(\g,K_{P}) \ar[r] \ar[d]_{\fint} & H^{\bt}(\G \backslash G /K_{P}) \ar[d]^{\fint} \\
H^{\bt}(\g,K_{G}) \ar[r]_<<<<<{\kappa} & H^{\bt}(\G \backslash G /K_{G})\;.}
\end{equation}
The homomorphism $\kappa$ is well known to be injective. The commutativity describes the relation between Propositions~\ref{prop:volGV} and~\ref{prop:volGV2}.
\end{rem}

Proposition~\ref{prop:volGV} or~\ref{prop:volGV2} implies the following.
\begin{cor}\label{cor:FGGV}
Under the assumption and with the notation of Proposition~\ref{prop:volGV}, we have
\[
\frac{1}{(2\pi)^{q+1}} \int_{\G \backslash G /K_{P}} \GV(\F_{P}) = c_{G}\, \vol(\G \backslash G /K_{G})\;,
\]
where the volume $\vol(\G \backslash G /K_{G})$ is considered with respect to the metric induced by the Killing form of $\g$.
\end{cor}

\subsubsection{Bott-Thurston-Heitsch type formulas for suspension foliations}

The homogeneous foliations are suspension foliations over locally symmetric spaces whose holonomy homomorphisms are the canonical embeddings of lattices. We will show Bott-Thurston-Heitsch type formulas (Theorem~\ref{thm:BT}) that can be applied to more general suspension foliations.

The suspension foliations $\F$ in the statement of Theorem~\ref{thm:BT} are $(G,G/P)$-foliations on the total spaces of $G/P$-bundles over manifolds $N$ which are transverse to the $G/P$-fibers by construction. In the case where $\dim G/P > 1$, it is easy to see that, conversely, any $(G,G/P)$-foliation on the total space of a $G/P$-bundle over a manifold $N$ which is transverse to the $G/P$-fibers is a suspension foliation like in the statement of Theorem~\ref{thm:BT}. In this section, we prove Theorem~\ref{thm:BT} for $(G,G/P)$-foliations on the total spaces of $G/P$-bundles over manifolds $N$ which are transverse to the $G/P$-fibers. Part of the argument will be used later in a more general situation.

Let $(G,G/P)$ be one of $(\SO_{0}(n+1,1),S^{n}_{\infty})$, $(\SU(n+1,1), S^{2n+1}_{\infty})$, $(\Sp(n+1,1),S^{4n+3}_{\infty})$ or $(F_{4(-20)},S^{15}_{\infty})$. Let $q = \dim G/P$ (the codimension of $(G,G/P)$-foliations). Consider the case of codimension $q > 1$; namely, all cases except $(\SO_{0}(2,1),S^{1}_{\infty})$ and $(\SU(1,1), S^{1}_{\infty})$. Let $N$ be a smooth manifold, and $p_{M} : M \to N$ an $S^{q}$-bundle over $N$. Let $\F$ be a $(G,G/P)$-foliation of $M$ which is transverse to the fibers of $p_{M}$. Since $G$ preserves an orientation of $G/P$, it follows that $p_{M}$ is orientable.

We have two $G$-equivariant fibrations on $G/K_{P}$:
\[
\xymatrix{ G/P  && G/K_{P} \ar[ll]_<<<<<<<<<{\phi_{P}} \ar[d]^{\phi_{K_{G}}}  \\
                      && G/K_{G}\;. }
\]
Now, it is easy to see that the fibers of $\phi_{P}$ and $\phi_{K_{G}}$ are of complementary dimension and transverse to each other. This observation implies the following.
\begin{lemma}\label{lem:lift}
Let $\dev : \wt{M} \to G/P$ be the developing map of $\F$. For any $\pi_{1}M$-equivariant map $s : \wt{M} \to G/K_{G}$, there exists a unique map $\widehat{\dev} : \wt{M} \to G/K_{P}$ which is $\pi_{1}M$-equivariant, satisfies $\widehat{\F} = \widehat{\dev}^{*}\F_{P}$ and makes the following diagram commutative: 
\[
\xymatrix{G/P  && G/K_{P} \ar[ll]_<<<<<<<<<{\phi_{P}} \ar[d]^{\phi_{K_{G}}} \\
         \wt{M} \ar[u]^{\dev} \ar[rr]_<<<<<<<<<<{s} \ar[rru]^<<<<<<<<<<<<{\widehat{\dev}} && G/K_{G}\;.  }
\]
Moreover, if $s$ is regular at a point $x\in \widetilde{M}$, then $\widehat{\dev}$ is regular at $x$.
\end{lemma}

\begin{proof}
The equality $\widehat{\F} = \widehat{\dev}^{*}\F_{P}$ is a trivial consequence of the construction like in Proposition~\ref{prop:h2}. To prove the latter part of Lemma~\ref{lem:lift}, note that $\dev$ is a submersion.
\end{proof}

Regard $\hol(\F)$ as a homomorphism $\pi_{1}N \cong \pi_{1}M \to G$. Given an orientation of $G/K_{G}$, the volume $\vol(\hol(\F))$ of $\hol(\F)$ is defined in $H^{q+1}(N;\RR)$ as mentioned in Section~\ref{sec:vol}. 

\begin{prop}\label{prop:r1}
We orient $G/K_{G}$ and the fibers of $\phi_{K_{G}}$ like in Proposition~\ref{prop:volGV}. Then
\begin{equation}\label{eq:GV1}
\frac{1}{(2\pi)^{q+1}} \fint_{p_{M}}  \GV(\F) = c_{G} \vol(\textstyle{\hol_{\F}})
\end{equation}
in $H^{q+1}(N;\RR)$ for an orientation of the fibers of $p_{M}$, where $c_{G}$ is the function of $(G,G/P)$ given in Table~\ref{table: c_G}.
\end{prop}

\begin{proof}
Take a $\pi_{1}N$-equivariant map $\overline{s} : \wt{N} \to G/K_{G}$. We get a $\pi_{1}M$-equivariant map $s = \overline{s} \circ p_{\wt{M}} : \wt{M} \to G/K_{G}$, where $p_{\wt{M}} : \wt{M} \to \wt{N}$ is the canonical projection. By Lemma~\ref{lem:lift}, we get a $\pi_{1}M$-equivariant map $\widehat{\dev} : \wt{M} \to G/K_{P}$ which makes the following diagram commutative:
\begin{equation}\label{eq:pullback}
\xymatrix{  \wt{M} \ar[r]^<<<<<{\widehat{\dev}} \ar[d]_{p_{\wt{M}}}    & G/K_{P} \ar[d]^{\phi_{K_{G}}} \\
               \wt{N} \ar[r]_<<<<<{\overline{s}} & G/K_{G}\;. }
\end{equation}
Since $\F$ is transverse to the fibers of $p_{M}$, the restriction of $\widehat{\dev}$ to each fiber of $p_{\wt{M}}$ is a covering map onto a fiber of $\phi_{K_{G}}$. Since $p_{\wt{M}}$ and $\phi_{K_{G}}$ are $S^{q}$-bundles and $q>1$, the restriction of $\widehat{\dev}$ to each fiber of $p_{\wt{M}}$ is a diffeomorphism. Thus the diagram~\eqref{eq:pullback} is the pull-back of fiber bundles. We fix an orientation of the fibers of $p_{M}$ so that it is compatible with the orientation of the fibers of $\phi_{K_{G}}$ under $\widehat{\dev}^{*}$. Then $\fint_{p_{\wt{M}}} \widehat{\dev}^{*} \beta = \overline{s}^{*} \fint_{\phi_{K_{G}}} \beta$ for any $\beta\in\Omega^{\bt}(G/K_{P})$. We have $\frac{1}{(2\pi)^{q+1}}\fint_{\phi_{K_{G}}} \GV(\F_{P})  = c_{G}\, \omega_{G/K_{G}}$ in $(\bigwedge^{q+1}\g^{*})_{K_{P}}$ by Proposition~\ref{prop:volGV2}. Let $\wt{\F}$ be the lift of $\F$ to $\wt{M}$. Since $\wt{\F}=\widehat{\dev}^{*}\F_{P}$ by Lemma~\ref{lem:lift}, we have
\[
\frac{1}{(2\pi)^{q+1}}\fint_{p_{\wt{M}}} \GV(\wt{\F}) = \frac{1}{(2\pi)^{q+1}}\fint_{p_{\wt{M}}} \widehat{\dev}^{*} \GV(\F_{P}) = \overline{s}^{*} \fint_{\phi_{K_{G}}} \GV(\F_{P}) = c_{G}\, \overline{s}^{*} \omega_{G/K_{G}}
\]
in $\Omega^{\bt}(\wt{N})^{\pi_{1}N}$, which implies \eqref{eq:GV1}.
\end{proof}

\begin{proof}[Proof of Theorem~\ref{thm:BT}]
The theorem in the case where $G=F_{4(-20)}$ directly follows from Propositions~\eqref{eq:F4cg} and~\ref{prop:r1}. Consider the other cases. Since the sign of both sides of~\eqref{eq:BT} change when the orientation of the fibers of $p_{M}$ changes, it suffices to prove~\eqref{eq:BT} for any fixed orientation of the fibers of $p_{M}$.

We fix an orientation of $G/K_{G}$. As explained in Section~\ref{sec:example}, we may identify $G/K_{P}$ to the unit tangent sphere bundle of $G/K_{G}$. Consider the induced orientation of $G/K_{P}$ as the boundary of the unit tangent disk bundle of $G/K_{G}$. Fix the compatible orientation of  the fibers of $\phi_{K_{G}}$. By assumption, $G/K_{G}$ is of even dimension $q+1$. Since $G/K_{G}$ has a $G$-invariant metric, the Euler form $e$ of the oriented tangent bundle of $G/K_{G}$ is a left invariant volume form on $G/K_{G}$. Thus there exists a constant $\mu_{G}$ such that 
\begin{equation}\label{eq:mu}
e = {\mu}_{G}\,{\vol}_{G/K_{G}}\;,
\end{equation}
where $\vol_{G/K_{G}}$ is the left invariant oriented volume form of norm $1$ with respect to the Killing form on $\g$. Let $\vol_{\G}$ and $e_{\G}$ be the volume forms on $\G \backslash G / K_{G}$ such that $p_{N}^{*}\vol_{\G} = \vol_{G/K_{G}}$ and $p^{*}_{N}e_{\G} = e$, where $p_{N} : G/K_{G} \to N$ is the universal covering of $N$. By the Hirzebruch proportionality principle~\cite[Theorem~3.3]{CahnGilkeyWolf1976} (see also~\cite{KobayashiOno1990}), we can compute the constant $\mu_{G}$ by using the compact dual $K_{\Gc}/K_{G}$ of $G/K_{G}$ as follows:
\[
\mu_{G} = \frac{\int_{\G \backslash G / K_{G}} e_{\G}}{\int_{\G \backslash G / K_{G}}\vol_{\G}}  = (-1)^{(q+1)/2}\,\frac{e(K_{\Gc}/K_{G})}{\vol(K_{\Gc}/K_{G})}\;,
\]
where $e(K_{\Gc}/K_{G})$ is the Euler number of $K_{\Gc}/K_{G}$, and the volume $\vol(K_{\Gc}/K_{G})$ is considered with respect to the metric induced by the Killing form on $\gc$. The volume $\vol(K_{\Gc}/K_{G})$ was computed in~\cite{AbeYokota1997}, obtaining Table~\ref{table: e and vol}, where we also indicate the Euler number $e(K_{\Gc}/K_{G})$ of $K_{\Gc}/K_{G}$. Thus, with this computation of $\mu_{G}$, Theorem~\ref{thm:BT} follows from~\eqref{eq:mu} and Proposition~\ref{prop:r1}, where the constant $r_{G}$ in~\eqref{eq:BT} is $r_{G} = \mu_{G}^{-1} c_{G}$ with $c_{G}$ given in Table~\ref{table: c_G}. Since the sign of $c_{G}$ is equal to $(-1)^{q(q+1)/2}$ as we saw in the proof of Proposition~\ref{prop:volGV}, the sign of $r_{G}$ is equal to $(-1)^{(q+1)^{2}/2}$, which implies that $r_{G}$ is positive.
\end{proof}

\renewcommand{\arraystretch}{2.5}
\begin{table}[h]
\[
\begin{array}{|c|c|c|}\hline
K_{\Gc}/K_{G}  & e & \vol \\ \hline
\RR P^{n+1} & 1 & \displaystyle 2^{\frac{n-1}{2}}n^{\frac{n+1}{2}}\vol(S^{n+1}) \\
\CC P^{n+1} & n+2 & \displaystyle \frac{2^{n+1}(n+2)^{n+1}\pi^{n+1}}{(n+1)!} \\
\HH P^{n+1} & n+2 & \displaystyle  \frac{2^{6(n+1)}(n+3)^{2(n+1)}\pi^{2(n+1)}}{(2n+3)!} \\ \hline
\end{array}
\]
\caption{The Euler number and the volume of $K_{\Gc}/K_{G}$.}
\label{table: e and vol}
\end{table}
\renewcommand{\arraystretch}{1}

\section{The case where $G/P = S^{q}$ for even $q$}\label{sec:even}

\subsection{Integration along the fibers of Haefliger structures}

For transversely projective foliations, the integration of secondary invariants along the fibers of the Haefliger structures was computed by Brooks-Goldman~\cite[Lemma~2]{BrooksGoldman1984} and Heitsch~\cite[Lemma in Section~5]{Heitsch1986} to prove Proposition~\ref{prop:01}, which is an essential part of the proof of Theorem~\ref{thm:BGH}. In this section, we will see that such computation is reduced to a computation in Lie algebra cohomology in the case where $G/P$ is a sphere. This observation enables us to state a sufficient condition, that implies Proposition~\ref{prop:01}, in terms of Lie algebra cohomology. We will also see that Proposition~\ref{prop:01} is not true for transversely conformally flat foliations of even codimensions. In this section, the coefficient ring of cohomology is $\CC$.

Let $G$ be a connected algebraic linear semisimple Lie group. Let $P$ be a parabolic subgroup of $G$ such that $G/P$ is diffeomorphic to a sphere. Let $\X_{G}(\F)$ be the principal $G$-bundle over $M$ associated to $\F$. Consider the diagram of bundle maps between fiber bundles over $M$,
\begin{equation}\label{eq:fiberbundles}
\xymatrix{ \X_{\Gc}(\F)/K_{P} \ar[d] & \X_{G}(\F)/K_{P} \ar[d] \ar[l] \\
 \X_{\Gc}(\F)/K_{G} & \X_{G}(\F)/K_{G}\;, \ar[l] } 
\end{equation}
where the horizontal maps are inclusions defined by fiberwise complexification and the vertical maps are canonical projections. Let $\H^{\bt}(\Gc/K_{P})$, $\H^{\bt}(G/K_{P})$ and $\H^{\bt}(\Gc/K_{G})$ be the local systems over $M$ associated to the fiber bundles $\X_{\Gc}(\F)/K_{P}$, $\X_{G}(\F)/K_{P}$ and $\X_{\Gc}(\F)/K_{G}$, respectively. Note that the local system associated to $\X_{G}(\F)/K_{G}$ is trivial because the fiber $G/K_{G}$ is contractible. By using integration along fibers of the vertical maps of~\eqref{eq:fiberbundles}, we get the commutative diagram
\begin{equation}\label{eq:fib2}
\xymatrix{ H^{\bt}(M;\H^{\bt}(\Gc/K_{P})) \ar[d]_{\fint} \ar[r] & H^{\bt}(M;\H^{\bt}(G/K_{P})) \ar[d]^{\fint} \\
 H^{\bt}(M;\H^{\bt}(\Gc/K_{G})) \ar[r] & H^{\bt}(M)\;.} 
\end{equation}
Observe that we have natural isomorphisms
\begin{gather}
H^{\bt}(\g,K_{P}) \otimes \CC \cong H^{\bt}(\k_{\Gc},K_{P}) \otimes \CC \cong H^{\bt}(K_{\Gc}/K_{P}) \cong H^{\bt}(\Gc/K_{P})\;, \label{eq:Weyl1}\\
H^{\bt}(\g,K_{G}) \otimes \CC \cong H^{\bt}(\k_{\Gc}, K_{G}) \otimes \CC \cong H^{\bt}(K_{\Gc}/K_{G}) \cong H^{\bt}(\Gc/K_{G})\;, \label{eq:Weyl2}
\end{gather}
where the first isomorphisms in the two equations are the well known isomorphism in the Weyl's trick ~\cite[Section~3]{KobayashiOno1990}. We get the commutative diagram
\begin{equation}\label{eq:fib3}
\xymatrix{ H^{\bt}(\g, K_{P}) \otimes \CC  \ar[d]_{\fint} \ar[r] & H^{\bt}(\Gc/K_{P}) \ar[d]_{\fint} \ar[r] & H^{\bt}(M;\H^{\bt}(\Gc/K_{P})) \ar[d]_{\fint} \\
H^{\bt}(\g,K_{G}) \otimes \CC \ar[r] & H^{\bt}(\Gc/K_{G}) \ar[r] & H^{\bt}(M;\H^{\bt}(\Gc/K_{G}))\;.} 
\end{equation}

Recall that $\X_{\Gc}(\F)/K_{P}$ has a $(G,G/P)$-foliation $p^{*}\E_{\hol(\F)}$, which is obtained by pulling back the foliation $\E_{\hol(\F)}$ on $\X_{\Gc}(\F)/P$ defined by the flat $G$-connection by the canonical projection $p : \X_{\Gc}(\F)/K_{P} \to \X_{\Gc}(\F)/P$. By combining Theorem~\ref{thm:factrization}, diagrams~\eqref{eq:fib2} and~\eqref{eq:fib3}, and the definition of the characteristic homomorphisms, we get the following.

\begin{prop}\label{prop:fibint}
The following diagram is commutative: 
\begin{equation}\label{eq:bigdiag}
\xymatrix{ H^{\bt}(WO_{q}) \ar[d]_{\Delta_{\F_{P}}} \ar[drr]^{\Delta'_{p^{*}\E_{\hol(\F)}}} & & \\
H^{\bt}(\g,K_{P}) \otimes \CC \ar[d]_{\fint} \ar[r] & H^{\bt}(M;\H^{\bt}(\Gc/K_{P})) \ar[d]_{\fint} \ar[r] & H^{\bt}(M;\H^{\bt}(G/K_{P})) \ar[d]_{\fint} \\
H^{\bt}(\g,K_{G}) \otimes \CC \ar[r] \ar @/_2pc/ [rr]_{\Xi_{\hol (\F)}} & H^{\bt}(M;\H^{\bt}(\Gc/K_{G})) \ar[r] & H^{\bt}(M)\;,} 
\end{equation}
where $\Delta'_{p^{*}\E_{\hol(\F)}}$ is the map induced by the characteristic homomorphism $\Delta_{p^{*}\E_{\hol(\F)}}$ of $p^{*}\E_{\hol(\F)}$, and $\Xi_{\hol (\F)} : H^{\bt}(\g,K_{G}) \to H^{\bt}(M)$ is the characteristic homomorphism of the flat $G/K_{G}$-bundle $\X_{G}(\F)/K_{G} \to M$ mentioned in Section~\ref{sec:vol}.
\end{prop}

By Propositions~\ref{prop:eachsigma} and~\ref{prop:fibint}, and the following fact, we will get a sufficient condition for the finiteness of secondary characteristic classes in terms of Lie algebra cohomology (Proposition~\ref{prop:suff}). 
 
\begin{lemma}\label{lem:Gysin}
Let $\sigma$ be a cohomology class of $\X_{G}(\F)/K_{P}$. Then $\sigma$ belongs to the image of $\pi^{*}_{G/K_{P}} : H^{\bt}(M) \to H^{\bt}(\X_{G}(\F)/K_{P})$ if and only if $\fint \sigma = 0$.
\end{lemma}

\begin{proof}
Note that $\X_{G}(\F)/K_{P}$ is homotopy equivalent to a sphere bundle $\X_{G}(\F)/P$ over $M$. Since $\X_{G}(\F)/P$ has a section, the Gysin sequence splits to give the exact sequence
\[
\xymatrix{0 \ar[r] & H^{\bt}(M) \ar[r]^<<<<<{\pi^{*}_{G/K_{P}}} & H^{\bt}(\X_{G}(\F)/K_{P}) \ar[r]^<<<<<{\fint} & H^{\bt}(M) \ar[r] & 0\;,
} 
\]
and the result follows.
\end{proof}

The composite of the upper horizontal maps of~\eqref{eq:bigdiag} is induced on the $E_{2}$-terms of the Leray-Hirsch spectral sequence of $\X_{G}(\F)/K_{P} \to M$ by the characteristic homomorphism $H^{\bt}(\g,K_{P}) \to H^{\bt}(\X_{G}(\F)/K_{P})$ of the $(G,G/P)$-foliation $p^{*}\E_{\hol (\F)}$ on $\X_{G}(\F)/K_{P}$ mentioned in Proposition~\ref{prop:h2}. Thus, as a consequence of Propositions~\ref{prop:eachsigma} and~\ref{prop:fibint}, and Lemma~\ref{lem:Gysin}, we get the following.

\begin{prop}\label{prop:suff}
If $\fint \Delta_{\F_{P}}(\sigma) = 0$ for $\sigma\in H^{\bt}(WO_{q})$, then 
\[
\# \{\, \Delta_{\F}(\sigma) \in H^{\bt}(M;\RR) \mid \F \in \Fol(G,G/P) \,\} < \infty\;.
\]
\end{prop}

This proposition reduces the latter condition to the former condition, which involves only Lie algebra cohomology. Thus the following proposition gives an alternative proof of a consequence of the residue formulas of Heitsch.

\begin{prop}[{Heitsch~\cite[Theorem~4.2]{Heitsch1978} and~\cite[Theorem~2.3]{Heitsch1983}}]\label{prop:Heitsch}
In the case where $(G,G/P) = (\SL(q+1;\RR),S^{q})$ for even $q$, we have $\fint \Delta_{\F_{P}}(\sigma) = 0$ for any $\sigma\in H^{\bt}(WO_{q})$.
\end{prop}

\begin{proof} We will use the notation of Example~\ref{sec:trproj}. First, we show $\fint \GV(\F_{P}) = 0$. By~\eqref{eq:GVtrproj1} and~\eqref{eq:sign}, we get
\begin{align*}
 \GV(\F_{P}) & = -(q+1)^{q+1}\, (q+1)!\, E^{\vee}_{11} \wedge \bigwedge_{k=2}^{q+1} E^{\vee}_{1k} \wedge E^{\vee}_{k1} \\
& = \frac{(-1)^{\frac{q(q-1)}{2}+1} (q+1)^{q+1}\, (q+1)!}{2^{q}}\, E^{\vee}_{11} \wedge \bigwedge_{k=2}^{q+1} (E^{\vee}_{1k} + E^{\vee}_{k1}) \wedge \bigwedge_{k=2}^{q+1} (E^{\vee}_{1k} - E^{\vee}_{k1})\;.
\end{align*}
Here, $\bigwedge_{k=2}^{q+1} (E^{\vee}_{1k} - E^{\vee}_{k1})$ is a volume form of $\SO(q+1)/\SO(q) \approx S^{q}$. Thus $\fint  \GV(\F_{P})$ is obtained by integrating $E^{\vee}_{11} \wedge \bigwedge_{k=2}^{q+1} (E^{\vee}_{1k} + E^{\vee}_{k1})$ over $S^{q}$. But, since $q$ is even, $E^{\vee}_{11} \wedge \bigwedge_{k=2}^{q+1} (E^{\vee}_{1k} + E^{\vee}_{k1})$ is an odd function on $S^{q}$; namely, we have
\[
s^{*}\Big(E^{\vee}_{11} \wedge \bigwedge_{k=2}^{q+1} (E^{\vee}_{1k} + E^{\vee}_{k1}) \Big) = - E^{\vee}_{11} \wedge \bigwedge_{k=2}^{q+1} (E^{\vee}_{1k} + E^{\vee}_{k1})\;,
\]
where $s$ is the antipodal map of $S^{q}$. So the integration of $E^{\vee}_{11} \wedge \bigwedge_{k=2}^{q+1} (E^{\vee}_{1k} + E^{\vee}_{k1})$ over $S^{q}$ is zero. This implies that $\fint \GV(\F_{P})=0$.

Note that $h_{I}(\Theta_{MC})$ is $K_{G}$-basic; namely, $h_{I}(\Theta_{MC})$ is the pull-back of a differential form on $\Gc/K_{G}$. Thus, by~\eqref{eq:GVtrproj2}, 
\[
\fint \Delta_{\F_{P}}(h_{1}h_{I}c_{1}^{q}) = \frac{1}{(2\pi)^{q+1}}h_{I}(\Theta_{MC}) \fint \GV(\F_{P})= 0\;.
\]
Since other secondary characteristic classes are generated by the classes of the form $h_{1}h_{I}c_{1}^{q}$ by Theorem~\ref{thm:Pittie}, the result follows.
\end{proof}

\begin{rem}\label{rem:simpler}
Heitsch~\cite{Heitsch1986} applied consequences of his residue formulas, Theorem~\ref{thm:BT1} and Proposition~\ref{prop:Heitsch}, to prove our Proposition~\ref{prop:01} for the case where $(G,G/P)=(\SL(q+1;\RR),S^{q})$ for any $q$, and therefore Theorem~\ref{thm:BGH}. For even $q$, our proof of Proposition~\ref{prop:Heitsch} is slightly simpler than the original proof of Heitsch~\cite{Heitsch1986}. It is because we directly computed the map $H^{\bt}(\g,K_{P}) \to H^{\bt}(\g,K_{G})$ in Section~\ref{sec:ex}, while Heitsch applied his residue formulas~\cite[Theorem~4.2]{Heitsch1978},~\cite[Theorem~2.3]{Heitsch1983}. Thus we obtained a slightly simpler proof of Theorem~\ref{thm:BGH} for even $q$. Note that we already gave an alternative proof of Theorem~\ref{thm:BGH} for odd $q$ in Section~\ref{sec:oddtrproj} by using Theorem~\ref{thm:finite}.
\end{rem}

In the case where $(G,G/P)$ is $(\SO_{0}(n+1,1),S^{n}_{\infty})$ for odd $n$, $(\SU(n+1,1),S^{2n+1}_{\infty})$, $(\Sp(n+1,1),S^{4n+3}_{\infty})$ or $\left(F_{4(-20)},S^{15}_{\infty}\right)$, our Bott-Thurston-Heitsch type formulas (Theorem~\ref{thm:BT}) imply that the integration of $\GV(\F_{P})$ along the fibers of the sphere bundle $G/K_{P} \to G/K_{G}$ is nonzero, but it is a constant multiple of the Euler class of the tangent sphere bundle of $G/K_{G}$. So we cannot apply Proposition~\ref{prop:suff} in this case to show the finiteness of the secondary characteristic classes. Nevertheless we get the following. Let $\varphi : \X_{G}(\F)/K_{P} \to \X_{G}(\F)/K_{G}$ be the canonical projection.

\begin{prop}
In the case where $(G,G/P)$ is equal to one of $(\SO_{0}(n+1,1),S^{n}_{\infty})$ for odd $n$, $(\SU(n+1,1),S^{2n+1}_{\infty})$, $(\Sp(n+1,1),S^{4n+3}_{\infty})$ or $\left(F_{4(-20)},S^{15}_{\infty}\right)$, we have $\fint_{\varphi} \GV(p^{*}\E_{\hol (\F)})=0$ in $H^{\bt}(M)$ for any $(G,G/P)$-foliation $\F$ of $M$.
\end{prop}

\begin{proof}
The sphere bundle $\varphi$ has a section because it is homotopic to the Haefliger structure $\X_{G}(\F)/P \to M$, which has a section (see Section~\ref{sec:Haefliger}). Thus its Euler class $e(\varphi)$ is zero. Since $\varphi$ is a sphere bundle with a $(G,G/P)$-foliation transverse to fibers, we get $\fint \GV(p^{*}\E_{\hol (\F)}) = r_{G}\, e(\varphi) = 0$ by the Bott-Thurston-Heitsch type formulas in Theorem~\ref{thm:BT}.
\end{proof}

\begin{rem}\label{rem:alt}
Note that the Godbillon-Vey class is essentially the unique nontrivial secondary class in this case by Proposition~\ref{prop:onlyGV3}. Thus Lemma~\ref{lem:Gysin} gives us another proof of Proposition~\ref{prop:01} for these $(G,G/P)$, and therefore another proof of Theorem~\ref{thm:finite}.
\end{rem}

On the other hand, the situation is different for transversely conformally flat foliations of even codimension. Let $(G,G/P)=(\SO_{0}(n+1,1),S^{n}_{\infty})$ for even $n$. Consider an $S^{n}$-bundle $M \to N$ and a $(G,G/P)$-foliation $\F$ of $M$ transverse to the fibers with a nontrivial volume $\vol(\hol (\F))$. For example, we can take the fiber bundle $\G \backslash G/K_{P} \to \G \backslash G/K_{G}$ foliated by the homogeneous foliation for a torsion-free uniform lattice $\G$ of $G$. Recall that $\varphi$ is the $S^{q}$-bundle $\X_{G}(\F)/K_{P} \to \X_{G}(\F)/K_{G}$ associated to $\F$ with the $(G,G/P)$-foliation $p^{*}\E_{\hol (\F)}$ transverse to the fibers. We get the following.

\begin{prop}
$\fint_{\varphi} \GV(p^{*}\E_{\hol (\F)})$ is nonzero.
\end{prop}

\begin{proof}
The volume of $p^{*}\E_{\hol (\F)}$ is equal to $p_{K_{G}}^{*}\vol(\hol (\F))$, which is nontrivial by assumption. On the other hand, $\fint_{\varphi} \GV(p^{*}\E_{\hol (\F)})$ is a nonzero constant multiple of the volume $p_{K_{G}}^{*}\vol(\hol (\F))$ by Proposition~\ref{prop:r1}.
\end{proof}

\subsection{Finiteness with fixed Euler class}

Consider the case where $G/P=S^{q}$ for even $q$. In this section, we will show Theorem~\ref{thm:even}. In this section, the coefficient ring of cohomology is $\RR$. Since the Euler classes of even dimensional sphere bundles are trivial with real coefficients, the assumption of Theorem~\ref{thm:finite} is never satisfied by Proposition~\ref{prop:cri}. Thus the Gysin sequence of the sphere bundle $\phi^{\CC} : \Gc/K_{P} \to \Gc/K_{G}$ splits to give the exact sequence
\begin{equation}\label{eq:kn}
\xymatrix{0 \ar[r] & H^{\bt}(\Gc/K_{G}) \ar[r]^{(\phi^{\CC})^{*}} & H^{\bt}(\Gc/K_{P}) \ar[r]^<<<<<{\fint_{\phi^{\CC}}} &  H^{\bt-q}(\Gc/K_{G}) \ar[r] & 0\;.}
\end{equation}
Let $\chi (\nu \F_{P})$ be the Euler class of the normal bundle of the $P/K_{P}$-coset foliation $\F_{P}$ on $G/K_{P}$, which is of degree $q$.

\begin{prop}\label{prop:Euler}
$\fint_{\phi^{\CC}} \chi (\nu \F_{P}) = 2$.
\end{prop}

\begin{proof}
Let $\phi_{P} : G/K_{P} \to G/P = S^{q}$ be the canonical projection. Consider the composite
  \[
    \xymatrix{K_{G}/K_{P} \ar[r] & G/K_{P} \ar[r]^<<<<<{\phi_{P}} & G/P}\;.
  \]
Since $\phi^{*}_{P}TS^{q} = \nu \F_{P}$, we get
\[
\int_{K_{G}/K_{P}} \chi (\nu \F) = \int_{S^{q}} \chi (TS^{q}) = 2\;,
\]
which implies the equality of the statement.
\end{proof}

From~\eqref{eq:Weyl1},~\eqref{eq:Weyl2},~\eqref{eq:kn} and Proposition~\ref{prop:Euler}, we get the following.

\begin{prop}\label{prop:Z}
We have
\[
H^{\bt}(\g,K_{P}) \cong H^{\bt}(\g,K_{G}) \otimes \RR [\chi]/(\chi^{2})
\]
as an $H^{\bt}(\g,K_{G})$-module, where $\chi$ is the Euler class of the normal bundle of $\F$.
\end{prop}

Consider the characteristic homomorphism $\Xi_{\hol (\F)} : H^{\bt}(\g,K_{G}) \to H^{\bt}(M)$, which depends only on $\hol (\F) : \pi_{1}M \to G$ (Section~\ref{sec:vol}).

\begin{prop}\label{prop:fixedEulerclass}
Let $\F_{0}$ and $\F_{1}$ be two $(G,G/P)$-foliations of $M$ with the same holonomy homomorphism. If $\chi (\nu \F_{0}) = \chi (\nu \F_{1})$, then $\Delta_{\F_{0}}(\sigma) = \Delta_{\F_{1}}(\sigma)$ for any $\sigma\in H^{\bt}(WO_{q})$.
\end{prop}

\begin{proof}
By Theorem~\ref{thm:factrization}, it is sufficient to prove that $\Delta_{\F_{0}}(\sigma) = \Delta_{\F_{1}}(\sigma)$ for any $\sigma\in H^{\bt}(\g,K_{P})$. For $\sigma\in H^{\bt}(\g,K_{G})$, we get $\Delta_{\F_{0}}(\sigma) =\Delta_{\F_{1}}(\sigma)$ because $\Delta_{\F_{i}}(\sigma)$ is determined only by the holonomy homomorphism according to Proposition~\ref{prop:fibint}. Since $H^{\bt}(\g,K_{P})$ is generated by $\chi$ and $1$ as an $H^{\bt}(\g,K_{G})$-module, we get $\Delta_{\F_{0}}(\sigma) =\Delta_{\F_{1}}(\sigma)$ for any $\sigma\in H^{\bt}(\g,K_{P})$.
\end{proof}

Since $\pi_{0}(\Hom(\pi_{1}M,G))$ is finite (see Remark~\ref{rem:fincomp}), Theorem~\ref{thm:rigid} and Proposition~\ref{prop:fixedEulerclass} imply Theorem~\ref{thm:even}.

\subsection{Infiniteness of classes divisible by the Euler class}

In this section, we will show Theorem~\ref{thm:infinite}, an infiniteness result.

Let $G$ be a connected semisimple Lie group and let $P$ be a parabolic subgroup of $G$ such that $G/P \cong S^{q}$, where $q$ is even. Any $\sigma\in H^{\bt}(WO_{q})$ is said to be {\em divisible} by the Euler class $\chi$ if there exists some $\tau\in H^{\bt}(\g,K_{P})$ such that $\Delta_{\F_{P}}(\sigma) = \tau \cdot \chi$. Note that such $\tau$ belongs to $H^{\bt}(\g,K_{G})$ for any nontrivial divisible class $\sigma$ by Proposition~\ref{prop:Z}. Proposition~\ref{prop:Z} also implies that, if $\sigma\in H^{\bt}(WO_{q})$ is not divisible by the Euler class, then $\fint \sigma = 0$. So, by Proposition~\ref{prop:suff}, we get the finiteness of the possible values of $\Delta_{\F}(\sigma)$. On the other hand, we will show that divisible classes may take infinitely many different values.

\begin{thm}\label{thm:inf}
Let $G$ be a connected semisimple Lie group and let $P$ be a parabolic subgroup of $G$ such that $G/P \cong S^{q}$, where $q$ is even. Let $\sigma \in H^{\bt}(WO_{q})$ be a class divisible by the Euler class. Assume that the restriction map $H^{\bt}(\g) \to H^{\bt}(\k_{G})$ is surjective. Then there exists a connected manifold $X$ with finitely presented fundamental group and an infinite family $\{\F_{m}\}_{m \in \ZZ}$ of $(G,G/P)$-foliations on $X$ such that $\Delta_{\F_{m}}(\sigma) \neq \Delta_{\F_{m'}}(\sigma)$ if $m \neq m'$.
\end{thm}

To prove Theorem~\ref{thm:inf}, we note the following fact.

\begin{lemma}\label{lem:inf}
Let $X \to Y$ be an $S^{q}$-bundle with a section. Then, for any $m\in \ZZ$, there exists a smooth bundle map $f_{m} : X \to X$ whose restriction to each $S^{q}$-fiber is of degree $m$.
\end{lemma}

\begin{proof}
We fix a smooth fiberwise metric on $X \to Y$ so that each $S^{q}$-fiber is the standard round sphere. Let $L$ be the image of a section of $X \to Y$. We can assume that $L$ is a smooth submanifold of $X$.  For $x\in X$, let $F_{x}$ be the $S^{q}$-fiber of $X \to Y$ containing $x$, let $\{x_{0}\}=F_{x} \cap L$, and let $c_{x}$ be a great circle of $F_{x}$ through $x$ and $x_{0}$. Under the identity $c_{x}\equiv\RR/2\pi\ZZ$ with $x_{0}\equiv0$ given by the length parametrization, let $f_{m}(x)=mx$ for $m\in \ZZ$. This defines a smooth map $f_{m} : X \to X$ whose restriction to each fiber is of degree $m$.
\end{proof}

\begin{proof}[Proof of Theorem~\ref{thm:inf}]
Let $\G$ be a torsion-free uniform lattice of $G$. Note that $\G$ is finitely presented because it is the fundamental group of the closed manifold $\G \backslash G/K_{P}$. Since $q$ is even, the Euler class of the $S^{q}$-bundle $\G \backslash G/K_{P} \to \G \backslash G/K_{G}$ is zero. Hence it has a section. Then, by Lemma~\ref{lem:inf}, we take a smooth map $f_{m} : \G \backslash G/K_{P} \to \G \backslash G/K_{P}$ of degree $m$ for any $m\in \mathbb{Z}$. Let $\wt{f}_{m} : G/K_{P} \to G/K_{P}$ be the lift of $f_{m}$ to the universal cover. Define $\Phi_{m} : G \times G/K_{P} \to G/K_{P}$ by $\Phi_{m}(g,x) = g \wt{f}_{m}(x)$. Since $\wt{f}_{m}$ is $\G$-equivariant, we get 
\[
\Phi_{m}(g_{1}g_{2},x) = g_{1}g_{2}\wt{f}_{m}(x) = g_{1} \wt{f}_{m}(g_{2}x) = \Phi_{m}(g_{1},g_{2}x)
\]
for $g_{1}\in G$, $g_{2}\in \G$ and $x\in G/K_{P}$. Then $\Phi_{m}$ induces a smooth map $\Psi_{m} : X \to \G \backslash G/K_{P}$, where $X$ is the quotient of $G \times G/K_{P}$ by the $\G$-action given by $g_{2} \cdot (g_{1},x) = (g_{1}g_{2}^{-1},g_{2}x)$. This $\Psi_{m}$ is a flat principal $G$-bundle over $\G \backslash G/K_{P}$ by construction. Since $\pi_{1}G$ is a finite group, $\pi_{1}X$ is also finitely presented.

Let $\ch_{m} : H^{\bt}(\g) \to H^{\bt}(X)$ be the characteristic homomorphism of $\Psi_{m}$ as a flat principal $G$-bundle over $\G \backslash G/K_{P}$. Let $F$ be a fiber of $\Psi_{m}$, which is homotopy equivalent to $K_G$. By the assumption, the composite of 
\[
\xymatrix{H^{\bt}(\g) \ar[r]^{\ch_{m}} & H^{\bt}(X) \ar[r] & H^{\bt}(F) \cong H^{\bt}(\k_{G})}
\]
is surjective, where the second arrow is the restriction map to $F$. Thus $\Psi_{m}^* : H^{\bt}(\G \backslash G/K_{P}) \to H^{\bt}(X)$ is injective by the Leray-Hirsch theorem.

Consider the $(G,G/P)$-foliation $\F_{m} = \Psi_{m}^{*}\F_{\G}$ on $X$, where $\F_{\G}$ is the foliation of $\G \backslash G/K_{P}$ whose lift to the universal cover $G/K_{P}$ is the $P/K_{P}$-coset foliation $\F_{P}$. By assumption, there exists some $\tau\in H^{\bt}(\g,K_{G})$ such that
\begin{equation}\label{eq:div}
\Delta_{\F_{P}}(\sigma) = \Xi_{\hol (\F_{P})}(\tau) \cdot \chi(\nu \F_{P})\;.
\end{equation}
Since the map $\pi_{1}X \to \pi_{1}(\G \backslash G/K_{P})$ induced by $\Psi_{m}$ is independent of $m$, we get
\begin{equation}\label{eq:psi1}
\Xi_{\hol (\F_{m})}(\tau) = \Xi_{\hol (\F_{1})}(\tau)\;
\end{equation}
for any $m$. On the other hand, since $\chi(\nu \F_{\G})$ is represented by the Poincar\'{e} dual of any $S^{q}$-fiber of $\G \backslash G/K_{P} \to \G \backslash G/K_{G}$, we get
\begin{equation}\label{eq:psi2}
\chi(\nu \F_{m}) = \Psi_{m}^{*}\chi(\nu \F_{\G}) = m\Psi_{1}^{*}\chi(\nu \F_{\G}) = m\chi(\nu \F_{1})
\end{equation}
by construction. By~\eqref{eq:div},~\eqref{eq:psi1} and~\eqref{eq:psi2}, we get $\Delta_{\F_{m}}(\sigma) = m\Delta_{\F_{1}}(\sigma)$. By the injectivity of $\Psi_{1}^{*}$, $\Delta_{\F_{1}}(\sigma)$ is nontrivial of infinite order. Hence we get $\Delta_{\F_{m}}(\sigma) \neq \Delta_{\F_{m'}}(\sigma)$ for $m \neq m'$. 
\end{proof}

Note that the manifolds $X$ are noncompact in our construction. We get Theorem~\ref{thm:infinite} as a corollary of Theorem~\ref{thm:inf} as follows.

\begin{proof}[Proof of Theorem~\ref{thm:infinite}]
By Propositions~\ref{prop:volGV},~\ref{prop:fibint} and~\ref{prop:Z}, there exists a constant $c$ such that $\GV(\F) = c\, \chi(\nu\F) \vol(\hol (\F))$ for transversely conformally flat foliations $\F$ of even codimension. So the Godbillon-Vey class is divisible in this case. Moreover, the surjectivity of the restriction map $H^{\bt}(\so(n+1,1)) \to H^{\bt}(\so(n+1))$ follows from $H^{\bt}(\so(n+1,1)) \otimes \CC \cong H^{\bt}(\so(n+2);\CC)$ and the surjectivity of $H^{\bt}(\so(n+2)) \to H^{\bt}(\so(n+1))$ (see, for example,~\cite[Theorems~VI and~VII in Section~6.23]{GreubHalperinVanstone1976}). Thus the assumption of Theorem~\ref{thm:inf} is satisfied, which implies Theorem~\ref{thm:infinite}.
\end{proof}

\section{Rigidity of foliations on homogeneous spaces}\label{sec:rigidity}

\subsection{Generalization of Bott-Thurston-Heitsch type formulas}\label{sec:genBTH}

Let $(G,G/P)$ be one of $(\SO_{0}(n+1,1),S^{n}_{\infty})$, $(\SU(n+1,1), S^{2n+1}_{\infty})$, $(\Sp(n+1,1),S^{4n+3}_{\infty})$ or $(F_{4(-20)},S^{15}_{\infty})$. Let $q = \dim G/P$. Consider the case of codimension $q > 1$; namely, all cases except $(\SO_{0}(2,1),S^{1}_{\infty})$ and $(\SU(1,1), S^{1}_{\infty})$. Let $M = \G \backslash G / K_{P}$ and $N = \G \backslash G / K_{G}$. Let $\F$ be a $(G,G/P)$-foliation of $\G \backslash G / K_{P}$ whose holonomy homomorphism is $\hol (\F) : \pi_{1}M \to G$. Since $\pi_{1}M \cong \pi_{1}N$, we regard $\hol (\F): \pi_{1}N \to G$. We orient $M$ and $N$ with the orientation of $G/K_{P}$ and the fibers of $\phi_{K_{G}} : G/K_{P} \to G/K_{G}$ in Proposition~\ref{prop:volGV}. The volume $\vol(\hol (\F))$ is defined in $H^{q+1}(N;\RR)$ with the orientation of $G/K_{G}$ as mentioned in Section~\ref{sec:vol}.

\begin{lemma}\label{lem:r2}
If $(G,G/P)$ is $(\SO_{0}(n+1,1),S^{n}_{\infty})$ for odd $n > 1$, $(\SU(n+1,1), S^{2n+1}_{\infty})$ for $n>0$, $(\Sp(n+1,1),S^{4n+3}_{\infty})$ or $(F_{4(-20)},S^{15}_{\infty})$, then
\begin{align}\label{eq:GV12}
\frac{1}{(2\pi)^{q+1}}\int_{M} \GV(\F) & = c_{G}\, \int_{N} \vol(\hol (\F)) \\
& = r_{G}\, \int_{N} e(p_{M})\;, \label{eq:GV123}
\end{align}
where $e(p_{M})$ is the Euler class of $p_{M} : M \to N$, and $r_{G}$ and $c_{G}$ are the functions of $(G,G/P)$ mentioned in Theorem~\ref{thm:BT} and Proposition~\ref{prop:volGV}, respectively. If $(G,G/P)$ is $(\SO_{0}(n+1,1),S^{n}_{\infty})$ for $n$ even, then~\eqref{eq:GV12} is true.
\end{lemma}

\begin{proof}
First, we will prove~\eqref{eq:GV12} for all cases of $(G,G/P)$. The first part of this proof is like the proof of Proposition~\ref{prop:r1}. Take a $\pi_{1}N$-equivariant map $\overline{s} : \wt{N} \to G/K_{G}$ so that $\overline{s}$ is regular at a point $x$. Let $p_{\wt{M}} : \wt{M} \to \wt{N}$ denote the canonical projection $\wt{M} \to \wt{N}$. We get a $\pi_{1}M$-equivariant map $s = \overline{s} \circ p_{\wt{M}} : \wt{M} \to G/K_{G}$. By Lemma~\ref{lem:lift}, we obtain a $\pi_{1}M$-equivariant map $\widehat{\dev} : \wt{M} \to G/K_{P}$ which is regular on $p_{\wt{M}}^{-1}(x)$ and makes the following diagram commutative:
\[
\xymatrix{ \wt{M} \ar[r]^<<<<<{\widehat{\dev}} \ar[d]_{p_{\wt{M}}}    & G/K_{P} \ar[d]^{\phi_{K_{G}}} \\
              \wt{N} \ar[r]_<<<<<{\overline{s}} & G/K_{G}\;, }
\]
where $\phi_{K_{G}} : G/K_{P} \to G/K_{G}$ is the canonical projection. Let $p_{\wt{\mathcal{Z}}} : \wt{\mathcal{Z}} \to N$ be the pull-back of the fiber bundle $\phi_{K_{G}} : G/K_{P} \to G/K_{G}$ by $\overline{s}$. We get the commutative diagram:
\begin{equation}\label{eq:fibs}
\xymatrix{ \wt{M} \ar[r]^{\wt{\psi}} \ar[rd]_{p_{\wt{M}}} &  \wt{\mathcal{Z}} \ar[rr]^<<<<<<<<<<{\xi_{\wt{\mathcal{Z}}}} \ar[d]^{p_{\wt{\mathcal{Z}}}}  && G/K_{P} \ar[d]^{\phi_{K_{G}}} \\
              & \wt{N} \ar[rr]_<<<<<<<<<<{\overline{s}}  && G/K_{G}\;, }
\end{equation}
where $\xi_{\wt{\mathcal{Z}}}$ is the canonical map and $\wt{\psi}$ is the map induced by the universality of the pull-back. By taking the quotient of the left triangle of~\eqref{eq:fibs} by $\Gamma$, we get the following diagram:
\begin{equation}\label{eq:fibs2}
\xymatrix{ M \ar[r]^{\psi} \ar[rd]_{p_{M}} & \mathcal{Z} \ar[d]^{p_{\mathcal{Z}}} \\
              & N\;, }
\end{equation}
where $\mathcal{Z}$ is the quotient of $\wt{\mathcal{Z}}$ by the induced $\G$-action and $\psi$ is the map induced by $\wt{\psi}$.

Let $\F_{\mathcal{Z}}$ be the foliation on $\mathcal{Z}$ whose lift to the universal cover $\wt{\mathcal{Z}}$ is $\xi_{\wt{\mathcal{Z}}}^{*}\F_{P}$. By applying Proposition~\ref{prop:volGV} like in the proof of Proposition~\ref{prop:r1}, we get
\begin{equation}\label{eq:intapp}
\frac{1}{(2\pi)^{q+1}} \fint_{p_{\mathcal{Z}}} \GV(\F_{\mathcal{Z}}) = c_{G}\, \vol(\hol (\F))
\end{equation}
in $H^{q+1}(N;\RR)$. Since $\F = \psi^{*} \F_{\mathcal{Z}}$, we obtain $\GV(\F) = \psi^{*} \GV(\F_{\mathcal{Z}})$. Hence 
\begin{equation}\label{eq:GV2}
\frac{1}{(2\pi)^{q+1}} \int_{M} \GV(\F) = \frac{\deg \psi}{(2\pi)^{q+1}} \int_{\mathcal{Z}} \psi^{*} \GV(\F_{\mathcal{Z}})  \\
= c_{G}\, \deg \psi \int_{N} \vol(\hol (\F))\;,
\end{equation}
where $\deg \psi$ is the degree of $\psi$ as a continuous map. Since $\psi$ is a bundle map that covers the identity map on $N$, we get
\begin{equation}\label{eq:deg1}
\deg \psi = \deg\left(\psi|_{p_{M}^{-1}(x)}\right)\;.
\end{equation}
Here, $\psi|_{p_{M}^{-1}(x)} : p_{M}^{-1}(x) \to p_{\mathcal{Z}}^{-1}(x)$ is a covering map because $\psi$ is regular on $p_{M}^{-1}(x)$. Since $\pi_{1}(p_{M}^{-1}(x)) \cong \pi_{1}(p_{\mathcal{Z}}^{-1}(x)) \cong \pi_{1}(S^{q}) = 1$ because $q > 1$, we obtain
\begin{equation}\label{eq:deg2}
\deg\left(\psi|_{p_{M}^{-1}(x)}\right) = 1\;.
\end{equation}
By~\eqref{eq:GV2},~\eqref{eq:deg1} and~\eqref{eq:deg2}, we get~\eqref{eq:GV12}. 

We get~\eqref{eq:GV123} by using Theorem~\ref{thm:BT} at~\eqref{eq:intapp} instead of Proposition~\ref{prop:volGV}. Note that $e(p_{M}) = e(p_{\mathcal{Z}})$, because $\psi$ is a bundle map of degree one on each fiber.
\end{proof}

We obtain the following direct consequences.

\begin{cor}\label{cor:equality}
\begin{enumerate}

\item\label{i: GV(F) = GV(F_G)} If $(G,G/P)$ is equal to $(\SO_{0}(n+1,1),S^{n}_{\infty})$ for $n$ odd, $(\SU(n+1,1), S^{2n+1}_{\infty})$, $(\Sp(n+1,1),S^{4n+3}_{\infty})$ or $(F_{4(-20)},S^{15}_{\infty})$, then any $(G,G/P)$-foliation $\F$ of $M$ satisfies $\GV(\F) = \GV(\F_{\G})$ and $\hol(\F) = \hol(\F_{\G})$.

\item\label{i: vol(hol(F)) = vol(G)} If $(G,G/P)$ is $(\SO_{0}(n+1,1),S^{n}_{\infty})$ for $n$ even, then $\GV(\F) = \GV(\F_{\G})$ if and only if $\vol(\hol (\F)) = \vol(\G)$, where $\vol(\G)$ is the volume of $\Gamma \hookrightarrow G$ {\rm(}see Example~\ref{ex:volumeG}{\rm)}. 

\end{enumerate}
\end{cor}

Combining Lemma~\ref{lem:r2} with well known properties of the volume, we get the following consequences.

\begin{prop}\label{prop:GV}
If $\GV(\F)$ is nontrivial, then the image of the holonomy homomorphism $\pi_{1}M \to G$ is Zariski dense in $G$.
\end{prop}

\begin{proof}
If $\GV(\F)$ is nontrivial, then $\vol(\hol(\F))$ is also nontrivial by~\eqref{eq:GV12}. Then the image of $\hol(\F)$ is Zariski dense in $G$ by \cite[Proposition~2.1]{Corlette1991}.
\end{proof}

\begin{prop}\label{prop:MW}
If $(G,G/P) = (\SO_{0}(n+1,1),S^{n}_{\infty})$ for even $n$, then $\int_{M} \GV(\F) \leq \int_{M} \GV(\F_{\G})$.
\end{prop}

\begin{proof}
This is a consequence of~\eqref{eq:GV12} and the following generalized version of the Milnor-Wood inequality (see~\cite[Theorem~1.1]{FrancavigliaKlaff2006}): We have
\begin{equation}\label{eq:MW}
\int_{N} \vol(h) \leq \int_{N} \vol(\G)
\end{equation}
for any homomorphism $h : \G \to G$.
\end{proof}

\begin{rem}
The inequality~\eqref{eq:MW} is true also for any other simple Lie group $G$. In fact, it is a consequence of the positivity of the simplicial volume of locally symmetric spaces due to Lafont-Schmidt~\cite{LafontSchmidt2006} (one applies the Hahn-Banach theorem~\cite[Corollary in page~225]{Gromov1982} with~\cite[Corollary~7]{Bucher2008}). But here we need only the case of $(G,G/P) = (\SO_{0}(n+1,1),S^{n}_{\infty})$ for even $n$, where Corollary~\ref{cor:equality}-(i)  does not work.
\end{rem}

\subsection{Rigidity of $(G,G/P)$-foliations of $\G \backslash G/K_{P}$ of higher codimensions}

To prove Theorem~\ref{thm:superrigidity}-\eqref{i: F is smoothly conjugate to F_Gamma}, we will apply the following generalized version of Mostow rigidity.

\begin{thm}[Goldman~\cite{Goldman1988} for the case where $G=\PSO(2,1)$, Dunfield~\cite{Dunfield1999} for $G = \PSO(n+1,1)$, and Corlette~\cite{Corlette1991} for $G = \PSU(n+1,1)$]\label{thm:Mostow}
Let $G$ denote $\PSO(n+1,1)$ or $\PSU(n+1,1)$, and $\G$ a torsion-free uniform lattice of $G$. Any homomorphism $h : \G \to G$ with $\vol(h) = \vol(\G)$ is conjugate to the canonical inclusion $\G \to G$ by an inner automorphism of $G$.
\end{thm}

\begin{rem}
	\begin{enumerate}
	
		\item Francaviglia-Klaff~\cite{FrancavigliaKlaff2006} and Bucher-Burger-Iozzi~\cite{BucherBurgerIozzi2012} generalized the definition of the volume of representations of uniform lattices to non-uniform lattices. (These two definitions do not agree in the non-uniform case.) It allows them to prove Theorem~\ref{thm:Mostow} in a way similar to~\cite{Dunfield1999}, including the case where $\G$ is a non-uniform lattice of $\SO(n+1,1)$. 

		\item Note that the assumption of the above theorem of Goldman is the equality $e(h) = e(\G)$ for the Euler classes. But, because of the proportionality of the Euler class and the volume, it is equivalent to the equality on the volume.

		\item To prove Theorem~\ref{thm:superrigidity} for the case where $G$ is $\Sp(n+1,1)$ or $F_{4(-20)}$, we will apply the superrigidity theorem of Corlette~\cite{Corlette1992}, which asserts that any homomorphism $\G \to G$ from a uniform lattice $\G$ of $G$ is conjugate to the canonical inclusion if its image is Zariski dense. This rigidity is stronger than the case of Theorem~\ref{thm:Mostow}, so we do not need the equality on the volumes.
		
	\end{enumerate}
\end{rem}

\begin{proof}[Proof of Theorem~\ref{thm:superrigidity}-\eqref{i: F is smoothly conjugate to F_Gamma}]
If $(G,G/P)$ is one of $(\SO_{0}(n+1,1),S^{n}_{\infty})$ for $n$ odd or $(\SU(n+1,1), S^{2n+1}_{\infty})$, Corollary~\ref{cor:equality}-\eqref{i: GV(F) = GV(F_G)} implies $\vol(\hol (\F)) = \vol(\G)$. If $(G,G/P)$ is $(\SO_{0}(n+1,1),S^{n}_{\infty})$ for $n$ even, then we get $\vol(\hol (\F)) = \vol(\G)$ by the assumption and Corollary~\ref{cor:equality}-\eqref{i: vol(hol(F)) = vol(G)}. Thus Theorem~\ref{thm:Mostow} implies that $\hol (\F) : \pi_{1}N \to G$ is conjugate to $\pi_{1}N = \G \hookrightarrow G$ by an inner automorphism of $G$. Hence the standard map $\phi_{K_{G}} : G/K_{P} \to G/K_{G}$ is conjugate to a $\pi_{1}M$-equivariant map $s : G/K_{P} \to G/K_{G}$, which is a submersion. Then we get a $\pi_{1}M$-equivariant submersion $\widehat{\dev} : G/K_{P} \to G/K_{P}$ by Lemma~\ref{lem:lift}. It induces a covering map $\overline{\dev} : \G \backslash G/K_{P} \to \G \backslash G/K_{P}$, which must be a diffeomorphism because $\GV(\F) = \GV(\F_{\G})$.
\end{proof}

\begin{proof}[Proof of Theorem~\ref{thm:superrigidity}-\eqref{i: int_M GV(F) leq int_M GV(F_Gamma}]
Corollary~\ref{cor:equality}-\eqref{i: GV(F) = GV(F_G)} and Proposition~\ref{prop:GV} imply that the image of $\hol (\F) : \pi_{1}M \to G$ is Zariski dense in $G$. Thus Corlette's superrigidity theorem~\cite{Corlette1992} for uniform lattices in $\Sp(n+1,1)$ or $F_{4(-20)}$ implies that $\hol (\F) : \pi_{1}N \to G$ is conjugate to $\pi_{1}N = \G \hookrightarrow G$. The rest of the proof is the same as in the case~(i).
\end{proof}

\subsection{Codimension one case}

In the case where $(G,G/P)$ is $(\SO_{0}(2,1), S^{1}_{\infty})$ or $(\SU(1,1), S^{1}_{\infty})$, Lemma~\ref{lem:r2} is not true in general because $\pi_{1}S^{1} \cong \ZZ$. But the theory of codimension one foliations, due to Thurston and Levitt, resolves this problem. Note that, in this case, $K_{G}$ is isomorphic to $\SO(2)$ or $\U(1)$, $P$ is isomorphic to $\Aff_{+}(1;\RR)$ or $\Aff(1;\RR)$, and $K_{P}$ is trivial or $\{\pm 1\}$. Let $\F$ be a $(G,G/P)$-foliation on $M=\G \backslash G/K_{P}$. Here, $N=\G \backslash G /K_{G}$ is a closed Riemann surface and the projection $p : \G \backslash G/K_{P} \to \G \backslash G /K_{G}$ is a principal $S^{1}$-bundle. 

Theorem~\ref{thm:cod1} will be deduced from the following two results:
\begin{thm}[Chihi-ben Ramdane~\cite{ChihiRamdane2008}]\label{thm:CR}
If $\GV(\F)$ is nontrivial, then the image of the holonomy homomorphism of $\F$ is a uniform lattice or a dense subgroup of $G$. In particular, $\F$ is minimal.
\end{thm}

\begin{thm}[Thurston~\cite{Thurston1972a} and Levitt~\cite{Levitt1978}]\label{thm:TL}
A codimension one foliation $\F$ on $M$ without compact leaves is isotopic to a foliation transverse to the fibers of $p$.
\end{thm}

\begin{proof}[Proof of Theorem~\ref{thm:cod1}]
Assume that $\GV(\F)$ is nontrivial. Then $\F$ is minimal by Theorem~\ref{thm:CR}. By Theorem~\ref{thm:TL}, we can isotope $\F$ to a foliation transverse to the fibers of $p$. Since the Euler number of $p$ is equal to the Euler number of $N$ by construction and the Euler class is propotional to the volume, we get $\vol(\hol(\F)) = \vol(\G)$, where $\hol(\F)$ is the holonomy homomorphism of $\F$. According to Theorem~\ref{thm:Mostow}, $\hol(\F)$ is conjugate to $\hol(\F_{\G})$, which is the inclusion map $\G \hookrightarrow G$. Since the conjugation class of suspension foliations are determined by the conjugation class of the holonomy homomorphisms, the proof is concluded.
\end{proof}



\begin{thebibliography}{CGW76}

\bibitem[Asu03]{Asuke2003}
T.~Asuke, \emph{Complexification of foliations and complex secondary classes},
  Bull. Braz. Math. Soc. (N.S.) \textbf{34} (2003), 251--262.

\bibitem[Asu10]{Asuke2010}
\bysame, \emph{{Godbillon-Vey} class of transversely holomorphic foliations},
  MSJ Mem., vol.~24, Math. Soc. Japan, Tokyo, 2010.

\bibitem[AY97]{AbeYokota1997}
K.~Abe and I.~Yokota, \emph{Volumes of compact symmetric spaces}, Tokyo J.
  Math. \textbf{20} (1997), 87--105.

\bibitem[Bak78]{Baker1978}
D.~Baker, \emph{On a class of foliations and the evaluation of their
  characteristic classes}, Comment. Math. Helv. \textbf{53} (1978), 334--363.

\bibitem[BBI12]{BucherBurgerIozzi2012}
M.~Bucher, M.~Burger, and A.~Iozzi, \emph{A dual interpretation of the
  {Gromov-Thurston} proof of {Mostow} rigidity and volume rigidity for
  representations of hyperbolic lattices}, Trends in Harmonic Analysis,
  Springer INdAM Series, Springer, to appear in 2012.

\bibitem[BE85]{BensonEllis1985}
C.~Benson and D.B. Ellis, \emph{Characteristic classes of transversely
  homogeneous foliations}, Trans. Amer. Math. Soc. \textbf{289} (1985),
  849--859.

\bibitem[BG84]{BrooksGoldman1984}
R.~Brooks and W.M. Goldman, \emph{The {Godbillon-Vey} invariant of a
  transversely homogeneous foliation}, Trans. Amer. Math. Soc. \textbf{286}
  (1984), 651--664.

\bibitem[Blu79]{Blumenthal1979}
R.A. Blumenthal, \emph{Transversely homogeneous foliations}, Ann. Inst. Fourier
  (Grenoble) \textbf{29} (1979), 143--158.

\bibitem[Bor50]{Borel1950}
A.~Borel, \emph{Le plan projectif des octaves et les sph\`{e}res comme espaces
  homog\`{e}nes}, C. R. Acad. Sci. Paris \textbf{230} (1950), 1378--1380.

\bibitem[Bor53]{Borel1953}
\bysame, \emph{Sur la cohomologie des espaces fibr\'{e}s principaux et des
  espaces homog\`{e}nes de groupes de {Lie} compacts}, Ann. of Math. {(2)}
  \textbf{57} (1953), 115--207.

\bibitem[Bot72]{Bott1972}
R.~Bott, \emph{Lectures on characteristic classes and foliations. {Notes} by
  {Lawrence Conlon}, with two appendices by {J.~Stasheff}}, Lectures on
  Algebraic and Differential Topology (Second Latin American School in Math.,
  Mexico City, 1971), Lecture Notes in Math., vol. 279, Springer, Berlin, 1972,
  pp.~1--94.

\bibitem[Bot78]{Bott1978}
\bysame, \emph{On some formulas for the characteristic classes of
  group-actions. ({Appendix by Robert Brooks})}, Differential topology,
  foliations and Gelfand-Fuks cohomology (Proc. Sympos., Pontif\'{i}cia Univ.
  Cat\'{o}lica, Rio de Janeiro, 1976), Lecture Notes in Math., vol. 652,
  Springer, Berlin, 1978, pp.~25--61.

\bibitem[BT65]{BorelTits1965}
A.~Borel and J.~Tits, \emph{Groupes r\'{e}ductifs}, Inst. Hautes \'{E}tudes
  Sci. Publ. Math. \textbf{27} (1965), 55--150.

\bibitem[Buc08]{Bucher2008}
M.~Bucher, \emph{The proportionality constant for the simplicial volume of
  locally symmetric spaces}, Colloq. Math. \textbf{111} (2008), 183--198.

\bibitem[CbR08]{ChihiRamdane2008}
S.~Chihi and S.~ben Ramdane, \emph{On the {Godbillon-Vey} invariant and global
  holonomy of {$\mathbb{R}P^1$}-foliations}, Balkan J. Geom. \textbf{13}
  (2008), 24--34.

\bibitem[CC03]{CandelConlon2003}
A.~Candel and L.~Conlon, \emph{Foliations. {II}}, Grad. Stud. Math., vol.~60,
  Amer. Math. Soc., Providence, RI, 2003.

\bibitem[CGW76]{CahnGilkeyWolf1976}
R.S. Cahn, P.B. Gilkey, and J.A. Wolf, \emph{Heat equation, proportionality
  principle, and volume of fundamental domains}, Differential geometry and
  relativity (Reidel, Dordrecht), Mathematical Phys. and Appl. Math., vol.~3,
  1976, pp.~43--54.

\bibitem[Cor91]{Corlette1991}
K.~Corlette, \emph{Rigid representations of {K\"{a}hlerian} fundamental
  groups}, J. Differential Geom. \textbf{33} (1991), 239--252.

\bibitem[Cor92]{Corlette1992}
\bysame, \emph{Archimedean superrigidity and hyperbolic geometry}, Ann. of
  Math. {(2)} \textbf{135} (1992), 165--182.

\bibitem[Dun99]{Dunfield1999}
N.M. Dunfield, \emph{Cyclic surgery, degrees of maps of character curves, and
  volume rigidity for hyperbolic manifolds}, Invent. Math. \textbf{136} (1999),
  623--657.

\bibitem[FK06]{FrancavigliaKlaff2006}
S.~Francaviglia and B.~Klaff, \emph{Maximal volume representations are
  {Fuchsian}}, Geom. Dedicata \textbf{117} (2006), 111--124.

\bibitem[Fre85]{Freudenthal1985}
H.~Freudenthal, \emph{{Oktaven}, {Ausnahmegruppen} und {Oktavengeometrie}},
  Geom. Dedicata \textbf{19} (1985), 7--63.

\bibitem[GHV76]{GreubHalperinVanstone1976}
W.~Greub, S.~Halperin, and R.~Vanstone, \emph{Connections, curvature, and
  cohomology. {Volume} {III}: {Cohomology} of principal bundles and homogeneous
  spaces.}, Pure and Applied Mathematics, vol. 47-{III}, Academic Press
  [Harcourt Brace Jovanovich, Publishers], New York-London, 1976.

\bibitem[Gol88]{Goldman1988}
W.M. Goldman, \emph{Topological components of spaces of representations},
  Invent. Math. \textbf{93} (1988), 557--607.

\bibitem[Gro82]{Gromov1982}
M.~Gromov, \emph{Volume and bounded cohomology}, Inst. Hautes \'{E}tudes Sci.
  Publ. Math. \textbf{56} (1982), 5--99.

\bibitem[GV71]{GodbillonVey1971}
C.~Godbillon and J.~Vey, \emph{Un invariant des feuilletages de codimension
  $1$}, C. R. Acad. Sci. Paris S\'{e}r.~A-B \textbf{273} (1971), A92--A95.

\bibitem[Hae58]{Haefliger1958}
A.~Haefliger, \emph{Structures feuillet\'{e}es et cohomologie \`{a} valeur dans
  un faisceau de groupo\"{i}des}, Comment. Math. Helv. \textbf{32} (1958),
  248--329.

\bibitem[Hae79]{Haefliger1979}
\bysame, \emph{Differential cohomology}, Differential topology (Varenna, 1976)
  (Naples) (Liguori, ed.), 1979, pp.~19--70.

\bibitem[Han88]{Hantout1988}
Y.~Hantout, \emph{Classes caract\'{e}ristiques de {$\Gamma(G,H)$}-structures et
  finitude de leur \'{e}valuation}, Manuscripta Math. \textbf{62} (1988),
  383--399.

\bibitem[Hei73]{Heitsch1973}
J.L. Heitsch, \emph{Deformations of secondary characteristic classes}, Topology
  \textbf{12} (1973), 381--388.

\bibitem[Hei78]{Heitsch1978}
\bysame, \emph{Independent variation of secondary classes}, Ann. of Math. {(2)}
  \textbf{108} (1978), 421--460.

\bibitem[Hei83]{Heitsch1983}
\bysame, \emph{Flat bundles and residues for foliations}, Invent. Math.
  \textbf{73} (1983), 271--285.

\bibitem[Hei86]{Heitsch1986}
\bysame, \emph{Secondary invariants of transversely homogeneous foliations},
  Michigan Math. \textbf{33} (1986), 47--54.

\bibitem[Hir49]{Hirsch1949}
G.~Hirsch, \emph{La g\'{e}om\'{e}trie projective et la topologie des espaces
  fibr\'{e}s}, Topologie alg\'{e}brique (Paris), Colloques Internationaux du
  Centre National de la Recherche Scientifique, vol.~12, Centre de la Recherche
  Scientifique, 1949, pp.~35--42.

\bibitem[HK90]{HurderKatok1990}
S.~Hurder and A.~Katok, \emph{Differentiability, rigidity and {Godbillon-Vey}
  classes for {Anosov} flows}, Inst. Hautes \'{E}tudes Sci. Publ. Math.
  \textbf{72} (1990), 5--61.

\bibitem[Hur02]{Hurder2002}
S.~Hurder, \emph{Dynamics and the {Godbillon-Vey} class: a history and survey},
  Foliations: geometry and dynamics (Warsaw, 2000) (River Edge, NJ), World Sci.
  Publ., 2002, pp.~29--60.

\bibitem[Kna96]{Knapp1996}
A.W. Knapp, \emph{Lie groups beyond an introduction}, Progr. Math., vol. 140,
  Birkh\"{a}user Boston, Boston, MA, 1996.

\bibitem[KO90]{KobayashiOno1990}
T.~Kobayashi and K.~Ono, \emph{Note on {Hirzebruch}'s proportionality
  principle}, J. Fac. Sci. Univ. Tokyo Sect.~IA Math. \textbf{37} (1990),
  71--87.

\bibitem[KT74]{KamberTondeur1974}
F.W. Kamber and P.~Tondeur, \emph{Quelques classes caract\'{e}ristiques
  g\'{e}n\'{e}ralis\'{e}es non triviales de fibr\'{e}s feuillet\'{e}s}, C. R.
  Acad. Sci. Paris S\'{e}r.~A \textbf{279} (1974), 921--924.

\bibitem[KT75a]{KamberTondeur1975b}
\bysame, \emph{Foliated bundles and characteristic classes}, Lecture Notes in
  Math., vol. 493, Springer, New York, 1975.

\bibitem[KT75b]{KamberTondeur1975a}
\bysame, \emph{Non-trivial characteristic invariants of homogeneous foliated
  bundles}, Ann. Sci. \'{E}cole Norm. Sup. \textbf{8} (1975), 433--486.

\bibitem[Lev78]{Levitt1978}
G.~Levitt, \emph{Feuilletages des vari\'{e}t\'{e}s de dimension $3$ qui sont
  des fibres en cercles}, Comment. Math. Helv. \textbf{53} (1978), 572--594.

\bibitem[LS06]{LafontSchmidt2006}
J.-F. Lafont and B.~Schmidt, \emph{Simplicial volume of closed locally
  symmetric spaces of non-compact type}, Acta Math. \textbf{197} (2006),
  129--143.

\bibitem[Mat52]{Matsushima1952}
Y.~Matsushima, \emph{Some remarks on the exceptional simple {Lie} group
  {$F_{4}$}}, Nagoya Math. J. \textbf{4} (1952), 83--88.

\bibitem[Mil64]{Milnor1964}
J.W. Milnor, \emph{On the {Betti} numbers of real varieties}, Proc. Amer. Math.
  Soc. \textbf{15} (1964), 275--280.

\bibitem[Mit85]{Mitsumatsu1985}
Y.~Mitsumatsu, \emph{A relation between the topological invariance of the
  {Godbillon-Vey} invariant and the differentiability of {Anosov} foliations},
  Foliations (Tokyo, 1983) (Amsterdam), Adv. Stud. Pure Math., vol.~5,
  North-Holland, 1985, pp.~159--167.

\bibitem[Mor79]{Morita1979}
S.~Morita, \emph{On characteristic classes of conformal and projective
  foliations}, J. Math. Soc. Japan \textbf{31} (1979), 693--718.

\bibitem[Pat99]{Paternain1999}
G.P. Paternain, \emph{Geodesic flows}, Progr. Math., vol. 180, Birkh\"{a}user
  Boston, 1999.

\bibitem[Pel83]{Pelletier1983}
W.T. Pelletier, \emph{The secondary characteristic classes of solvable
  foliations}, Proc. Amer. Math. Soc. \textbf{88} (1983), 651--659.

\bibitem[Pit79]{Pittie1979}
H.V. Pittie, \emph{The secondary characteristic classes of parabolic
  foliations}, Comment. Math. Helv. \textbf{54} (1979), 601--614.

\bibitem[Qui06]{Quint2006}
J.-F. Quint, \emph{An overview of {Patterson-Sullivan} theory}, Workshop The
  barycenter method, FIM, Zurich, 2006.

\bibitem[Ras80]{Rasmussen1980}
O.H. Rasmussen, \emph{Continuous variation of foliations in codimension two},
  Topology (1980), 335--349.

\bibitem[Rez96]{Reznikov1996}
A.~Reznikov, \emph{Rationality of secondary classes}, J. Differential Geom.
  \textbf{43} (1996), 674--692.

\bibitem[Sul76]{Sullivan1976}
D.~Sullivan, \emph{A generalization of {Milnor}'s inequality concerning affine
  foliations and affine manifolds}, Comment. Math. Helv. \textbf{51} (1976),
  183--189.

\bibitem[Tar04]{Tarquini2004}
C.~Tarquini, \emph{Feuilletages conformes}, Ann. Inst. Fourier (Grenoble)
  \textbf{54} (2004), 453--480.

\bibitem[Thu72a]{Thurston1972a}
W.~Thurston, \emph{Foliations of {$3$}-manifolds which are circle bundles},
  Master's thesis, University of California at Berkeley, Berkley, 1972.

\bibitem[Thu72b]{Thurston1972b}
\bysame, \emph{Non-cobordant foliations on {$S^{3}$}}, Bull. Amer. Math. Soc.
  \textbf{78} (1972), 511--514.

\bibitem[Whi57]{Whitney1957}
H.~Whitney, \emph{Elementary structure of real algebraic varieties}, Ann. of
  Math. {(2)} \textbf{66} (1957), 545--556.

\bibitem[Yam75]{Yamato1975}
K.~Yamato, \emph{Examples of foliations with non trivial exotic characteristic
  classes}, Osaka J. Math. \textbf{12} (1975), 401--417.

\bibitem[Yok55]{Yokota1955}
I.~Yokota, \emph{On the cell structure of the octanion projective plane
  {$\Pi$}}, J. Inst. Polytech. Osaka City Univ. Ser. A. \textbf{6} (1955),
  31--37.

\bibitem[Yok75]{Yokota1975}
\bysame, \emph{On a non compact simple {Lie} group {$F_{4,1}$} of type
  {$F_{4}$}}, J. Fac. Sci. Shinshu Univ. \textbf{10} (1975), 71--80.

\bibitem[Yok90]{Yokota1990}
\bysame, \emph{Realizations of involutive automorphisms {$\sigma$} and
  {$G_{\sigma}$} of exceptional linear {Lie} groups {$G$}. {I}: {$G=G_2$},
  {$F_4$} and {$E\sb 6$}}, Tsukuba J. Math. \textbf{14} (1990), no.~1,
  185--223.

\bibitem[Yok09]{Yokota2009}
\bysame, \emph{Exceptional {Lie} groups}, Preprint,
  http://arxiv.org/abs/0902.0431, 2009.

\bibitem[Zhu12]{Zhukova2012}
N.I. Zhukova, \emph{Global attractors of complete conformal foliations}, Mat.
  Sb. \textbf{203} (2012), no.~3, 79--106, translation in Sb. Math. 203 (2012),
  no. 3-4, 380-405.

\end{thebibliography}

\providecommand{\bysame}{\leavevmode\hbox to3em{\hrulefill}\thinspace}
\providecommand{\MR}{\relax\ifhmode\unskip\space\fi MR }
\providecommand{\MRhref}[2]{%
  \href{http://www.ams.org/mathscinet-getitem?mr=#1}{#2}
}
\providecommand{\href}[2]{#2}

\end{document}